\documentclass[11pt,leqno]{article}
\usepackage{amssymb,amsfonts}
\usepackage{amsmath,amsthm,amsxtra}
\usepackage[all]{xy}
\usepackage{epsfig}
\usepackage{slashed}
\usepackage{color}
\usepackage{verbatim}

\newcommand\bigcheck[1]{#1 \raise1ex\hbox{$\hspace{-1ex}{}^\vee$}}
\newcommand\sucheck[1]{#1 \raise0.5ex\hbox{$\hspace{-1ex}{}^\vee$}}

\newcommand{\ch}{{\rm ch}}

\renewcommand{\Im}{\mathop{\rm Im  \, }}

\renewcommand{\sl}{s\ell}

\newcommand{\sdim}{\mathop{\rm sdim \, }}

\newcommand{\tw}{\rm tw \, }

\newcommand{\bz}{\bar{0}}

\newcommand{\A}{\mathcal{A}}

\newcommand{\CC}{\mathbb{C}}

\newcommand{\QQ}{\mathbb{Q}}
\newcommand{\RR}{\mathbb{R}}
\newcommand{\ZZ}{\mathbb{Z}}

\newcommand{\fg}{\mathfrak{g}}
\newcommand{\fh}{\mathfrak{h}}

\renewcommand{\tilde}{\widetilde}
\renewcommand{\hat}{\widehat}

\makeatletter
\renewcommand\section{\@startsection {section}{1}{\z@}%
                                   {-3.5ex \@plus -1ex \@minus -.2ex}%
                                   {2.3ex \@plus.2ex}%
                                   {\normalfont\large\bfseries}}
\renewcommand\subsection{\@startsection{subsection}{2}{\z@}%
                                     {-3.25ex\@plus -1ex \@minus -.2ex}%
                                     {0ex \@plus .0ex}%
                                     {\normalfont\normalsize\bfseries}}

\@addtoreset{equation}{section}
\makeatother

\newtheorem{theorem}{Theorem}[section]

\newtheorem{lemma}[theorem]{Lemma}
\newtheorem{corollary}[theorem]{Corollary}
\newtheorem{proposition}[theorem]{Proposition}

\newtheorem*{lemma*}{Lemma}

\theoremstyle{remark}
\newtheorem{remark}[theorem]{Remark}

\makeatletter
\def\@maketitle{\newpage
 \null
 \vskip 2em
 \begin{center}%
 \vskip 3em
  {\Large\bf \@title \par}%
  \vskip 1.5em
  {\normalsize
   \lineskip .5em
   \begin{tabular}[t]{c}\@author
   \end{tabular}\par}%
  \vskip 2em

 \end{center}%
 \par
 \vskip 2.5em}
\makeatother

\newcommand{\bs}[1]{{\boldsymbol #1}}

\renewcommand{\epsilon}{\varepsilon}

\definecolor{light}{gray}{.9}

\begin{document}

\title{Representations of affine superalgebras and mock theta functions II}

\author{ Victor G. Kac and Minru Wakimoto}

\author{Victor G. Kac\thanks{Department of Mathematics, M.I.T, Cambridge, MA 02139, USA.  kac@math.mit.edu } \:           and Minoru Wakimoto\thanks{~~wakimoto@r6.dion.ne.jp~~~~ Supported in part by Department of Mathematics, M.I.T.}}

\maketitle

\noindent
To the memory of Andrei Zelevinsky.

\section*{Abstract}
We show that the normalized supercharacters of principal admissible modules, associated to each integrable atypical module over the affine Lie superalgebra $\widehat{\sl}_{2|1}$ can be modified, using Zwegers' real analytic corrections, to form an $SL_2(\ZZ)$-invariant family of functions. Using a variation of Zwegers' correction, we obtain a similar result for $\widehat{osp}_{3|2}$. Applying the quantum Hamiltonian reduction, this leads to new families of positive energy modules over the $N=2$ (resp. $N=3$) superconformal algebras with central charge $c=3 (1-\frac{2m+2}{M})$, where $m \in \ZZ_{\geq 0}, M \in \ZZ_{\geq 2}$, gcd$(2m+2,M)=1$ if $m>0$ (resp. $c=-3\frac{2m+1}{M}$, where $m \in \ZZ_{\geq 0}, M \in \ZZ_{\geq 2}$ gcd$(4m +2, M) =1)$, whose modified supercharacters form an $SL_2(\ZZ)$-invariant family of functions. 

\section{Introduction}

This paper is the second in the series of our papers on modular invariance of 
modified normalized characters of irreducible highest weight representations 
$L(\Lambda)$ over an affine Lie superalgebra $\hat\fg$, associated to a simple finite-dimensional Lie superalgebra $\fg$. We shall keep the notation and conventions of 
the first paper \cite{KW}. 

We assume that $\fg$ is endowed with a non-degenerate invariant supersymmetric bilinear form $(.|.)$ and that its even part $\fg_{\bar{0}}$ is a 
reductive subalgebra. (These properties hold if the Killing form $\kappa$
on $\fg$ is non-degenerate.) 
The associated affine Lie superalgebra is
$\hat{\fg} = \fg [t,t^{-1}] \oplus \CC K \oplus \CC d$, where $K$ is a central element, $\fg [t, t^{-1}] \oplus \CC K$ is a central extension of the loop algebra $\fg [t, t^{-1}]$:
\[ [at^m , bt^n] = [a,b]t^{m+n} + m \delta_{m, -n} (a | b)K,\, a,b \in \fg,\,
 m, n \in \ZZ,  \]
and $ d = t \frac{d}{dt}$.

Choosing a Cartan subalgebra $\fh$ of $\fg_{\bar{0}}$, one defines the Cartan subalgebra of $\hat{\fg}$:

\[\hat{\fh} = \CC d + \fh + \CC K .\]
The restriction of the bilinear form $(.|.)$ to $\fh$ is symmetric 
non-degenerate, and one extends it from $\fh$ to $\hat{\fh}$, letting

\[(\fh | \CC K + \CC d) = 0,\quad  (K | K) = (d | d) = 0,\quad  (K | d) = 1.\] 
One identifies $\hat{\fh}^*$ with $\hat{\fh}$ via this bilinear form. 
Traditionally, the elements of $\hat{\fh}^*$ corresponding to $K$ and $d$ are denoted by $\delta$ and $\Lambda_0$, respectively. One uses the following coordinates on $\hat{\fh}$: 
\begin{equation}
\label{eq:0.1} 
h= 2\pi i(-\tau d + z + tK) = : (\tau, z, t),\,\,\hbox{where}\,\,\tau, 
t \in \CC,\, z \in \fh.
\end{equation}

Given a set of simple roots $\hat{\Pi} =
\{\alpha_0, \alpha_1, \ldots, \alpha_\ell \}$ of $\hat{\fg}$ and 
$\Lambda \in \hat{\fh}^*$, 
one defines the \textit{highest weight module} $L(\Lambda)$ over $\hat{\fg}$ as the irreducible module, which admits a non-zero vector $v_\Lambda$, such that 
\[h v_\Lambda = \Lambda(h)v_\Lambda \,\, \hbox{for}\,\, 
 h \in \hat{\fh}\,\, \hbox{and}\,\, \hat{\fg}_{\alpha_i}v_\Lambda = 0\, 
\,\hbox{for}\,\,\, 
i = 0, \ldots , \ell, \] 
where $\hat{\fg}_{\alpha_i}$ denotes the root subspace of $\hat{\fg}$, 
attached to the simple root $\alpha_i$. 
Since $K$ is a central element of $\hat{\fg}$, it is represented by a scalar 
$\Lambda(K)$, called the \textit{level} of $L(\Lambda)$ (or $\Lambda$). 

The \textit{character} $\mathrm{ch}^+$ and the \textit{supercharacter} $ch^-$ of $L(\Lambda)$ are defined as the following series, corresponding to the weight space decomposition of $L(\Lambda)$ with respect to $\hat{\fh}$, cf. 
(\ref{eq:0.1}):

\[  ch^\pm_{L(\Lambda)} (\tau, z, t) = tr^\pm_{L(\Lambda)}e^{2\pi i (-\tau d + z + tK)},        
\]
where $tr^+$ (resp. $tr^-$) denotes the trace (resp. supertrace). 
It is easy to see (as in \cite{K}, Chapter 10) that these series converge 
absolutely in the 
domain $\{h \in \hat{\fh} |\, \mathrm{Re}\, \alpha_i (h) > 0,\, i = 0, 1, \ldots, \ell \}$ to holomorphic functions. In all examples these functions extend to meromorphic functions in the domain
\begin{equation}
\label{eq:0.2}
X =\left\lbrace h \in \hat{\fh} |\, \mathrm{Re}(K | h) > 0\right\rbrace  = \left\lbrace (\tau, z, t) |\, \mathrm{Im}\tau > 0\right\rbrace. 
\end{equation} 

Note that, as a $\fg[t,t^{-1}] \oplus \CC K$-module, $L(\Lambda)$ remains 
irreducible and it is unchanged if we replace $\Lambda$ by $\Lambda + a\delta,\, 
a \in \CC$, and the character of the $\hat{\fg}$-module gets multiplied by $q^a$. Here and further $ q = e^{2\pi i \tau} = e^{-\delta}$.

In the case when $\fg$ is a simple Lie algebra, there exists an important collection of integrable and, more generally, admissible $\hat{\fg}$-modules $L(\Lambda)$, whose normalized characters have modular invariance property 
\cite{K}, \cite{KW1}. Recall that the \textit{normalized (super)character} 
$\mathrm{ch}_{\Lambda}^\pm$ is defined as 
\[ \mathrm{ch}^\pm_{\Lambda} (\tau, z, t) = q^{m_{\Lambda}} \mathrm{ch}^\pm_{L(\Lambda)} (\tau, z, t), \]
where $m_\Lambda \in \QQ$ is the "modular anomaly" (cf. formula 
(\ref{eq:2.a})). Recall that $\mathrm{ch}^\pm_{\Lambda + aK} = 
\mathrm{ch}^\pm_\Lambda,\, a \in \CC$. 

Recall the action of $SL_2(\RR)$ in the domain $X$ in coordinates (\ref{eq:0.1}):
\begin{equation}
\label{eq:0.3}
\left( 
\begin{array}{cc}
a & b\\ 
c & d\\
\end{array}
\right) \cdot (\tau, z, t) = \left( \frac{a \tau + b}{c \tau + d},\: \frac{z}{c \tau + d},\: t - \frac{c (z|z)}{2(c \tau + d)}\right) .
\end{equation}
By definition, modular invariance of the normalized (super)character of the 
$\hat{\fg}$-module $L(\Lambda)$ means that $L(\Lambda)$ is a member of a finite collection of irreducible highest weight $\hat{\fg}$-modules, such that the $\CC$-span of their (super)characters is $SL_2(\ZZ)$-invariant.

If the Killing form $\kappa$ on $\fg$ is non-degenerate (only such $\fg$ are considered in the present paper), a $\hat{\fg}$-module $L(\Lambda)$ (and the highest weight $\Lambda$) is called (partially) \textit{integrable} if for any root $\alpha$ of $\hat{\fg}$, such that $\kappa (\alpha, \alpha) > 0$, the elements from the root space $\hat{\fg}_\alpha$ act locally nilpotently on $L(\Lambda)$. (If $\fg$ is a Lie algebra, this property is equivalent to integrability, 
as defined in \cite{K}.)

The conjectural Kac--Wakimoto (super)character formula for a tame $\hat{\fg}$-module $L(\Lambda)$ 
(see \cite{KW}, Definition 3.5(c)) reads:
\begin{equation}
\label{eq:0.4}
\hat{R}^\pm \mathrm{ch}^\pm_{L(\Lambda)} = 
\sum_{w \in \hat{W}^\#}^{} \epsilon_\pm (w) w \frac{e^{\Lambda + \hat{\rho}}}{\prod_{\beta \in T_{\bar{\Lambda}}}{(1 \pm e^{-\beta})}}.
\end{equation}
Here and further on $\bar{\Lambda}$ denotes the restriction of $\Lambda$ 
to $\fh$ 
; $\hat{R}^\pm$ is the affine (super)denominator:
\[ \hat{R}^\pm = e^{\hat{\rho}} \frac{\prod_{\alpha \in \hat{\Delta}_{\bar{0}, +}}^{} (1- e^{- \alpha})}{\prod_{\alpha \in \hat{\Delta}_{\bar{1},+}}^{} (1 \pm e^{- \alpha})}, 
\]
$\hat{\rho}$ is the affine Weyl vector defined by 
$2(\hat{\rho}| \alpha_i) = 
(\alpha_i | \alpha_i), i = 0, 1, \ldots , \ell$; 
$\hat{\Delta}_{\bar{0},+}$ and 
$\hat{\Delta}_{\bar{1},+}$ 
are the sets of 
positive even and odd roots of $\hat{\fg}$ (counting multiplicities);
$T_{\bar{\Lambda}} \subset 
\hat{\Pi}$ is a subset of the set of positive roots of $\fg$, 
consisting of pairwise orthogonal isotropic roots, orthogonal to 
$\bar{\Lambda}$, and maximal with this property in this set;
$\hat{W} = W \ltimes t_L$ is the affine Weyl group , where $ W $ is the 
Weyl group of $\fg_{\bar{0}}$, and the subgroup $t_L$ consists of 
\textit{translations} 
$t_\gamma$, $\gamma \in L$, where $L$ is the coroot lattice of 
$\fg_{\bar{0}}$, which are defined by
\begin{equation}
\label{eq:0.5}
t_\gamma (\lambda) = \lambda + \lambda (K) \gamma - ((\lambda|\gamma)+ \frac{1}{2} \lambda(K) (\gamma | \gamma))\delta,\,\,\lambda \in \hat{\fh}^*;
\end{equation} 
$\hat{W}^\# = W^\# \ltimes t_{L^\#}$ (resp. $W^\#$)
is the subgroup of the affine (resp. finite) Weyl group, generated by 
reflections in 
$\alpha \in \hat{\Delta}_+$  (resp. $\alpha \in \Delta_+$) with $\kappa(\alpha,\alpha)>0$,  where $L^\#$ is the sublattice of $L$, 
spanned by the coroots $\alpha$ with $\kappa (\alpha, \alpha) > 0$ of 
$\fg_{\bar{0}}$;
finally, $\epsilon_\pm (w) = (-1)^{s_\pm (w)} $, for a decomposition of $w$ 
in a product of $s_+$ reflections, with respect to non-isotropic even roots, and $s_-$ is the number of those of them, for which the half is not a root 
(note that $\epsilon_+(t_\alpha) = 1$). Formula (\ref{eq:0.4}) 
is a slightly more precise version of formula (3.14) from \cite{KW}.

In all cases studied in the present paper, formula (\ref{eq:0.4}) 
is proven in \cite{GK}. Note that it can be rewritten, after multiplying both 
sides by a suitable power of $q$, as
\begin{equation}
\label{eq:0.6}
q^{\frac{\mathrm{sdim} \fg}{24}} \hat{R}^\pm \mathrm{ch}^\pm_{\Lambda} = 
\sum_{w \in W^\#} \epsilon^\pm (w)\, 
w (\Theta^\pm_{\Lambda + \hat{\rho} , T}),
\end{equation}
where $T = T_{\bar{\Lambda}}$ and $\Theta^\pm_{\Lambda + \hat{\rho}, T}$ is a 
mock theta function of degree $(\Lambda + \hat{\rho})(K)$.
 
Recall \cite{KW} that, for $\lambda \in \hat{\fh}^*$, such that $\lambda(K) > 0$, a \textit{mock theta function} $\Theta^\pm_{\lambda, T}$ of degree $n = \lambda (K)$ is defined by the following series:
 \begin{equation}
\label{eq:0.7}
 \Theta^\pm_{\lambda, T} = q^{\frac{(\lambda | \lambda)}{2n}} \sum_{\gamma \in M}^{} \epsilon_\pm (t_\gamma) t_\gamma \frac{e^\lambda}{\prod_{\beta \in T}^{} (1 \pm e^{- \beta})   }, 
 \end{equation}
 where $M \subset \fh$ is a positive definite integral lattice ($=L^\#$
in (\ref{eq:0.6})), $t_\gamma$ are the translations, defined by (\ref{eq:0.5}) 
, and $T \subset \fh$ is a finite subset, consisting of pairwise orthogonal isotropic vectors, orthogonal to $\lambda$. This series converges to a 
meromorphic function in the domain $X$, which in coordinates (\ref{eq:0.1}) 
takes the form
\begin{equation}
\label{eq:0.8}
\Theta_{\lambda, T} (\tau, z, t) = e^{2 \pi int} \sum_{\gamma \in \frac{\bar{\lambda}}{n}+M}^{} \epsilon^\pm (t_\gamma) \frac{q^{n \frac{(\gamma | \gamma)}{2}}  e^{2 \pi i n \gamma (z)}}{\prod_{\beta \in T}^{}(1 \pm q^{-(\gamma| \beta)} e^{-2\pi i \beta (z)})}.
\end{equation}

Of course, if $T = \emptyset$, we get the usual Jacobi form, which is modular invariant, up to a weight factor. The normalized (super)denominator 
$q^{\frac{\mathrm{sdim} \fg}{24}} \hat{R}^\pm$ is modular invariant, up to the same weight factor, since it can be expressed as a ratio of products of four 
standard Jacobi forms $\vartheta_{a b}\, (a,b = 0 \,\hbox{or}\, 1$) of degree 
2 (see \cite{KW}, Section 4). 

In particular, if $\fg$ is a simple Lie algebra and $L(\Lambda)$ is an integrable $\hat{\fg}$-module, then formula (\ref{eq:0.4}) turns into the usual 
Weyl-Kac character formula, where 
$\hat{W}^\#=\hat{W}$, and 
$T_{\bar{\Lambda}} = \emptyset$ (in this case, of course, 
$\mathrm{ch}^+ = \mathrm{ch}^-$ and $\epsilon_+ (w) = \epsilon_- (w) = \mathrm{det} (w)$). 
Therefore (\ref{eq:0.6}) holds with the usual Jacobi forms, hence 
$\mathrm{ch}_\Lambda$ is modular invariant.  

However, modular invariance fails for mock theta functions, but 
sometimes it can be achieved by adding non-holomorphic real analytic 
corrections, discovered by Zwegers \cite{Z}.

In our previous paper \cite{KW} we studied the first non-trivial case of 
$\fg = \sl_{2|1}$ and the non-typical 
(i.e. $T \neq \emptyset$) integrable 
$\hat{\fg}$-module $L(m \Lambda_0)$, where $m$ is a positive integer. 
Namely, we found a modification of the numerator of 
$ \mathrm{ch}^\pm_{m\Lambda_0} $ 
(i.e. the RHS of (\ref{eq:0.6})) in the spirit of Zwegers, so that the 
corresponding modified normalized supercharacter 
$\mathrm{\tilde{ch}}^-_{m\Lambda_0}$ is modular invariant (see \cite{KW}, 
Theorem 7.3 for 
$M = 1,\, \epsilon = \epsilon' = 0$).

Moreover, since the numerators of the normalized supercharacters of admissible modules and their modifications are expressed by a simple substitution via that for the integrable modules, and their denominators remain the same, we obtain modular invariance for the modified normalized supercharacters of admissible modules, associated to $L(m \Lambda_0)$ (\cite{KW}, Theorem 7.3 for  
$\epsilon = \epsilon' = 0$). 
Finally, it turns out that the modified normalized supercharacters and characters, along with their Ramond twisted analogues, again form a modular invariant family (\cite{KW}, Theorem 7.3).

In the present paper we show that similar results hold for arbitrary atypical integrable highest weight $\hat{\fg}$-modules and the associated principal admissible modules in the case when $\fg$ is a simple Lie superalgebra of rank 
$\ell = 2$, i.e. $\fg$ is either $\sl_{2|1}$ or $\mathrm{osp}_{3|2}$.


For $ \fg = \sl_{2|1} $ an arbitrary atypical integrable weight is of the 
form 
$ \Lambda_{m;s} = (m-s) \Lambda_0 + s \alpha, $ where 
$ m, s \in \ZZ_{\geq 0},\, s \leq m $, and $ \alpha $ is an isotropic root of 
$\fg$. We show that for each $ \Lambda_{m;s} $ the corresponding modified normalized supercharacter $ \mathrm{\tilde{ch}}^-_{\Lambda_{m;s}}  $ is modular invariant, and consequently, for each $ M \in \ZZ_{\geq 1}, $ such that 
$\mathrm{gcd} (M, 2m+2 )= 1$ if $ m > 0 $, the associated family of principal 
admissible modified 
normalized supercharacters is modular invariant (Theorem \ref{th:2.9}). 
It turns out more convenient to use slightly changed Zwegers' real analytic 
functions $ R_{j, m+1}(\tau, v) $; they are given by (\ref{eq:1.4}). 
The method of the present paper is simpler than that of [KW], but still the 
key result is that $ \Phi^{[m;s]} - \Phi^{\left[ m;s\right) }|_S $ is a 
holomorphic function, where  $ \Phi^{[m;s]} $ is the numerator of $  \mathrm{ch}^-_{\Lambda_{m;s}} $. It is in the proof of the latter fact
that the restriction $ \ell = 2 $ is essential.

For $ \fg = \mathrm{osp}_{3|2} 
(= B(1,1)) $ the only atypical integrable 
weight is $ m \Lambda_0, $ where $ m \in \ZZ_{\geq 0} $. The method is 
similar to that for $ \fg = \sl_{2|1} $. The corresponding Zwegers' type real 
analytic functions $ R^{[B]}_{j + \frac{1}{2}, m + \frac{1}{2}} (\tau, \upsilon)$ are introduced in Section \ref{sec:4}. As a result, we construct a modular 
invariant modified normalized supercharacter for the 
$\hat{\mathrm{osp}}_{3|2}$-module
$L(m\Lambda_0)$, and, for 
each $ M \in \ZZ_{\geq 1} $, such that $ \mathrm{gcd}(M, 4m +2) = 1 $, the associated modular invariant family of principal admissible modified normalized supercharacters (Theorem \ref{th:5.12}).

As in \cite{KW}, in all cases considered, modified normalized
supercharacters and characters, along with their twisted analogues, again 
form a modular invariant 
family (Theorems \ref{th:2.9} and \ref{th:5.12}).

Next, as in \cite{KW}, we apply the quantum Hamiltonian reduction to the 
principal admissible $ \hat{\fg}$-modules. As usual, the integrable and a few 
admissible $ \hat{\fg}$-modules get erased (i.e. give zero), but 
what remains of 
the modular invariant families of modified normalized characters of
$ \hat{\fg}$-modules, produce modular 
invariant families of modified characters and supercharacters of 
Neveu-Schwarz and 
Ramond $ N =2  $ (resp. $ N =3 $) superconformal algebra positive energy 
modules. Namely, in the case $ \fg = \sl_{2|1} $ we obtain new families of 
$  N = 2 $ superconformal algebra modules with central charge 
$  c = 3(1 - \frac{2m+2}{M}) $, where $  m \in \ZZ_{\geq 1}, M \in \ZZ_{\geq 2} $ and $ \mathrm{gcd}(M, 2m+2) = 0 $  (Theorem \ref{th:2.11}); 
and in the case of 
$\fg =  \mathrm{osp}_{3|2} $ we obtain new families of $ N = 3 $ 
superconformal algebra modules with central charge
$  c = -3 \frac{2m+1}{M} $, 
where $  m \in \ZZ_{\geq 0}, M \in \ZZ_{\geq 2} $ and  $ \mathrm{gcd}(M, 4m+2) = 1 $  (Theorem \ref{th:6.3}).

In our subsequent paper, we will consider the case $ \fg = D(2, 1;a) $ and the corresponding big $ N =4 $ superconformal algebras, obtained from
$\hat{\fg}$ by the quantum Hamiltonian reduction \cite{KW2}, \cite{KW3} .

\section{Transformation properties of the mock theta functions 
$\Phi^{[m;\,s]}$ and their modifications $\tilde{\Phi}^{[m;\,s]}$.}
\label{sec:1}

In this section we study modular and elliptic transformation properties of supercharacters
$ch^{-}_{L(\Lambda)}$ of integrable highest weight modules $L(\Lambda)$ over the affine Lie superalgebra $\widehat{\sl}_{2|1}$. We choose the set of simple roots $\{\alpha_0, \alpha_1, \alpha_2\}$, where $\alpha_0$ is even and $\alpha_1, \alpha_2$ are odd, and the scalar products are :
$$
(\alpha_0|\alpha_0)=2,\, (\alpha_1|\alpha_1)=(\alpha_2|\alpha_2)=0,\, (\alpha_0|\alpha_1) = (\alpha_0|\alpha_2)=-1, \,(\alpha_1|\alpha_2)=1.
$$
Then the highest weight $\Lambda$ of an integrable module $L(\Lambda)$, 
such that $(\Lambda|\alpha_1)=0$, is of the form (up to adding a multiple of the
imaginary root $\delta$):
$$
\Lambda_{m;s}=(m-s)\Lambda_0+s \Lambda_2=m\Lambda_0+s \alpha_1, \;\text{where}\;  m,s \in \ZZ_{\geq 0},\,\, 0\leq s \leq m.
$$
It is easy to see that the highest weight $\Lambda$ of any atypical integrable level $m$ $\widehat{\sl}_{2|1}$-module can be brought to this form by odd reflections (up to adding a multiple of an imaginary root).

In our paper [KW] we studied transformation properties of the supercharacter of only the vacuum $\widehat{\sl}_{2|1}$-module $L(m \Lambda_0)$, by a more complicated method, close to the original approach of [$Z$].

Since the transformation properties of the superdenominator $\widehat{R}^-$ are well understood [KW], it suffices to study those of the numerator $\widehat{R}^- ch^{-}_{L(\Lambda)}$.

The numerator of the supercharacter of $L(\Lambda_{m;s})$ is given by the 
formula (see Conjecture 3.8 in \cite{KW} and its proof in \cite{S} or in 
\cite{GK}):
$$
\hat{R}^{-} ch^{-}_{L(\Lambda_{m;s})}
=e^{(m+1) \Lambda_0} 
\sum\limits_{j \in \ZZ}\left(  \dfrac{e^{-j(m+1)(\alpha_1+\alpha_2)} 
e^{s \alpha_1} q^{j^2(m+1)-j s}}{1 - e^{-\alpha_1}q^{j}} - 
\dfrac{e^{j(m+1)(\alpha_1+\alpha_2)} e^{-s\alpha_2} 
q^{j^2(m+1)-js}}{1 - e^{\alpha_2} q^j} \right),
$$
where, as usual, $q=e^{-\delta}$. Hence in coordinates 
$$
h=2\pi i\,(-\tau \Lambda_0 - z_1 \alpha_2- z_2\alpha_1+t\delta)\,\,\, 
(\hbox {i.e.}\,\, e^{-\alpha_1(h)} = e^{2\pi i z_1},\,\, e^{-\alpha_2(h)}=e^{2\pi i z_2},\,
\, e^{-\delta(h)}=e^{2\pi i \tau}=q)
$$
%
it is given by the 
function
$$
\left(\hat{R}^-\ch^-_{L(\Lambda_{m;\,s})} \right) (h) 
= \Phi^{[m;\,s]}(\tau, z_1, z_2, t)
= e^{2\pi i(m+1)t} \left(\Phi_1^{[m;\, s]}(\tau, z_1, z_2) - 
\Phi_1^{[m;\, s]} (\tau,- z_2,- z_1)\right),
$$
where
\begin{equation}
	\label{eq:1.1}
\Phi_1^{[m;\, s]}(\tau, z_1, z_2) 
= \sum\limits_{j \in \ZZ} \dfrac{e^{2\pi i(m+1)j(z_1+z_2)}
 e^{-2\pi i s z_1} q^{j^2(m+1)-js}}{1-e^{2\pi i z_1} q^{j}}\;\;.
\end{equation}
It is a holomorphic function in the domain
$X=\{(\tau, z_1,z_2,t)\in \CC^4|\, \Im \tau >0\}.$

Recall the formula for the classical theta functions 
$\Theta_{j, m} (\tau, z, t)$ 
of degree $m$, where $m$ is a positive integer and $j \in \ZZ$ mod $2m \ZZ$ :

\begin{equation} \label{eq:1.2} \Theta_{j, m} (\tau, z, t) 
= e^{2\pi i m t} 		\sum\limits_{n \in \ZZ}{e^{2\pi i m z (n+\frac{j}{2m})} q^{m(n+\frac{j}{2m})^2}}.
\end{equation}
These are holomorphic functions in the domain 
$X_0=\{(\tau, z, t) \in \CC^3 |\, \Im \tau>0\}$.
Let $\Theta_{j, m} (\tau,z) = \Theta_{j, m} (\tau, z, 0)$.
In the study of the functions $\Phi^{[m;s]}$ we shall use, in particular, 
the following obvious properties
\begin{equation}\label{eq:1.a} 
\Theta_{j, m} (\tau, -z)= \Theta_{-j, m} (\tau, z), 
\end{equation}
\begin{equation}
\label{eq:1.b} 
 \Theta_{j, m} (\tau, z+b)=(-1)^{bj} \Theta_{j, m} (\tau, z)\,\,\hbox{if}\,\,\,\, b\in \ZZ. 
\end{equation}

\vspace{-1ex}
\noindent
\begin{lemma}
\label{lem:1.1} The functions $\Phi^{[m;\,s]}$ satisfy the following 
properties :
$$$$
\vspace{-10ex}
 \begin{list}{}{}
	\item(0) $\Phi^{[m;\,s]}- \Phi^{[m;\,s]}|_S$ is a holomorphic function
in the domain $X$.
	 
	\item(1) $\Phi^{[m;\,s]}(\tau+1,z_1,z_2,t)=\Phi^{[m;\,s]}(\tau,z_1,z_2,t)$.
	
	\item(2) $\Phi^{[m;\,s]}(\tau,z_1+a,z_2+b,t)=\Phi^{[m;\,s]}(\tau,z_1,z_2,t)$ if $a,b \in \ZZ$.
	
	\item(3) $\Phi^{[m;\,s]}(\tau,z_1,z_2,t)- e^{2\pi i(m+1)z_1} \Phi^{[m;\,s]}(\tau, z_1, z_2+\tau, t)$ 
	
	$= e^{2\pi i(m+1)t} \sum\limits_{k=0}^{m} e^{\pi i (k- s)(z_1-z_2)} q^{-\frac{(k-s)^2}{4(m+1)}} \left(\Theta_{k-s,\, m+1} -\Theta_{-(k-s),\,m+1} \right) (\tau, z_1+z_2)$.
	
	\item(4) $\Phi^{[m;\,s]} (\tau,z_1,z_2,t) - e^{-2\pi i(m+1)z_2} \Phi^{[m;\,s]}(\tau,z_1-\tau,z_2,t)$
	
	$= e^{2\pi i(m+1)t}\sum\limits_{k=0}^{m} e^{\pi i(k-s)(z_1-z_2)} q^{-\frac{(k-s)^2}{4(m+1)}} \left( \Theta_{k-s,\,m+1} - \Theta_{-(k-s),\,m+1} \right) (\tau, z_1+z_2)$.
	
	\item(5) $\Phi^{[m;\,s]}(\tau, z_1+j\tau,z_2+j\tau,t)= q^{-j^2(m+1)} e^{-2\pi ij(m+1)(z_1+z_2)} \Phi^{[m;\,s]}(\tau, z_1,z_2,t)$ if $j\in \ZZ$.

	\item(6) $\Phi^{[m;\,s]}(\tau,-z_2,-z_1,t)=-\Phi^{[m;\,s]}(\tau,z_1,z_2,t)$.
\end{list}
\end{lemma}

\begin{proof}
Recall that, by definition of the action of $SL(2,\ZZ)$ (see \cite{KW},
Section 4):
\begin{equation}	
\label{eq:1.3}		
\Phi^{[m;s]}_1|_S (\tau, z_1, z_2) = \frac{1}{\tau} e^{-\frac{2 \pi i (m+1)z_1 z_2}{\tau}} \Phi^{[m;s]}_1 \left(-\frac{1}{\tau}, \frac{z_1}{\tau}, \frac{z_2}{\tau} \right).
\end{equation}
In order to prove $(0)$, it suffices to show that 
$$
\Phi^{[m;s]}_1 (\tau, z_1, z_2) - \Phi^{[m;s]}_1 |_S (\tau, z_1, z_2),
$$
viewed as a function in $z_1$, has zero residues at all poles. 
The poles of both $\Phi^{[m;s]}_1$ and $\Phi^{[m;s]}_1|_S$ are the points 
$z_1 \in \ZZ + \tau \ZZ$, and it is easy to see that :
$$
\hbox{Res}_{\,z_1\,=\, n+j \tau}\;\; \Phi^{[m;s]}_1 = \hbox{Res}_{\;z_1\,=\,n+j \tau} \;\;\Phi^{[m;s]}_1|_S = -\frac{1}{2\pi i} e^{-2 \pi i j (m+1)z_2},
$$
as required.

Claims (1), (2) and (6) are obvious. In order to prove (3), we rewrite :
$$
\begin{array}{ll}
			&	\Phi^{[m;s]}_1 (\tau, z_1, z_2) \,-\, e^{2\pi i (m+1) z_1}\; \Phi^{[m;s]}_1 (\tau, z_1, z_2+\tau) \\[1.5ex]
&=\sum\limits_{j\in \ZZ}{e^{2\pi ij (m+1) (z_1+z_2)}\; e^{-2\pi i s z_1}\; 
q^{j^2 (m+1)-j s}}\;\,\frac{1-\left(e^{2\pi i z_1 q^j}\right)^{m+1}}{1- e^{2\pi i z_1} q^j} \\[1.5ex]
&=e^{-2\pi i s z_1} \sum\limits_{k=0}^m{e^{\pi i k (z_1-z_2)}}  \;\,
\sum\limits_{j \in \ZZ} {\left(e^{2\pi i j (m+1)(z_1+z_2)} e^{\pi i k (z_1+z_2)}\right) \,\;q^{j^2(m+1)+j(k-s)}} \\[2ex]
&=\sum\limits_{k=0}^m e^{\pi i (k-s)(z_1-z_2)}\, q^{-\frac{(k-s)^2}{4(m+1)}}
\, \Theta_{k-s, m+1} (\tau, z_1 + z_2).
\end{array}
$$
Using (\ref{eq:1.a}), 
we obtain the claim.

Claim (4) follows easily from (6) by sending $(z_1,z_2)$ to
$(-z_2,-z_1)$ in (3) and using (\ref{eq:1.a}).

In order to prove claim (5), note that 
$$
\Phi^{[m;s]}_1 (\tau, z_1+\tau, z_2+\tau)= q^{-(m+1)} e^{-2\pi i (m+1) (z_1+z_2)} \Phi^{[m;s]}_1 (\tau, z_1, z_2),
$$
via replacing $j$ by $j-1$ in the sum, defining the LHS (see (\ref{eq:1.1})). The same formula holds for $\Phi^{[m;s]}_1 (\tau, -z_2-\tau, -z_1-\tau)$ via replacing $j$ by $j+1$ in the sum. Hence (5) holds for $j=1$, and by induction on $j$ it holds for arbitrary  $j \geq 1$.
\end{proof}

We change the coordinates, letting
$$
z_1=v-u, \quad z_2=-v-u,\quad\text{i.e.}\quad u=-\dfrac{z_1+z_2}{2},\quad v=\dfrac{z_1-z_2}{2},
$$
and denote
$$
\varphi^{[m;s]}(\tau, u, v, t)= \Phi^{[m;s]} (\tau,z_1,z_2,t).
$$
Formula (\ref{eq:1.3}) becomes :
$$
\varphi^{[m;s]}|_S \;\,(\tau, u, v, t)=\frac{1}{\tau} \varphi^{[m,s]} 
\left(-\frac{1}{\tau}, \frac{u}{\tau}, \frac{v}{\tau}, t-\frac{u^2-v^2}{\tau}\right).
$$
The following lemma is immediate by Lemma \ref{lem:1.1}.

\vspace{1ex}
\noindent
\begin{lemma}
  \label{lem:1.2} The function $\varphi^{[m;\,s]}$ has the following properties:

  \begin{list}{}{}
  \item (0) $\varphi^{[m;s]}- \varphi^{[m;s]}|_S$ is a holomorphic 
function in the domain $X$. 
	
	\item(1) $\varphi^{[m;s]}(\tau+1,u,v,t)=\varphi^{[m;s]}(\tau,u,v,t)$.
	
	\item(2) $\varphi^{[m;s]}(\tau,u+a,v+b,t)=\varphi^{[m;s]}(\tau,u,v,t)$ if $a,b \in \frac{1}{2} \ZZ$ are such that $a+b \in \ZZ$.
	
	\item(3) $\varphi^{[m;s]}(\tau,u,v,t)- e^{2\pi i(m+1)(v-u)} \varphi^{[m;s]}(\tau, u-\frac{\tau}{2}, v-\frac{\tau}{2},t)$ 
	
	$= -e^{2\pi i(m+1)t} \sum\limits_{j=0}^{m} e^{2\pi i (j- s)v} q^{-\frac{(j-s)^2}{4(m+1)}} \left(\Theta_{j-s,\, m+1} -\Theta_{-(j-s),\, m+1} \right) (\tau, 2u)$.
	
	\item(4) $\varphi^{[m;s]} (\tau,u,v,t) - e^{2\pi i(m+1)(u+v)}\; \varphi^{[m;s]}(\tau,u+\frac{\tau}{2},v-\frac{\tau}{2},t)$ 
	
	$= -e^{2\pi i(m+1)t}\sum\limits_{j=0}^{m} e^{2\pi i(j-s)v}\; q^{-\frac{(j-s)^2}{4(m+1)}} \left( \Theta_{j-s,\, m+1} - \Theta_{-(j-s),\, m+1} \right) (\tau, 2u)$.
	
	\item(5) $\varphi^{[m;s]}(\tau, u-\tau,v,t) =  e^{4\pi i(m+1)u} q^{-(m+1)} \varphi^{[m;s]}(\tau, u,v,t)$
	
	\item(6) $\varphi^{[m;s]}(\tau,-u,v,t)=-\varphi^{[m;s]}(\tau,u,v,t)$.
 \end{list}
 \end{lemma}
>From claims (3) and (4) of Lemma \ref{lem:1.2}, we obtain the following:

\begin{lemma}
 \label{lem:1.3}
$$ 
\begin{array}{ll}
\varphi^{[m;\,s]}(\tau,u,v,t) -e^{2\pi i(m+1)(2v-\tau)} \varphi^{[m;\,s]}(\tau,u,v-\tau,t) \\
= -e^{2\pi i(m+1)t} \sum\limits_{j=0}^{2m+1} e^{2\pi i(j-s)v} q^{-\frac{(j-s)^2}{4(m+1)}} \left(\Theta_{j-s,\,m+1} - \Theta_{-(j-s),\,m+1} \right) (\tau, 2u).
\end{array}
$$
\hfill $\Box$
\end{lemma}

We put
$$
G^{[m;\,s]}(\tau,u,v,t) = \varphi^{[m;\,s]}(\tau,u,v,t) - \varphi^{[m;\,s]}|_S(\tau,u,v,t).
$$


\begin{lemma}
  \label{lem:1.4} The function $G^{[m;\,s]}$ satisfies the following properties:
 
   \begin{list}{}{}
  
  \item(1) $G^{[m;\,s]}(\tau, u, v, t)$ is holomorphic with respect to $v$.
	
	\item(2) $G^{[m;\,s]}(\tau,u,v,t)- e^{2\pi i(m+1)(2v-\tau)} G^{[m;\,s]}(\tau,u,v-\tau,t) $
	
	$= -e^{2\pi i(m+1)t} \sum\limits_{j=-s}^{-s+2m+1} e^{2\pi ijv} q^{-\frac{j^2}{4(m+1)}} \left( \Theta_{j,\, m+1} -\Theta_{-j,\, m+1} \right) (\tau, 2u)$.
	
	\item(3) $G^{[m;\,s]}(\tau,u,v+1,t)-G^{[m;\,s]}(\tau,u,v,t)=i e^{2\pi i(m+1)t} \dfrac{1}{\sqrt{2(m+1)}}(-i\tau)^{-\frac{1}{2}}$
	$$
	\times \qquad \sum\limits_{j, k=-s}^{-s+2m+1}
	{\hspace{-2.5ex} e^{-\frac{\pi ijk}{m+1}} 	\;\,
	e^{\frac{2\pi i (m+1)}{\tau} \left( v+ \frac{k}{2(m+1)}\right)^2}} \left(\Theta_{j,\, m+1}-\Theta_{-j,\, m+1} \right) (\tau,2u).
	$$
	
	\item(4) $G^{[m;\,s]}$ is determined uniquely by the above three properties.
\end{list}
\end{lemma}

\begin{proof}
Without loss of generality we can let $t=0$. We have, using Lemma \ref{lem:1.2} (2) for $a=0,\,b =1$ :
$$
\begin{array}{ll}
G^{[m;\,s]} (\tau,u,v,0) - e^{2\pi i (m+1)(2v-\tau)}\; G^{[m;\,s]} (\tau,u,v-\tau,0)\\[1ex]
= \varphi^{[m;\,s]} (\tau,u,v,0) - \tau^{-1}\; e^{-\frac{2\pi i (m+1)}{\tau} (u^2-v^2)}\; \varphi^{[m;\,s]} \left(-\frac{1}{\tau}, \frac{u}{\tau}, \frac{v}{\tau}, 0\right) \\[1ex]
-e^{2\pi i (m+1) (2v-\tau)} \left( \varphi^{[m;\,s]} (\tau,u,v-\tau,0) -\tau^{-1}\; e^{-\frac{2\pi i (m+1)}{\tau} [u^2-(v-\tau)^2]}  \varphi^{[m;\,s]}\left(-\frac{1}{\tau}, \frac{u}{\tau}, \frac{v-\tau}{\tau}, 0 \right)\right) \\[1ex]
=\varphi^{[m;\,s]}(\tau,u,v,0) - e^{2\pi i (m+1) (2v-\tau)} \varphi^{[m;\,s]} (\tau,u,v-\tau,0) \\[1ex]
= -\sum\limits^{2m+1}_{j=0}{e^{2\pi i (j-s)v}\; q^{-\frac{(j-s)^2}{4(m+1)}}}  
\left(\Theta_{j-s,\, m+1}-\Theta_{-(j-s),\, m+1}\right) (\tau,2u).
\end{array} 
$$
The last equality holds by Lemma \ref{lem:1.3}. This proves (2).

In order to prove (3), note that after the substitution $(\tau,u,v)\rightarrow \left(-\frac{1}{\tau}, \frac{u}{\tau}, \frac{v}{\tau}\right)$ in Lemma \ref{lem:1.3} and using the modular transformation formula of $\Theta_{\pm(j-s),\,m+1}$
(see e.g. the Appendix in \cite{KW}), we obtain :
$$
\begin{array}{l}
\varphi^{[m;\,s]}\left(-\frac{1}{\tau}, \frac{u}{\tau}, \frac{v}{\tau}, 0 \right)-
e^{\frac{2\pi i (m+1)}{\tau} (2v+1)} \varphi^{[m;\,s]}\left(-\frac{1}{\tau}, \frac{u}{\tau}, \frac{v+1}{\tau}, 0 \right)\\[1ex]
=-\sum\limits^{2m+1}_{j=0}{e^{\frac{2\pi i (m+1)}{\tau}(j-s) v}}  
e^{\frac{2\pi i (j-s)^2}{4(m+1)\tau}}
\left(\Theta_{j-s,\,m+1}-\Theta_{-(j-s),\,m+1}\right) \left(-\frac{1}{\tau}, \frac{2u}{\tau}\right)\\[1ex]
=\frac{-(-i\tau)^{\frac{1}{2}}}{\sqrt{2(m+1)}} 
e^{\frac{2\pi i (m+1)}{\tau} (u^2-v^2)} 
\displaystyle\sum\limits_{0 \leq j \leq 2m+1 \atop k \in \ZZ/2(m+1)\ZZ} {e^{-\frac{\pi i (j-s)k}{m+1}} e^{\frac{2\pi i (m+1)}{\tau} \left(v+\frac{j-s}{2(m+1)}\right)^2}} \left(\Theta_{k,\,m+1}- \Theta_{-k,\,m+1}\right) (\tau,2u).
\end{array}
$$
Using this and Lemma \ref{lem:1.2} (2), we obtain :
$$
\begin{array}{l}
G^{[m;\,s]} (\tau,u,v+1,0) - G^{[m;\,s]} (\tau,u,v,0) \\[1ex]
= \varphi^{[m;\,s]} (\tau,u,v+1,0) - \tau^{-1}\; e^{-\frac{2\pi i (m+1)}{\tau} (u^2-(v+1)^2)}\; \varphi^{[m;\,s]} \left(-\frac{1}{\tau}, \frac{u}{\tau}, \frac{v+1}{\tau}, 0\right)	\\[1ex]
-\left(\varphi^{[m;\,s]}(\tau,u,v,0)-\tau^{-1}\;
e^{-\frac{2\pi i(m+1)}{\tau}(u^2-v^2)}\;\varphi^{[m;\,s]}
\left(-\frac{1}{\tau}, \frac{u}{\tau}, \frac{v}{\tau}, 0 \right) \right) \\[1ex]
=\tau^{-1} e^{-\frac{2\pi i (m+1)}{\tau} (u^2-v^2)}  \left(\varphi^{[m;\,s]} \left(-\frac{1}{\tau}, \frac{u}{\tau}, \frac{v}{\tau}, 0 \right) - 
e^{\frac{2\pi i (m+1)}{\tau} (2v+1)} \varphi^{[m;\,s]}\left(-\frac{1}{\tau}, \frac{u}{\tau}, \frac{v+1}{\tau}, 0 \right) \right)\\[1ex]
=\frac{-i}{-i \tau}\cdot \frac{-(-i \tau)^{\frac{1}{2}}}{\sqrt{2(m+1)}}
\displaystyle\sum\limits_{0 \leq j \leq 2m+1 \atop k \in \ZZ/2(m+1)\ZZ} 
{e^{-\frac{\pi i (j-s)k}{m+1}} e^{\frac{2\pi i (m+1)}{\tau} \left(v+\frac{j-s}{2(m+1)}\right)^2}} \left(\Theta_{k,\,m+1}-\Theta_{-k,\,m+1}\right) (\tau,2u),
\end{array}
$$
which proves (3).

The proof of claim (4) is the same as that Proposition 5.4(c) in [KW]. 
Namely, if $F(\tau,u,v,t)$ is the difference of two functions, satisfying (1)-(3), then the function 
$$
P(\tau,u,v,t) = F(\tau,u,v,t)\vartheta_{11} ((m+1)\tau, (m+1)v)^2
$$ 
is holomorphic and doubly periodic in $v$, and vanishes at $v=0$, hence is zero.
\end{proof} 
\begin{lemma}
  \label{lem:1.5}

\quad Let $a_j(\tau,v)\,\,(j\in \ZZ,\,-s \leqq j \leq -s+2m+1)$ be functions 
satisfying the following conditions (i), (ii), (iii):
  \begin{list}{}{}
	\item(i) $a_j (\tau, v)$ is holomorphic with respect to $v$,
	
	\item(ii) $a_j (\tau, v)-e^{2\pi i(m+1)(2v-\tau)} a_j (\tau, v-\tau)= 
	-e^{2\pi i jv}\;\,  e^{-\frac{\pi i \tau}{2(m+1)}j^2}$,
	
	\item(iii) $a_j(\tau,v+1)-a_j(\tau,v)$
	
$=\dfrac{i}{\sqrt{2(m+1)}}(-i\tau)^{-\frac{1}{2}} 
\displaystyle\sum\limits^{-s+2m+1}_{k =-s} e^{-\frac{\pi ijk}{m+1}}\;\, 
e^{\frac{2\pi i(m+1)}{\tau} \left(v+\frac{k}{2(m+1)}\right)^2}$.
\end{list}
(By the argument in the proof of Lemma \ref{lem:1.4}(4), $a_j(\tau,v)$
is uniquely determined by the properties (i),(ii),(iii).)
\noindent Then the function
$$
G^{[m;\,s]} (\tau, u, v, t) := e^{2\pi i(m+1) t} 
\sum\limits^{-s+2m+1}_{j=-s}{\hspace{-2.5ex}a_j (\tau, v) \left(\Theta_{j,m+1} - \Theta_{-j,m+1}\right)} (\tau, 2u) 
$$
\noindent satisfies the properties (1), (2), (3) of Lemma \ref{lem:1.4}.
\end{lemma}

\begin{proof}
It is straightforward.
\end{proof}

In order to construct functions $a_j(\tau,v)$ satisfying the conditions (i), (ii), (iii) of Lemma \ref{lem:1.5}, we define the (modified) Zwegers functions 
$R_{j;\,m+1}(\tau,v)$ 
($j\; \in\; \ZZ$) as follows:
\begin{equation}
	\label{eq:1.4}
R_{j;\,m+1}(\tau,\, v) :=
\end{equation}
\begin{equation}
\displaystyle\sum\limits_{n \in \ZZ \atop n \equiv
- j \;
\hbox{mod}\; 2(m+1)}\left(\!{\rm sgn} (n+\frac{1}{2} + j - 2(m+1))\!\!\right. \left.- E \left(\left(n+2(m+1) \frac{{\rm Im}\,v}{{\rm Im}\,\tau} \right) \sqrt{\frac{{\rm Im} \,\tau}{m+1}}\right) \right) 
e^{-\frac{\pi i n^2}{2(m+1)}\tau - 2\pi i n v}  \nonumber ,
\end{equation}
where $E(x)=2\int\limits_0^x e^{-\pi u^2} du$.  In the same way as in \cite{Z}
one shows that this series converges to a real analytic function for all
$v, \tau \in \CC,\, \Im \tau >0$. 
\begin{lemma}
  \label{lem:1.6} The functions $R_{j;m+1}$ have the following properties: 
  \begin{list}{}{}
  \item (1) $R_{j;\,m+1} (\tau, v+\frac{1}{2}) = (-1)^j R_{j;\,m+1} (\tau,v)$,
hence $R_{j;\,m+1} (\tau,v+1) = R_{j;\,m+1} (\tau,v)$.
 
\item (2) $R_{j;\,m+1} (\tau, v) = - R_{2m+2-j;\,m+1} (\tau,-v)$.

 \item (3) $R_{j;\,m+1} (\tau,v) - e^{2\pi i (m+1) (2v-\tau)} R_{j;\,m+1} (\tau,v-\tau) = -2\; e^{\,\frac{\pi i \tau}{2(m+1)} j^2} e^{2\pi i j v}$.
  \end{list}
\end{lemma}
\begin{proof}
Claims (1) and (2) are
immediate from the definition of the function $R_{j;\, m+1}$. In order to prove (3), we first prove

\begin{equation}
	\label{eq:1.5}
R_{j;\,m+1}	(\tau, v-\tau) - e^{2\pi i (m+1) (\tau-2v)} R_{j;\, m+1} (\tau, v)\\[1ex] 
= 2 e^{-\frac{\pi i \tau}{2(m+1)} (2m+2-j)^2 + 2\pi i (2m+2-j)\tau} 
e^{-2\pi i (2m+2-j)v}.
\end{equation}
\noindent We have :
$$
\begin{array}{ll}
R_{j;\, m+1} (\tau, v-\tau) &=\displaystyle\sum\limits_
{n \equiv -j \;\hbox{mod}\; 2(m\!+\!1)} \!\!\!\left(\!\hbox{sign}\;(n\!+\!\frac{1}{2}\! +\! j-2(m\!+\!1))\right.\!\\[1ex]
&- E \!\left.\left(\left(\!n-2 (m\!+\!1) \!+\! 2(m\!+\!1) \frac{\hbox{Im}\; v}{\hbox{Im}\; \tau}\right) \!\!\sqrt{\frac{\hbox{Im} \;\tau}{m+1}}\right)\right)
e^{-\frac{\pi i n^2\tau}{2(m+1)} - 2\pi i n (v - \tau)}. 
\end{array}
$$

\noindent Letting $n=n^\prime + 2(m+1)$, the RHS can be rewritten as follows :
$$
\begin{array}{ll}
\displaystyle\sum\limits_{n^\prime \equiv -j\;\hbox{mod}\;2(m+1)} &\displaystyle \left(\hbox{sign} (n^\prime + \frac{1}{2} + j) - 
E \left(\left(n^\prime +2(m+1) \frac{\hbox{Im}\;v}{\hbox{Im}\,\tau}\right)
\sqrt{\frac{\hbox{Im}\;\tau}{m+1}}\right)\right)\\
&\times e^{-\frac{\pi i \tau}{2(m+1)} (n^\prime + 2(m+1))^2 -2\pi i (n^\prime + 2 (m+1)) (v - \tau)}.
\end{array}
$$
This can be rewritten, using that
$$
\hbox{sign}\; \left(n + \frac{1}{2} + j \right) = \hbox{sign}\; \left(n + \frac{1}{2} + j -2(m+1) \right) + 2 \delta_{n,\, -j}, 
$$
as follows :
$$
2 e^{-\frac{\pi i \tau}{2(m+1)} (2m+2-j)^2 + 2\pi i (2m+2-j)\tau}\;\; e^{-2\pi i (2m + 2 -j) v} + (I),
$$
where
$$
\begin{array}{ll}
(I)&= \sum\limits_{n \equiv -j \;\hbox{mod}\; 2(m+1)} 
\displaystyle\left(\hbox{sign}\;(n + \frac{1}{2} + j -2(m+1))\right.\\[2ex]  
& \left.
- E \left(\left(n+2(m+1) \frac{\hbox{Im}\; v}{\hbox{Im}\; \tau}\right) \sqrt{\frac{\hbox{Im}\, \tau}{m+1}}\right)\right) 
e^{-\frac{\pi i \tau n^2}{2(m+1)} -2\pi i n v} e^{2\pi i (m+1)(\tau - 2v)}   \\[1ex]
&= e^{2\pi i (m+1)(\tau - 2v)} R_{j;\,m+1} (\tau, v).
\end{array}
$$
This completes the proof of formula (\ref{eq:1.5}). Multiplying both sides of 
(\ref{eq:1.5}) by 
$-e^{2\pi i (m+1)( 2v - \tau)}$, we obtain claim (3).
\end{proof}

\vspace{1ex}
We put 
\begin{equation}
\label{eq:1.8a}
y = {\rm Im} \,\tau\, (>0),\,\,\, v = a\tau - b \,\,\hbox{(so that}\,\, 
\overline{v} = a \overline{\tau}-b), 
\,\,\hbox{where}\,\, a,\,b\; \in\; \RR.
\end{equation}
Then in the real coordinates $(a,b,y,\rm Re\,\tau)$ we have :
\begin{equation}
\label{eq:1.6}
\frac{\partial}{\partial a} + \overline{\tau} \frac{\partial}{\partial b} = 2iy \frac{\partial}{\partial v} ,\;\; \frac{\partial}{\partial a} + \tau \frac{\partial}{\partial b} = -2iy \frac{\partial}{\partial \overline{v}}
\end{equation}
and 
\begin{equation}
	\label{eq:1.7}
	2iy = \tau -\overline{\tau}, \quad 2iy a = v - \overline{v}, \quad 2iy b = \overline{\tau} v - \tau \overline{v}.
\end{equation}

\begin{lemma}
  \label{lem:1.7} We have:
  $$$$
  \vspace{-7ex}
    \begin{list}{}{}
\item (1) $\left(\frac{\partial}{\partial a} + \tau \frac{\partial}{\partial b}\right) R_{j;\,m+1} (\tau, v) = -4 \sqrt{(m+1)y} \;\, e^{-4\pi (m+1) a^2 y}\;\, \Theta_{j,\, m+1} (-\overline{\tau}, 2\overline{v})$.

\item (2) $\left(\frac{\partial}{\partial a} + \tau \frac{\partial}{\partial b} \right) R_{j;\, m+1} \left(-\frac{1}{\tau}, \frac{v}{\tau}\right) = -4 \frac{\tau}{|\tau|} \sqrt{(m+1)y}\,\; e^{-4\pi (m+1) \frac{b^2}{|\tau|^2}y} \Theta_{j,\, m+1} \left(\frac{1}{\overline{\tau}}, \frac{2\overline{v}}{\overline{\tau}}\right)$.
 \end{list}
\end{lemma}
\begin{proof}
Due to the second formula in (\ref{eq:1.6}), 
we have: 
$$
\left(\frac{\partial}{\partial a} + \tau \dfrac{\partial}{\partial b} \right) e^{\frac{-\pi i n^2}{2(m+1)}\tau - 2\pi i n v} = 0,
$$
and since 
$$
\left(\frac{\partial}{\partial a} + \tau \frac{\partial}{\partial b}\right) 
E \left(\left(n+2  (m+1) a\right) \sqrt{\frac{y}{m+1}}\right) = 4 e^{-\pi (n+2 (m+1)a)^2 \frac{y}{m+1}}\;\, (m+1) \sqrt{\frac{y}{m+1}}\;\;\;, 
$$
we obtain, using formula (\ref{eq:1.7}) : 
$$
\begin{array}{ll}
\displaystyle\left(\frac{\partial}{\partial a} +\tau \frac{\partial}{\partial b}\right) R_{j;\, m+1} (\tau, v) = -4 \sqrt{(m+1) y} \sum\limits_{n \equiv -j \;\hbox{mod}\; 2(m+1)}{\hspace{-6ex} e^{-\frac{\pi y}{m+1} (n+2(m+1)a)^2 -\frac{\pi i n^2}{2(m+1)} \tau - 2\pi i n(a\tau-b)}} \\
= -4 \sqrt{(m+1)y} \sum\limits_{n \equiv -j \;\hbox{mod}\; 2(m+1)}{e^{-\pi i n^2 \frac{\overline{\tau}}{4(m+1)} - 2\pi i n \overline{v}}}\;\;\;.
\end{array}
$$
Replacing $n$ by $-n$ in the last sum, we obtain claim (1) (cf (\ref{eq:1.2})).

In order to prove claim (2), let
$$
\tau^\prime=-\dfrac{1}{\tau},\;\;v^\prime=\dfrac{v}{\tau},\;\; y^\prime=\Im \tau^\prime=\dfrac{y}{|\tau|^2}\;\;,
$$
and introduce the coordinates $a^\prime,b^\prime \in \mathbb{R}$ by
$$
v^\prime =a^\prime \tau^\prime-b^\prime,\;\;\overline{v}^\prime= a^\prime \overline{\tau}^\prime-b^\prime.
$$
Then we have :
$$	a=-b^\prime,\;\;b=a^\prime ;\,\, \dfrac{\partial}{\partial a^\prime}+\tau^\prime \dfrac{\partial}{\partial b^\prime}=\dfrac{1}{\tau} \left(\dfrac{\partial}{\partial a} + \tau \dfrac{\partial}{\partial b} \right).
$$
Using this and claim (1), we obtain :
$$
\begin{array}{l}
\dfrac{1}{\tau} \left(\dfrac{\partial}{\partial a} + \tau \dfrac{\partial}{\partial b} \right) R_{j;\,m+1}\left( -\dfrac{1}{\tau},\,\dfrac{v}{\tau} \right) = \left(\dfrac{\partial}{\partial a^\prime} + \tau^\prime \dfrac{\partial}{\partial b^\prime} \right) R_{j;\,m+1} \left(\tau^\prime,\, v^\prime \right)\\
=-4\sqrt{(m+1)y^\prime} \,e^{-4\pi(m+1)a^{\prime^ 2} y^\prime} \Theta_{j,m+1}
\left( -\overline{\tau^\prime},\, 2\overline{v^\prime} \right)\\
=-4\sqrt{ \dfrac{(m+1)y}{|\tau|^2}}\, e^{-4\pi(m+1)b^{2} \frac{y}{|\tau|^2}} \Theta_{j,\,m+1}
\left(\dfrac{1}{\overline{\tau}} ,\, \dfrac{2\overline{v}}{\overline{\tau}} \right),
\end{array}
$$
proving (2)
\end{proof}

The formula in Lemma \ref{lem:1.7} (2) together with the modular 
transformation formula for $\Theta_{j,\,m+1}$ gives the following :

\begin{lemma}
  \label{lem:1.8}
$$
\begin{array}{l}
\left(\frac{\partial}{\partial a} + \tau \frac{\partial}{\partial b}\right) 
R_{j;\, m+1} \left(-\frac{1}{\tau}, \frac{v}{\tau}\right) \\[1ex] 
= -4i (-i\tau)^{\frac{1}{2}} \sqrt{\frac{y}{2}} \;e^{\frac{\pi  (m+1)}{y \tau} 
(\overline{\tau} v^2 -2 \tau v \overline{v} + \tau \overline{v}^2)}\!\!\!
\sum\limits_{k \in \ZZ/2(m+1)\ZZ}
{\hspace{-4.5ex} e^{\,-\frac{\pi i j k}{m+1}} 
\Theta_{k,\,m+1}(-\overline{\tau}, -2\overline{v})}. 
\end{array}
$$
\hfill $\Box$
\end{lemma}

For $j\,\in\, \ZZ$ such that $-s \leqq j \leq -s+ 2m+1$, we define the following functions :

$$
\overset{\sim}{a_j} (\tau, v) := R_{j;\, m+1} (\tau, v) + \frac{i}{\sqrt{2(m+1)}} 
{(-i \tau)^{-\frac{1}{2}}}\;\; e^{\frac{2\pi i (m+1) v^2}{\tau}} 
\sum\limits^{-s + 2m+1}_{k=-s} 
{\hspace{-2.5ex} e^{\,-\frac{\pi i j\,k}{m+1}} R_{k;\,m+1} \left(-\frac{1}{\tau}, \frac{v}{\tau}\right)}.
$$
 
\begin{lemma} 
  \label{lem:1.9}
The functions $\overset{\sim}{a_j}(\tau,v),\, -s \leqq j \leq -s+2m+1$, satisfy the following properties\,:

  \begin{list}{}{}
\item (1) $\overset{\sim}{a_j}(\tau,v)$ are holomorphic in $v$,

\item (2) $\overset{\sim}{a_j}(\tau,v) -e^{2\pi i (m+1)(2v -\tau)} \overset{\sim}{a_j}(\tau,v-\tau)=-2e^{-\frac{\pi i\tau}{2(m+1)}j^2} e^{2\pi i jv}$,

\item (3) $\overset{\sim}{a_j}(\tau,v+1) - \overset{\sim}{a_j}(\tau,v) = \frac{2i}{\sqrt{2(m+1)}} (-i\tau)^{-\frac{1}{2}} 
\displaystyle\sum\limits^{-s+2m+1}_{k=-s} {\hspace{-2.5ex}e^{-\frac{\pi i j k}{m+1}} e^{\frac{2\pi i (m+1)}{\tau}\left(v + \frac{k}{2(m+1)}\right)^2}}.$
  \end{list}
\end{lemma} 
\begin{proof} 
In view of the second formula in (\ref{eq:1.6}), we can rewrite Lemma \ref{lem:1.8} as follows :
\begin{equation}
	\label{eq:1.9}
\begin{array}{ll}
	\left(\dfrac{\partial}{\partial a} + \tau \dfrac{\partial}{\partial b} \right) \left((-i\tau)^{-\frac{1}{2}} e^{\frac{2\pi i(m+1)}{\tau}v^2} R_{j;\,m+1} \left(-\dfrac{1}{\tau}, \dfrac{v}{\tau} \right) \right) \\
= -4i \sqrt{\dfrac{y}{2}} e^{-4 \pi(m+1)a^2y} \displaystyle\sum\limits_{k\in \mathbb{Z}/2(m+1)\mathbb{Z}} e^{-\frac{\pi i jk}{m+1}} \Theta_{k,\,m+1} (-\overline{\tau},-2\overline{v}).
\end{array}\,.
\end{equation}

In order to prove (1), it suffices to show the following :
\begin{equation}
	\label{eq:1.10}
\begin{array}{ll}
	\left(\dfrac{\partial}{\partial a} + \tau \dfrac{\partial}{\partial b} \right) \left(\dfrac{-i}{\sqrt{2(m+1)}}(-i\tau)^{-\frac{1}{2}} e^{\frac{2\pi i(m+1)}{\tau}v^2} \sum\limits^{-s+2m-1}_{k=-s} e^{-\frac{\pi ijk}{m+1}} R_{k;\,m+1} \left(-\dfrac{1}{\tau}, \dfrac{v}{\tau} \right) \right)  \\
	= \left(\dfrac{\partial}{\partial a} + \tau \dfrac{\partial}{\partial b} \right) R_{j;\,m+1}(\tau,v).
\end{array}
\end{equation}
The LHS of (\ref{eq:1.10}) is equal to
$$
\dfrac{-i}{\sqrt{2(m+1)}} \sum\limits^{-s+2m+1}_{k=-s}e^{-\frac{\pi ijk}{m+1}}
 \left(\dfrac{\partial}{\partial a} + \tau \dfrac{\partial}{\partial b} \right) \left((-i\tau)^{-\frac{1}{2}} e^{\frac{2\pi(m+1)v^2}{\tau}} R_{k;\,m+1} \left(-\dfrac{1}{\tau}, \dfrac{v}{\tau} \right)\right),
$$
which by (\ref{eq:1.9}) is equal to
$$
\dfrac{-i}{\sqrt{2(m+1)}} \sum\limits^{-s+2m+1}_{k=-s}e^{-\frac{\pi ijk}{m+1}}
\left(- 
\sqrt{\dfrac{y}{2}} e^{-4\pi(m+1)ya^2} \sum\limits_{n\in \mathbb{Z}/2(m+1)\mathbb{Z}} e^{-\frac{\pi i k n}{m+1}} \Theta_{n,\,m+1} (-\overline{\tau},-2\overline{v})\right)
$$
$$
=-2\sqrt{\dfrac{y}{m+1}} e^{-4\pi(m+1)ya^2}  \sum\limits_{n\in \mathbb{Z}/2(m+1)\mathbb{Z}} \left( \sum\limits^{-s+2m+1}_{k=-s} e^{\frac{-\pi jk}{m+1}} e^{-\frac{\pi ikn}{m+1}} \right) \Theta_{n,\,m+1}(-\overline{\tau},-2\overline{v}).
$$
Since $\sum\limits^{-s+2m+1}_{k=-s} e^{-\frac{\pi ijk}{m+1}} e^{-\frac{\pi ikn}{m+1}} = 2(m+1)$ if $j+n\equiv 0$ mod\,$2(m+1)$, and $=0$ otherwise, we obtain that
the LHS of (\ref{eq:1.10}) is equal to $-4\sqrt{(m+1)y}\, e^{-4\pi(m+1)ya^2} \Theta_{-j,\,m+1}(-\overline{\tau},-2\overline{v})$.
Hence, using (\ref{eq:1.a}), 
we deduce from Lemma \ref{lem:1.7} (1) that (\ref{eq:1.10}) holds.

Next, we prove (2). We have :
$$
\begin{array}{l}
\overset{\sim}{a}_j(\tau,v-\tau)= R_{j;\,m+1}(\tau, v-\tau)\\[1ex]
+\dfrac{i}{\sqrt{2(m+1)}} (-i\tau)^{-\frac{1}{2}} 
e^{\frac{2\pi i(m+1)}{\tau}(v-\tau)^2} \sum\limits^{-s+2m+1}_{k=-s} 
e^{-\frac{\pi ijk}{m+1}} R_{k;\,m+1}\left(-\dfrac{1}{\tau}, \dfrac{v-\tau}{\tau}\right) =R_{j;\,m+1} (\tau,v-\tau) \\[1ex]
+e^{2\pi i(m+1)(\tau-2v)}\dfrac{i}{\sqrt{2(m+1)}}(-i\tau)^{-\frac{1}{2}} e^{\frac{2\pi i(m+1)}{\tau}v^2} \sum\limits^{-s+2m+1}_{k=-s} 
e^{-\frac{\pi ijk}{m+1}} R_{k;\,m+1}\left(-\dfrac{1}{\tau}, \dfrac{v}{\tau}\right)\\[1ex]
=R_{j;\,m+1} (\tau,v-\tau)+e^{2\pi i(m+1)(\tau-2v)} \left(\overset{\sim}{a}_{j} (\tau, v)- R_{j;\,m+1}(\tau,v)\right).
\end{array}
$$
Multiplying both sides by $e^{2\pi i(m+1)(2v-\tau)}$, we obtain
$$
\overset{\sim}{a}_j(\tau,v)-e^{2\pi i(m+1)(2v-\tau)} \overset{\sim}{a}_j (\tau, v-\tau)= R_{j;\,m+1}(\tau, v)- e^{2 \pi i(m+1)(2v-\tau)} R_{j;\,m+1} (\tau,  v-\tau).
$$
Applying Lemma \ref{lem:1.6} (3) to the RHS, we obtain claim (2).

In order to prove (3), note that, replacing $(\tau,v)$ by $\left(-\frac{1}{\tau}, \frac{v}{\tau}\right)$, Lemma \ref{lem:1.6}(3) gives 

\begin{equation}
	\label{eq:1.11}
	\begin{array}{ll}
R_{j;\,m+1}\left(-\dfrac{1}{\tau}, \dfrac{v}{\tau}\right)- e^{\frac{2\pi i (m+1)}{\tau} (2v+1)} \; R_{j;\,m+1} \left(-\dfrac{1}{\tau}, \dfrac{v+1}{\tau}\right) 
= -2 e^{\frac{\pi i j^2}{2(m+1)\tau}} e^{\frac{2\pi i j v}{\tau}}.	
\end{array}
\end{equation}
By definition of $R_{j;\,m+1}$ and using Lemma \ref{lem:1.6} (1), we have
$$
\begin{array}{ll}
\overset{\sim}{a}_j (\tau, v+1) - \overset{\sim}{a}_j (\tau, v) = \frac{i}{\sqrt{2(m+1)}} (-i\tau)^{-\frac{1}{2}} \sum\limits^{-s+2m+1}_{k=-s} e^{-\frac{\pi i j k}{m+1}} \left(e^{\frac{2\pi i (m+1)}{\tau} (v+1)^2} 
R_{k;\,m+1} \left(-\dfrac{1}{\tau}, \dfrac{v+1}{\tau}\right)\right. \\[3ex]
\left.-e^{\frac{2\pi i (m+1)}{\tau} v^2} R_{k;\,m+1} \left(-\dfrac{1}{\tau}, \dfrac{v}{\tau}\right)\right) = \frac{i}{\sqrt{2(m+1)}} (-i\tau)^{-\frac{1}{2}} \sum\limits^{-s+2m+1}_{k=-s} e^{-\frac{\pi i j k}{m+1}} e^{\frac{2\pi i (m+1)}{\tau} v^2} \\[3ex]
\times\left(e^{\frac{2\pi i (m+1)}{\tau}(2v+1)} R_{k;\,m+1} \left(-\dfrac{1}{\tau}, \dfrac{v+1}{\tau}\right)-R_{k;\,m+1}\left(-\dfrac{1}{\tau}, \dfrac{v}{\tau}\right) \right).
\end{array}
$$
Using (\ref{eq:1.11}), we deduce:
$$
\begin{array}{ll}
\overset{\sim}{a}_j (\tau, v+1) - \overset{\sim}{a}_j (\tau, v) = \frac{2i}{\sqrt{2(m+1)}} (-i \tau)^{-\frac{1}{2}} \; \sum\limits^{-s+2m+1}_{k=-s} e^{-\frac{\pi ijk}{m+1}}
e^{\frac{2\pi i (m+1)}{\tau} v^2 + \frac{\pi i k^2}{2(m+1)\tau} + \frac{2\pi i k v}{\tau}}\\
=\frac{2i}{\sqrt{2(m+1)}} (-i \tau)^{-\frac{1}{2}} \sum\limits^{-s+2m+1}_{k=-s} e^{-\frac{\pi ijk}{m+1}} 
e^{\frac{2\pi i (m+1)}{\tau} \left(v+\frac{k}{2(m+1)}\right)^2} ,
\end{array}	
$$
proving claim (3).
\end{proof}

Lemma \ref{lem:1.9} shows that the functions $a_j (\tau,v) := \frac{1}{2} \overset{\sim}{a_j}(\tau,v)$ satisfy the conditions (i), (ii), (iii) of 
Lemma \ref{lem:1.5}, 
hence, by Lemma \ref{lem:1.5} and Lemma \ref{lem:1.4}(4), we have
\begin{proposition}
\label{prop:1.10}

$G^{[m;\,s]} (\tau, u, v, t) = \frac{1}{2} e^{2\pi i(m+1) t} 
\displaystyle\sum\limits_{j=-s}^{-s+2m+1} {\hspace{-2.5ex} \overset{\sim}{a_j} (\tau, v) \left(\Theta_{j,\, m+1}- \Theta_{-j,\, m+1}\right)} (\tau, 2u) $.
\end{proposition}

We put
\begin{equation}
	\label{eq:1.12}
	\varphi^{[m;\,s]}_{\rm add}(\tau, u, v, t):=-\frac{1}{2}e^{2\pi i(m+1)t} \sum\limits_{j=-s}^{-s+2m+1}{\hspace{-2.5ex}R_{j;\, m+1}(\tau, v)
	\left(\Theta_{j,\, m+1}- \Theta_{-j,\, m+1}\right)} \;(\tau, 2u) 
\end{equation}
and
\vspace{1.5ex}
\begin{equation}
	\label{eq:1.13}
\overset{\sim}{\varphi}^{[m;\,s]} := \varphi^{[m;\,s]}+\varphi^{[m;\,s]}_{\rm add}.
\end{equation}
Then we obtain the following theorem.

\pagebreak

\begin{theorem}
  \label{th:1.11}
$$$$
\vspace{-7ex}
\begin{list}{}{}
\item(1) $ \overset{\sim}{\varphi}^{[m;\,s]}|_S=\overset{\sim}{\varphi}^{[m;\,s]}$; \,\, $\overset{\sim}{\varphi}^{[m;\,s]} (\tau+1, u, v, t) = \overset{\sim}{\varphi}^{[m;\,s]} (\tau, u, v, t)$.

\item (2) $ \overset{\sim}{\varphi}^{[m;\,s]} (\tau, u+a, v+b, t) =  \overset{\sim}{\varphi}^{[m;\,s]} (\tau, u,v,t)$ for all $a$, $b\; \in\; \frac{1}{2}\ZZ$ such that $a+b\; \in\; \ZZ$.

\item (3) $ \overset{\sim}{\varphi}^{[m;\,s]}  (\tau, u+a\tau, v+b\tau,t) =  
q^{(m+1)(b^2-a^2)} e^{4\pi i (m+1)(bv-au)}  
\overset{\sim}{\varphi}^{[m;\,s]} (\tau,u,v,t)$

\text{for all} $a,\,b\, \in\,\frac{1}{2}\, \ZZ$ \text{such that} $a+b\, \in\, \ZZ$.
\end{list}
\end{theorem}

\begin{proof}
We have by definition of $\varphi^{[m;\,s]}_{\hbox{add}}$ :
$$
\begin{array}{ll}
\varphi^{[m;\,s]}_{\hbox{add}}|_S (\tau, u, v, t) := \dfrac{1}{\tau} \varphi^{[m;\,s]}_{\hbox{add}} \left(-\dfrac{1}{\tau}, \dfrac{u}{\tau}, \dfrac{v}{\tau}, t -\dfrac{u^2-v^2}{\tau}\right)\\[3ex]
=-\dfrac{1}{2\tau} e^{2\pi i (m+1) \left(t-\frac{u^2-v^2}{\tau}\right)} \;\sum\limits^{-s+2m+1}_{j=-s}{R_{j;\,m+1} \left(-\dfrac{1}{\tau}, \dfrac{v}{\tau}\right) \left( \Theta_{j,\,m+1}-\Theta_{-j,\,m+1} \right)\left(-\dfrac{1}{\tau}, \dfrac{2u}{\tau}\right)}.
\end{array}
$$
Using the modular transformation formula of $\Theta_{j,\,m+1}$, 
the RHS is rewritten as follows :
$$
\begin{array}{l}
\frac{1}{2} e^{2\pi i (m+1)t} \sum\limits_{k=-s}^{-s+2m+1} 
\frac{i}{\sqrt{2(m+1)}}(-i\tau)^{-\frac{1}{2}} \;e^{\frac{2\pi i (m+1)}{\tau}v^2}\\[1ex] 
\times\sum\limits_{k=-s}^{-s+2m+1} 
e^{-\frac{\pi i j k}{m+1}} R_{j;\,m+1} \left(-\frac{1}{\tau}, \frac{v}{\tau}\right) (\Theta_{k,\,m+1} - \Theta_{-k,\,m+1}) (\tau,2u).
\end{array}
$$
By the definition of $\overset{\sim}{a}_k (\tau,v)$, this is equal to
$$
\frac{1}{2} e^{2\pi i (m+1) t} \sum\limits^{-s+2m+1}_{k= -s}{(\overset{\sim}{a}_k (\tau,v) - R_{k;\, m+1} (\tau,v))} (\Theta_{k,\,m+1} - \Theta_{-k,\,m+1}) (\tau, 2u).
$$
Using Proposition 1.10 and the definition of $\varphi^{[m;\,s]}_{\hbox{add}}$, this can be rewritten as
$G^{[m;\,s]} (\tau, u, v, t) + \varphi^{[m;\,s]}_{\hbox{add}} (\tau, u, v, t).$ Thus, by definition of $G^{[m;\,s]}$ we have 
$$
\left(\varphi^{[m;\,s]} + \varphi^{[m;\,s]}_{\hbox{add}}\right)|_S = \varphi^{[m;\,s]} + \varphi^{[m;\,s]}_{\hbox{add}},
$$
proving the first formula in (1). The second formula in (1) is straightforward.

Next, let $a$, $b\in \frac{1}{2}\ZZ$ be such that $a+b \in \ZZ$. Using 
Lemma \ref{lem:1.6} (1) and equation (\ref{eq:1.b}),
we obtain that $\varphi^{[m;\,s]}_{\hbox{add}} (\tau, u+a, v+b, t) = \varphi^{[m;\,s]}_{\hbox{add}} (\tau, u, v, t)$. 
This together with Lemma \ref{lem:1.2} (2) proves claim (2). 

Claim(3) is easily deduced from claims (1) and (2) as follows. By (1) we have: 
\begin{equation}
	\label{eq:1.14}
	\tau \overset{\sim}{\varphi}^{[m;\,s]} (\tau, u, v, t)= 
	e^{-\frac{2\pi i (m+1)}{\tau}(u^2-v^2)} \overset{\sim}{\varphi}^{[m;\,s]} \left(-\frac{1}{\tau}, \frac{u}{\tau}, \frac{v}{\tau}, t \right).
\end{equation}
Replacing $(u,v)$ by $(u+a\tau,v+a\tau)$ in this formula, and using claim (2), we obtain :
$$
\tau \overset{\sim}{\varphi}^{[m;\,s]} (\tau, u+a\tau, v+b\tau)=e^{-\frac{2\pi i(m+1)}{\tau}((u+a\tau)^2-(v+b\tau)^2)}
\overset{\sim}{\varphi}^{[m;\,s]} \left(-\frac{1}{\tau}, \frac{u}{\tau}, \frac{v}{\tau}, t \right).
$$
By (\ref{eq:1.14}), the RHS of this formula is equal to 
$\tau q^{(m+1)(b^2-a^2)} e^{4\pi i (m+1)(bv-au)}   \overset{\sim}{\varphi}^{[m;\,s]} (\tau, u, v, t)$, proving (3).
\end{proof}

Translating the formulas for $\varphi^{[m;\,s]}$ and its modification to $\Phi^{[m;\,s]}$ and its modification, and using (\ref{eq:1.a}), 
we obtain

\begin{equation}
	\label{eq:15}
\Phi^{[m;s]}_{\hbox{add}} 
(\tau, z_1, z_2, t) = \frac{1}{2} e^{2\pi i (m+1)t} 
\sum\limits^{-s+2m+1}_{j=-s}{R_{j;\,m+1} \left(\tau, \frac{z_1-z_2}{2}\right) (\Theta_{j,\,m+1}-\Theta_{-j,\,m+1}) (\tau, z_1+z_2)}
\end{equation}
and 
$$
\overset{\sim}{\Phi}^{[m;\,s]} 
= \Phi^{[m;\,s]} + \Phi^{[m;s]}_{\hbox{add}},
$$
where 
$\Phi^{[m;s]}_{\hbox{add}}$ is obtained from $\varphi^{[m;s]}_{\hbox{add}}$
by the change of variables $u
=-\frac{z_1+z_2}{2}$, $v =\frac{z_1-z_2}{2}$. 
Then Theorem \ref{th:1.11} implies

\begin{theorem}
\label{th:1.12} 
The function $\overset{\sim}{\Phi}^{[m;\,s]}$ has the following modular and elliptic transformation properties: 
$$$$
\vspace{-6ex}
\begin{list}{}{}

\item (1) $\overset{\sim}{\Phi}^{[m;\,s]}|_S = \overset{\sim}{\Phi}^{[m;\,s]}$, i.e.
$
\overset{\sim}{\Phi}^{[m;\,s]} \left(-\frac{1}{\tau}, \frac{z_1}{\tau}, \frac{z_2}{\tau}, t \right) = \tau e^{\frac{2\pi i (m+1) z_1 z_2}{\tau}} \overset{\sim}{\Phi}^{[m;\,s]} (\tau, z_1, z_2, t).
$

\item (2) $\overset{\sim}{\Phi}^{[m;\,s]} (\tau+1, z_1, z_2,t) = \overset{\sim}{\Phi}^{[m;\,s]} (\tau, z_1, z_2, t).$

\item (3) $\overset{\sim}{\Phi}^{[m;\,s]} (\tau, z_1+a, z_2+b,t) = \overset{\sim}{\Phi}^{[m;\,s]} (\tau, z_1, z_2, t)$ if $a$, $b \in \ZZ$.

\item(4) $\overset{\sim}{\Phi}^{[m;\,s]} (\tau, z_1 + a\tau, z_2+b\tau, t) = q^{-(m+1)ab}\; e^{-2\pi i (m+1) (b z_1+a z_2)} \; \overset{\sim}{\Phi}^{[m;\,s]} (\tau, z_1, z_2, t)$ if $a$, $b \in \ZZ$.
\end{list}
\end{theorem}
\hfill $\Box$


\begin{remark}   
\label{rem:1.13}
  After replacing $(\tau,z_1,z_2,t)$ by $(\frac{\tau}{M},\frac{z_1}{M},\frac{z_2}{M},t)$, where $M\neq 0$, Theorem \ref{th:1.12} (1) gives:
$$
\overset{\sim}{\Phi}^{[m;\,s]} ( \frac{\tau}{M},\frac{z_1}{M},\frac{z_2}{M},t ) = \frac{M}{\tau} \overset{\sim}{\Phi}^{[m;\,s]} ( -\frac{M}{\tau}, \frac{z_1}{\tau},\dfrac{z_2}{\tau},t-\frac{z_1z_2}{\tau M} ).
$$
\end{remark}


\section{Modular transformation formulae for modified normalized characters of admissible $\widehat{\sl}_{2|1}$-modules.}
\label{sec:2}

Recall an explicit description of principal admissible weights $\Lambda$ of an affine Lie superalgebra $\hat{\fg}$ (see [KW], Section 3).

Let $M$ be a positive integer and let 
$S_{M}=\{(M-1)\delta\, +\, \alpha_0, \alpha_1, \ldots, \alpha_\ell\}$, where 
$\{\alpha_0, \alpha_1, \ldots, \alpha_\ell \}$ is the set of simple roots of 
$\hat{\fg}$ and $\delta$ is the primitive imaginary root. Let $\Lambda^0$ be 
a partially integrable weight of level $m$. 
Then all principal admissible weights, associated to the pair 
$(M,\Lambda^0)$ are obtained as follows. 
Let $\fh$ be the Cartan subalgebra of $\fg$ and $W$ the (finite) Weyl group. 
Let $\beta \in \fh^\ast$ and $y \in W$ be such that the set 
$S :=t_\beta y (S_M)$ lies in the subset of positive roots of $\hat{\fg}$. 
Such subsets are called simple.
All principal admissible weights with respect to a simple subset $S$, 
associated to the pair 
$(M,\Lambda^0)$ are of the form (up to adding a multiple of $\delta$) :
\begin{equation}
	\label{eq:2.1}
\Lambda = (t_\beta y) (\Lambda^0 - (M-1) (K+h^\vee) \Lambda_0 + \hat{\rho}) - \hat{\rho},
\end{equation}
and all of them have level
\begin{equation}
\label{eq:2.2}
	k=\frac{m+h^\vee}{M}- h^\vee, 
\end{equation}
where $h^\vee$ is the dual Coxeter number.

For the Lie superalgebra $\widehat{\fg}=\widehat{\sl}_{2|1}$ the dual 
Coxeter number $h^\vee=1$ and $\widehat{\rho}=\Lambda_0$. There are two kinds 
of simple subsets:
$$
S^{+}_{k_1,k_2}=\{ k_0\delta+\alpha_0, k_1\delta+\alpha_1, k_2\delta+\alpha_2 \},\,M=
k_0+k_1+ k_2+1,\, 
k_i \in \mathbb{Z}_{\geq 0},\,y=1,\,\beta=-k_1\alpha_2-k_2 \alpha_1;
$$
$$
S^{-}_{k_1,k_2}=\{ k_0\delta-\alpha_0, k_1\delta-\alpha_1, k_2\delta-\alpha_2 \},\, M=k_0+ k_1+k_2-1,\, k_i \in \mathbb{Z}_{> 0},\, 
y=r_{\alpha_1+\alpha_2},\, \beta=k_1\alpha_2+k_2 \alpha_1.
$$
By (\ref{eq:2.1}) 
all principal admissible weights with respect to $S^{+}_{k_1,k_2}$ and 
$S^{-}_{k_1,k_2}$, associated to the pair $(M,\Lambda^0=\Lambda_{m;s})$ are respectively:
$$
\Lambda^{[s]+}_{k_1,k_2}=k\Lambda_0 -(k+1)( k_2-\frac{s}{k+1} ) \alpha_1 - (k+1)k_1\alpha_2-(k+1)k_1(k_2-\frac{s}{k+1})\delta ;
$$
$$
\Lambda^{[s]-}_{k_1,k_2}=k\Lambda_0 +(k+1)k_2\alpha_1+(k+1) ( k_1-\frac{s}{k+1} ) \alpha_2 - (k+1)k_2(k_1-\frac{s}{k+1})\delta,
$$
and, by (\ref{eq:2.2}), the level of all of them is $k=\frac{m+1}{M}-1$. 

Recall that the normalized character and supercharacter $ch^{\pm}_{\Lambda}$ of a $\widehat{\fg}$-module $L(\Lambda)$ of level $k$ is given by
\begin{equation}
\label{eq:2.a}
ch^{\pm}_{\Lambda}=e^{-(\frac{|\Lambda+\hat{\rho}|^2}{2(k+h^\vee)} -
\frac{\sdim \fg}{24})\delta} ch^{\pm}_{L(\Lambda)}.
\end{equation}

Formula (3.28) from [KW] (see [GK] for its proof) gives in the case of $\widehat{\fg}=\widehat{\sl}_{2|1}$ the following expressions for the normalized characters and supercharacters of principal admissible modules, in terms of the functions $\Phi^{[m;\,s]}$:

\begin{lemma}
\label{lem:2.1}
Let $\Lambda$ be a principal admissible weight of level $k=\frac{m+1}{M}-1$ for $\widehat{\sl}_{2|1}$, 
associated to the pair $(M,\Lambda_{m;\,s})$. 
Then the normalized supercharaters $ch^{-}_{\Lambda}$ are 
given by the following formulae (where $\widehat{R}^{-}$ is the affine 
 $\widehat{\sl}_{2|1}$ superdenominator): 
$$
\Lambda=\Lambda^{[s]+}_{k_1,k_2}: \,\, 
(\widehat{R}^{-}ch^{-}_{\Lambda}) (\tau,z_1,z_2,t) =
q^{\frac{m+1}{M}k_1k_2} e^{\frac{2\pi i(m+1)}{M}(k_2z_1+k_1z_2)}
\Phi^{[m;\,s]}(M\tau, z_1+k_1\tau,z_2+k_2\tau,\frac{t}{M});
$$
$$
\Lambda=\Lambda^{[s]-}_{k_1,k_2}:\,\, 
(\widehat{R}^{-}ch^{-}_{\Lambda}) (\tau,z_1,z_2,t)
=-q^{\frac{m+1}{M}(M-k_1)(M-k_2)} 
e^{\frac{2\pi i(m+1)}{M}((M-k_2)z_1+(M-k_1)z_2)}
$$
$$
 \times\Phi^{[m;\,s]} (M\tau, z_1+(M-k_1)\tau,z_2+(M-k_2)\tau,\frac{t}{M}).
$$ 
\hfill $\Box$
\end{lemma}

In order to derive the modular transformation formulae for the modified admissible $\widehat{\sl}_{2|1}$-characters we need the following theorem.

\begin{theorem}
\label{th:2.2}
Let $M$ be a positive integer and let $m$ be a non-negative integer, such that gcd\,$(M,2m+2)=1$ if $m>0$. Then

\vspace{1.5ex}
$
\overset{\sim}{\Phi}^{[m;\,s]} (-\frac{M}{\tau},\frac{z_1}{\tau},\frac{z_2}{\tau},t-\frac{z_1z_2}{\tau M})
$

$
 = \frac{\tau}{M}\sum\limits_{j,k\in\mathbb{Z}/M\mathbb{Z}} q^{\frac{m+1}{M}jk} e^{\frac{2\pi i(m+1)}{M}(kz_1+jz_2)}
\overset{\sim}{\Phi}^{[m;\,s]}(M\tau,z_1+j\tau,z_2+k\tau,t).
$

(By Theorem \ref{th:1.12}(4), each term in the RHS depends only on $j,k \mod M$.)
\hfill $\Box$
\end{theorem}

Given coprime positive integers $p$ and $q$, for each integer $n\in [s+1, s+q]$ there exist unique integers $n^\prime \in [s+1,s+p]$ and $b_n$, such that
\begin{equation}
	\label{eq:2.3}
n= n^\prime q + b_n\, p.	
\end{equation}
Furthermore, the set
\begin{equation}
\label{eq:2.4a}
I^{[s]}_{q,p} := \{b_n|\, n\in [s+1, s+q] \}
\end{equation}
consists of $q$ distinct integers. Any $n\in \ZZ$ can be uniquely represented in the form (\ref{eq:2.3}), where $n^\prime \in \ZZ$ and $b_n \in I^{[s]}_{q,p}$, and this decomposition has the following properties :

\begin{list}{}{}
\item (i) $n \geq s+1$ iff $n^\prime \geq s+1$;

\item (ii) if $j, j_0 \in \ZZ$ are such that $j\equiv q j_0\;\text{mod} \;p $, 
then $n \equiv \pm j \;\text{mod}\; 
p$ iff $n^\prime \equiv \pm j_0 \;\text{mod}\; p$.
\end{list}

We shall apply this setup to $p=2m+2$, $q=M$, and let $I^{[s]}
=I^{[s]}_{M,\, 2m+2}$. The same proof as that of Lemma 6.2 (a) in [KW] (except that we expand $\Phi^{[m;\,s]}_{1} (\tau, -z_2, -z_1)$ in the domain $\Im\,z_2>0$) gives the following result.

\begin{lemma} 
\label{lem:2.3}
Let $M$ be a positive integer and $m$ be a non-negative integer, such that \linebreak $\gcd (M,\,2m+2)=1$. Then  

\vspace{2ex}
$\Phi^{[m;\,s]} \left(\frac{\tau}{M}, \frac{z_1}{M}, \frac{z_2}{M}, t \right)$

\vspace{1.5ex}
$=\sum\limits_{0 \leq a < M \atop b \in I^{[s]}} e^{\frac{2\pi i(m+1)}{M} (az_1+(a+2b)z_2) } \;q^{\frac{m+1}{M} a(a+2b)}\; 
\Phi^{[m;s]} (M\tau, z_1+ (a+2b) \tau, z_2+a\tau, t)$.
\end{lemma}
\hfill $\Box$
\begin{remark}
\label{rem:2.4}
Since Remark 6.6 from[KW] holds for arbitrary $s$, we see that Lemma \ref{lem:2.3} holds if $gcd(M,m+1)=1$, the set $I^{[s]}$ is replaced by the set $I^{[s]}_{M,\,m+1}$, and $2b$ is replaced in each summand by $b$.
\end{remark}
The proof of the following lemma is the same as that of Lemma 6.3 from [KW].

\begin{lemma}
\label{lem:2.5}
Let $M$ and $m$ be as in Theorem \ref{th:2.2}. For each integer $j \in [-s,-s+2m+1]$ take the unique integer $j_0$ in the same interval, such that $j \equiv Mj_0 \;\text{mod}\; 2m+2$. Then

\begin{list}{}{}
	\item(1) $R_{j;\, m+1} \left(\frac{\tau}{M}, \frac{v}{M} \right) =
	\sum\limits_{b \in I^{[s]}} q^{-\frac{m+1}{M}b^2} 
e^{-\frac{4\pi i (m+1)}{M}bv} \;\; R_{j_0;\,m+1} \left(M\tau, v+b\tau \right).$
	
	\item(2) $\Theta_{\pm j, m+1} \left(\frac{\tau}{M}, \frac{2u}{M}\right) = \sum\limits_{a\in I^{[s]}}{{q^{\frac{m+1}{M}a^2}}\; e^{\frac{4\pi i (m+1)}{M} au} \;\Theta_{\pm j_0,\,m+1}} (M\tau, 2(u+a\tau)).$
\end{list}
\hfill $\Box$
\end{lemma}

>From the definition (1.13) of $\varphi^{[m;\,s]}_{\text{add}}$ and Lemmas \ref{lem:1.6} (2) and \ref{lem:2.5} we obtain 

\begin{lemma}
\label{lem:2.6}
Let $M$ and $m$ be as in Lemma \ref{lem:2.3}. Then 

$$
\varphi^{[m;\,s]}_{\text{add}} \left(\frac{\tau}{M}, \frac{u}{M}, \frac{v}{M}, t \right)
=\sum\limits_{a,b \in I^{[s]}}q^{\frac{m+1}{M} (a^2-b^2)} 
e^{\frac{4 \pi i (m+1)}{M}(au-bv)} \varphi^{[m;\,s]}_{\text{add}}
(M\tau, u+a\tau, v+b\tau,t).
$$
\hfill $\Box$
\end{lemma}

\begin{proof} \emph{of Theorem} \ref{th:2.2}.
It follows from Lemma \ref{lem:1.1} (5) and Theorem \ref{th:1.11} (3) that each term in the RHS of the equation in Lemma \ref{lem:2.3} remains unchanged if 
we add to $a$ or to $b$ an integer multiple of $M$. Similarly, it follows from Theorem \ref{th:1.12} (4) that each term in the RHS of the equation  in Lemma \ref{lem:2.6} remains unchanged if we add to $a$ or to $b$ an integer multiple 
of $M$. Hence, by Remark \ref{rem:1.13}, Theorem \ref{th:2.2} follows from Lemmas \ref{lem:2.3} and \ref{lem:2.6} and, for $m=0$, from Remark \ref{rem:2.4} and the observation that $\varphi^{[0;\,s]}_{\text{add}}=0$.
\end{proof}

As in [KW], let $\xi = -\frac{1}{2}(\alpha_1+ \alpha_2)$ and consider the 
twisted normalized admissible characters and supercharacters $t_{\xi} ch^{\pm}_{\Lambda}$ and their denominators and superdenominators $t_{\xi} \widehat{R}^{\pm}$. As in [KW], we shall use the following notation :
\begin{equation} \label{eq:2.6a}
ch_{\Lambda;\,0}^{(0)}=ch^{-}_{\Lambda}, \;\; ch^{(\frac{1}{2})}_{\Lambda;\,0}= ch^{+}_{\Lambda},\;\; ch^{(0)}_{\Lambda;\,\frac{1}{2}}=t_{\xi} ch^{-}_{\Lambda}, \;\; ch^{(\frac{1}{2})}_{\Lambda;\,\frac{1}{2}} = t_{\xi} ch^{+}_{\Lambda},
\end{equation}

\noindent and similarly for their (super)denominators :
\begin{equation} \label{eq:2.6b}
\widehat{R}^{(0)}_{\;\,0} = \widehat{R}^{-},\,\, \widehat{R}^{(\frac{1}{2})}_{\;\;0}= \widehat{R}^{+},\,\, \widehat{R}^{(0)}_{\,\,\frac{1}{2}}= t_{\xi} \widehat{R}^{-},\,\, \widehat{R}^{(\frac{1}{2})}_{\;\,\frac{1}{2}}=t_{\xi} \widehat{R}^{+}.
\end{equation}

Recall that the untwisted and twisted (super)denominators for 
$\widehat{\sl}_{2|1}$ can be conveniently 
written in terms of Jacobi's four theta functions of degree $2$
(see (7.4) in [KW]) :
\begin{equation} \label{eq:2.6c}
\widehat{R}^{(\epsilon)}_{\,\,\epsilon^\prime} (\tau, z_1, z_2, t) = 
(-1)^{2\epsilon (1-2\epsilon^\prime)} \, i e^{2\pi i t} \frac{\eta(\tau)^3 
\vartheta_{11} (\tau, z_1+z_2)} 
{\vartheta_{1-2\epsilon^\prime,\, 1-2\epsilon} (\tau, z_1) 
\vartheta_{1-2\epsilon^\prime,\, 1-2\epsilon} (\tau, z_2)},
\end{equation} 
where $\epsilon,\epsilon^\prime = 0$ or $\frac{1}{2}$.

As in the case $s=0$ in [KW], Section 7, the non-twisted and twisted 
normalized admissible (super)characters can be written in terms of the 
following functions :
\begin{equation}
	\label{eq:2.4}
\Psi^{[M,\, m,\, s;\,\epsilon]}_{a,\, b;\,\epsilon^{\prime}} (\tau,z_1,z_2,t) = q^{\frac{(m+1)ab}{M}} e^{\frac{2\pi i (m+1)}{M} (bz_1+az_2)} \Phi^{[m;\,s]} 
(M\tau, z_1+a\tau+\epsilon, z_2+b\tau+\epsilon, \frac{t}{M}),
\end{equation}
where $\epsilon, \epsilon^\prime=0$ or $\frac{1}{2}$, and $a,b\in \epsilon^\prime + \ZZ$.
Namely, Lemma \ref{lem:2.1}, along with Lemma \ref{lem:1.1} (5) and (6), 
implies the following character formulae.

\begin{proposition}
\label{prop:2.7}
\begin{list}{}{}
\item(1) If $\Lambda=\Lambda^{[s]+}_{j,\, k}$, where $j, k \in \ZZ$, $0 \leq j,k,j+k \leq M-1$, then 
$$
\left(\widehat{R}^{(\epsilon)}_{\,\,\epsilon^\prime} ch^{(\epsilon)}_{\Lambda;\,\epsilon^\prime} \right) (\tau, z_1, z_2, t) = \Psi^{[M,\,m,\,s;\,\epsilon]}_{j+\epsilon^\prime,\, k+\epsilon^\prime,\, \epsilon^\prime} (\tau, z_1,z_2,t).
$$

\item(2) If $\Lambda=\Lambda^{[s]-}_{j,\,k}$, where $j,k \in \ZZ$, $1 \leq j, k, j+k \leq M$, then 
$$\left(\widehat{R}^{(\epsilon)}_{\,\,\epsilon^\prime} 
ch^{(\epsilon)}_{\Lambda;\, \epsilon^{\prime}} \right) (\tau, z_1, z_2, t) = -\Psi^{[M,\,m,\,s ;\,\epsilon]}_{M+\epsilon^\prime-j,\, M+\epsilon^\prime-k;\,\epsilon^\prime} (\tau, z_1, z_2, t).
$$
\end{list}
\hfill $\Box$
\end{proposition}

As in [KW], introduce the modification 
$\overset{\sim}{\Psi}^{[M,\,m,\,s ;\, \epsilon]}_{a,\,b ;\,\epsilon^{\prime}}$ 
of the function 
$\Psi^{[M,\,m,\,s ;\,\epsilon]}_{a,\,b ;\,\epsilon^{\prime}}$ by replacing $\Phi^{[m;\,s]}$ in the RHS of (\ref{eq:2.4}) by its modification $\overset{\sim}{\Phi}^{[m;\,s]}$. The following theorem is immediate by Theorem \ref{th:2.2} and Theorem \ref{th:1.12} (2).

\begin{theorem}
\label{th:2.8}
Let $\epsilon$, $\epsilon^\prime=0$ or $\frac{1}{2}$, and let $j$, $k$  $\in\,\epsilon^\prime+\ZZ/M\ZZ$. Then
$$
\overset{\sim}{\Psi}^{[M,\,m,\,s;\,\epsilon]}_{j,\, k;\,\epsilon^\prime} \left(-\frac{1}{\tau}, \frac{z_1}{\tau}, \frac{z_2}{\tau}, t - \frac{z_1 z_2}{\tau} \right)
= \frac{\tau}{M} \sum\limits_{a,\,b \,\in\, \epsilon + \ZZ/M \ZZ}{e^{-\frac{2\pi i (m+1)}{M} (ak+bj)}} \;\overset{\sim}{\Psi}^{[M,\, m,\, s;\,\epsilon^\prime]}_{a,\, b;\,\epsilon} (\tau, z_1, z_2, t);
$$

$$
\overset{\sim}{\Psi}^{[M,\, m,\, s;\,\epsilon]}_{j,\,k ;\,\epsilon^\prime} (\tau+1, z_1, z_2, t) = e^{\frac{2\pi i (m+1)}{M}\,jk} \;\; \overset{\sim}{\Psi}^{[M,\,m,\,s; |\epsilon-\epsilon^\prime|]} (\tau, z_1, z_2, t).
$$
\hfill $\Box$ 
\end{theorem}

As in [KW], in order to state a unified modular transformation formula for modified normalized admissible characters, it is convenient, for each $s \in \ZZ$, $0 \leq s \leq m$, to introduce the following notations :
$$
ch^{[M,\, m,\, s;\,\epsilon]}_{j+\epsilon^\prime,\, k+\epsilon^\prime;\,\epsilon^\prime} := 
ch^{(\epsilon)}_{\Lambda^{[s]+}_{j,\,k};\, \epsilon^\prime},\;\; 
ch^{[M,\, m,\,s;\,\epsilon]}_{M+\epsilon^\prime-j,\, M+\epsilon^\prime-k;\,\epsilon^\prime} = 
-ch^{(\epsilon)}_{\Lambda^{[s]-}_{j,\,k};\,\epsilon^\prime}.
$$
Then 
$$
\{ch^{[M,\, m,\,s;\,\epsilon]}_{j+\epsilon^\prime,\, k+\epsilon^\prime;\,\epsilon^\prime}| j, k \in \ZZ,\, 0 \leq j, k \leq M-1\}
$$
\noindent  is (up to a sign) precisely the set of all admissible characters (resp. supercharacters), associated to the pair
$(M,\, \Lambda_{m;\,s})$, if $\epsilon^\prime = 0$ and $\epsilon=\frac{1}{2}$ (resp. $\epsilon=0$), and it is the set 
of all twisted admissible characters (resp. supercharacters), associated to 
the pair $(M,\, \Lambda_{m;\,s})$, if $\epsilon^\prime = \frac{1}{2}$ and 
$\epsilon=\frac{1}{2}$ (resp. $\epsilon=0$). 

In view of these observations, introduce the \emph{modified} normalized characters $(\epsilon=\frac{1}{2}, \epsilon^\prime=0)$, supercharacters $(\epsilon=0, \epsilon^\prime=0)$, twisted characters ($\epsilon = \frac{1}{2},\, 
\epsilon^\prime=\frac{1}{2}$), and twisted supercharacters 
($\epsilon=0, \epsilon^\prime = \frac{1}{2}$), letting 
\begin{equation}
	\label{eq:2.5}
	\overset{\sim}{ch}^{[M,\,m,\,s;\,\epsilon]}_{j,\, k;\,\epsilon^\prime} (\tau, z_1, z_2, t) = \frac{{\overset{\sim}{\Psi}}^{[M,\, m,\,s;\,\epsilon]}_{j,\,k;\,\epsilon^\prime} (\tau, z_1, z_2, t)}{\widehat{R}^{(\epsilon)}_{\epsilon^\prime} (\tau,z_1, z_2, t)}, \;\, j, k \in \epsilon^\prime + \ZZ,\,\, 0\leq j,\,k<M.
\end{equation}
Then from formulae (7.5) and (7.6) in [KW] and Theorem 2.8 we obtain the 
following theorem.

\begin{theorem}
\label{th:2.9}
Let $M$ be a positive integer and let $m$ be a non-negative integer, such that gcd $(M,2m+2)=1$ if $m>0$. One has the following modular transformation 
formulae for each $s\in\ZZ,\, 0\leq s\leq m$ ($\epsilon, \epsilon^\prime=0$ 
or $\frac{1}{2};\,j,k \in \epsilon^\prime + \ZZ,\, 0 \leq j, k<M$) :

$$
\begin{array}{l}
\overset{\sim}{ch}^{[M,\, m,\,s\,;\epsilon]}_{j,\,k;\,\epsilon^\prime} \left(-\frac{1}{\tau}, \frac{z_1}{\tau}, \frac{z_2}{\tau}, t-\frac{z_1 z_2}{\tau}\right)\\
= (-1)^{4 \epsilon \epsilon^\prime} \frac{1}{M} \sum\limits_{a, b\, \in\, \epsilon + \ZZ \atop 0 \leq a, b<M}{e^{-\frac{2\pi i (m+1)}{M} (ak+bj)}} \overset{\sim}{ch}^{[M, m,\,s;\,\epsilon^\prime]}_{a,b;\,\epsilon} (\tau, z_1, z_2, t); \\
\overset{\sim}{ch}^{[M,\, m,\,s;\,\epsilon]}_{j,k;\,\epsilon^\prime} (\tau+1, z_1, z_2, t) =
e^{2\pi i \frac{(m+1)jk}{M}-\pi i \epsilon^\prime} 
\;\overset{\sim}{ch}^{[M, m,\,s;\,|\epsilon-\epsilon^\prime|]}_{j,k;\,\epsilon^\prime} (\tau, z_1, z_2, t).
\end{array}
$$
\end{theorem}

As in [KW], Section 9, in order to perform the quantum Hamiltonian reduction for $\widehat{sl}_{2|1}$, we need a different choice of the twisting vector : $\xi^\prime=\frac{1}{2}(\alpha_1-\alpha_2)$. Then the obtained twisted denominators differ only by a sign, and only when $\epsilon=\epsilon^\prime=\frac{1}{2}$, so we keep for them the same notation. Their modular transformation formulae also differ by a sign (see equations (9.16) and (9.17) in [KW]). The character formulae differ (in the twisted case) from those, given by Proposition 
\ref{prop:2.7}, and are as follows (cf Proposition 9.2 from [KW]).

\begin{proposition}
\label{prop:2.10}
\begin{list}{}{}
\item(a) If $\Lambda=\Lambda^{[s]+}_{j,k}$, where $j,k\in \ZZ,\, 0\leq j,k,j+k \leq M-1,$ then
$$
\left(\widehat{R}^{(\epsilon)}_{\,\,\epsilon^\prime} ch^{(\epsilon)}_{\Lambda;\,\epsilon^\prime} \right) (\tau, z_1, z_2, t) = \Psi^{[M,\,m,\,s;\,\epsilon]}_{j+\epsilon^\prime,\, k-\epsilon^\prime;\, \epsilon^\prime} (\tau, z_1,z_2,t).
$$

\item(b) If $\Lambda=\Lambda^{[s]-}_{j,\,k}$, where $j,k \in \ZZ$, $1 \leq j, k, j+k \leq M$, then 
$$\left(\widehat{R}^{(\epsilon)}_{\,\,\epsilon^\prime} 
ch^{(\epsilon)}_{\Lambda;\, \epsilon^{\prime}} \right) (\tau, z_1, z_2, t) = -\Psi^{[M,\,m,\,s ;\,\epsilon]}_{M+\epsilon^\prime-j,\, M-\epsilon^\prime-k;\,\epsilon^\prime} (\tau, z_1, z_2, t).
$$
\end{list}
\hfill $\Box$
\end{proposition}

\section{Modular transformation formulae for modified characters of admissible $N=2$ modules}
\label{sec:3}

Recall (see \cite{KW}, Section 9) 
that the quantum Hamiltonian reduction associates to a principal admissible 
$\widehat{sl}_{2|1}$-
module $L(\Lambda)$ of level $\frac{m+1}{M}-1$, 
where 
$m,M\in \ZZ,\, m \geq 0,\, M\geq 1$, and 
$gcd(M,2m+2)=1$ if $m>0$, a module 
$H(\Lambda)$ over the $N=2$ superconformal algebra with central charge
\begin{equation}
	\label{eq:3.1}
	c=3(1-\dfrac{2m+2}{M}).
\end{equation}

In $[KW]$ we considered the quantum Hamiltonian reduction for the principal
admissible $\widehat{sl}_{2|1}$-
modules $L(\Lambda)$, such that $\Lambda^0=\Lambda_{m;\,s}$ with $s=0$. Here we consider the ease of arbitrary $s\in \ZZ,\, 0\leq s\leq m$, 
so we may assume that $m\geq 1$.

Recall that $H(\Lambda)=0$ iff $(\Lambda|\alpha_0)\in \ZZ_{\geq 0}$ and that 
$H(\Lambda)$ is irreducible otherwise.
If $M=1$, then a principal admissible $\widehat{sl}_{2|1}$-module $L(\Lambda)$ is partially integrable, hence $(\Lambda|\alpha_0)\in\ZZ_{\geq 0}$ and therefore $H(\Lambda)=0$.
If $M>1$, it follows from the formulas for the principal admissible weights $\Lambda=\Lambda^{[s]\pm}_{k_1,k_2}$ in Section 2, 
that $(\Lambda|\alpha_0)\in\ZZ_{\geq 0}$ iff $k_0=0$ (recall that $k_0=M-k_1-k_2\mp 1$). 

Thus, in what follows we may assume that $M\geq 2$ and $k_0>0$. Then $H(\Lambda)$ is an irreducible module over the $NS$ type $N=2$ superconformal algebra. Using the formulas for principal admissible weights in Section 2 and formulas (9.12) and (9.13) from [KW] for the lowest energy $h_{\Lambda}$ and spin $s_{\Lambda}$ of $H(\Lambda)$, we obtain for them the following explicit formulas :
\begin{equation}
		\label{eq:3.2}
		h_{\Lambda^{[s]\pm}_{k_1,k_2}}=\dfrac{m+1}{M} 
\left( ( k_1\pm\dfrac{1}{2}) ( k_2\pm \dfrac{1}{2}) -\dfrac{1}{4}\right)-
s\left( 
k_{(3\mp 1)/2}\pm \dfrac{1}{2} \right),
\end{equation}
\begin{equation}
	\label{eq:3.3}
	s_{\Lambda^{[s]\pm}_{k_1,k_2}}= \pm \dfrac{m+1}{M} (k_2-k_1)-s.
\end{equation}

Recall that the twisted quantum Hamiltonian reduction associates to a 
principal admissible twisted $\widehat{sl}_{2|1}$-module $L^{\tw}(\Lambda)$, where $\Lambda=\Lambda^{[s]\pm}_{k_1,k_2}$ has level $\dfrac{m+1}{M}-1$, a module $H^{\tw}(\Lambda)$ over the Ramond type $N=2$ superconformal algebra with 
central charge (\ref{eq:3.1}). As before, $gcd(M,2m+2)=1$ and we may assume that $m\geq 1,\, M\geq2$; also the module $H^{\tw}(\Lambda)=0$ iff $k_0=0$, and it is irreducible otherwise. Again, it is easy to compute the corresponding characteristic numbers\;:

\begin{eqnarray}
\label{eq:3.4}
	&h^{\tw}_{\Lambda^{[s]+}_{k_1, k_2}} &\!=\; \frac{m+1}{M}\, k_2 (k_1+1)-\frac{m+1}{4M}+\frac{1}{8} - s (k_1+1), \\
	\label{eq:3.5}
	&h^{\tw}_{\Lambda^{[s]-}_{k_1, k_2}} &\!=\; \frac{m+1}{M}\, k_2 (k_1-1) - \frac{m+1}{4M}+\frac{1}{8} - s k_2, \\
	\label{eq:3.6}
&s^{\tw}_{\Lambda^{[s]\pm}_{k_1, k_2}}  &\!=\; \pm \frac{m+1}{M}\, (k_2-k_1) - \frac{m+1}{M} + \frac{1}{2}-s.
\end{eqnarray}

Note that for each $s$ the set of all principal admissible weights 
(with $k_0>0$) of level $\frac{m+1}{M}-1$ is a union of two sets :
$$
\begin{array}{ll}
	\A^+ &= \left\{ \Lambda^{[s]+}_{k_1, k_2}|\, k_1, k_2 \in \ZZ_{\geq 0}, k_1+k_2 \leq M-2 \right\}, \\[2ex]
	\A^{-} &=  \left\{ \Lambda^{[s]-}_{k_1, k_2}|\, k_1, k_2 \in \ZZ_{\geq 1}, k_1+k_2 \leq M \right\}.
\end{array}
$$
Note also that we have the following bijective map :
$$
\nu : \A^+ \rightarrow \A^-,\,\,\, \nu \left(\Lambda^{[s]+}_{k_1, k_2}\right) = \Lambda^{[s]-}_{k_2+1, k_1+1}.
$$
It is immediate to see from (\ref{eq:3.2}), (\ref{eq:3.3}) (resp. (\ref{eq:3.4})-(\ref{eq:3.6})) that 
$$
h_{\Lambda^{[s]+}_{k_1,k_2}} = h_{\nu\left(\Lambda^{[s]+}_{k_1,k_2}\right)}, \;\; s_{\Lambda^{[s]+}_{k_1,k_2}} = s_{\nu\left(\Lambda^{[s]+}_{k_1,k_2}\right)}
\,\,\hbox{for}\,\, \Lambda^{[s]+}_{k_1,k_2} \in \A^+, 
$$
and the same holds for $h^{\tw}$ and $s^{\tw}$. Since the irreducible $N=2$ modules are uniqueley determined by their central charges and the characteristic numbers, the sets $\A^+$ and $\A^-$ correspond to the same set of irreducible $N=2$ modules. Hence, it suffices to consider only the highest weights $\Lambda^{[s]+}_{k_1, k_2} \in \A^+$. 

In order to compute the characters and supercharacters of the corresponding $N=2$ modules $H(\Lambda)$ and $H^{\tw} (\Lambda)$, we use formulae (9.4)  and (9.8) from [KW].

First, we have formula (9.21) from [KW] for the $N=2$ normalized denominators :
$$
\overset{2}{R}^{(\epsilon)}_{\epsilon^\prime} (\tau, z) = 
\frac{ (-1)^{(1-2\epsilon)(1-2\epsilon^\prime)}\eta(\tau)^3}
{\vartheta_{1-2\epsilon^\prime,1-2\epsilon} \;(\tau,z)},
$$

\noindent where $\epsilon, \epsilon^\prime=0$ or $\frac{1}{2}$. 
Here the superscript $\epsilon$ refers to the denominator 
if $\epsilon=\frac{1}{2}$ and to the  superdenominator if $\epsilon=0$, 
and the subscript $\epsilon'$ refers to the Neveu-Schwarz sector if 
$\epsilon^\prime=\frac{1}{2}$, and to the Ramond sector if $\epsilon^\prime=0$.

It is convenient to introduce the following two reindexings of the set 
$\A^{(+)}$ :
$$
\begin{array}{ll}
\A_{NS} &\!=\; \left\{\Lambda_{j,k} = \Lambda^{[s]+}_{j-\frac{1}{2}, k-\frac{1}{2}} |\,\,j,k \in \frac{1}{2} + \ZZ_{\geq 0},\,\, j+k \leq M-1 \right\}, \\[2ex]
\A_{R} &\!=\; \left\{\Lambda_{j,k} = \Lambda^{[s]+}_{j-1, k} |
\,\,j\in \ZZ_{\geq 1},\,k\in \ZZ_{\geq 0},\, j+k\leq M-1 \right\}.
\end{array}
$$
\noindent We let 
$$
H_{NS}(\Lambda_{j,k})=H \left(\Lambda^{[s]+}_{j-\frac{1}{2}, k-\frac{1}{2}}\right),\,
\Lambda_{j,k} \in \A_{NS};\,\, H_{R}(\Lambda_{j,k}) = H \left(\Lambda^{[s]+}_{j-1,k}\right), \,\Lambda_{j,k} \in \A_R.
$$
\noindent It follows from (\ref{eq:3.2})-(\ref{eq:3.6}) that the lowest energy and the spin of these $N=2$ modules (with central charge (\ref{eq:3.1})) 
are as follows :
\begin{equation}
	\label{eq:3.7}
h^{NS}_{jk} = \frac{m+1}{M} jk - \frac{m+1}{4M} -sj,\;\; s^{NS}_{jk} = \frac{m+1}{M} (k-j)-s;	
\end{equation} 
\begin{equation}
	\label{eq:3.8}
h^{R}_{jk} = \frac{m+1}{M}jk - \frac{m+1}{4M}+\frac{1}{8}-sj,\;\;s^{R}_{jk} = \frac{m+1}{M} (k-j) + \frac{1}{2}-s.	
\end{equation}

Introduce the following notation for the characters and supercharacters of these $N=2$ modules :
$$
\begin{array}{ll}
ch^{N=2 [M, m,s;\epsilon]}_{j,k;\frac{1}{2}} (\tau,z) &\!\!=\; ch^{\pm}_{H_{NS}(\Lambda_{j,k})} (\tau,z),\, \Lambda_{j,k} \in \A_{NS}; \\[1.5ex]
ch^{N=2 [M, m,s;\epsilon]}_{j,k;0} (\tau,z) &\!\!=\; ch^{\pm}_{H_{R}
(\Lambda_{j,k})} (\tau,z),\, \Lambda_{j,k} \in \A_{R}.
\end{array}
$$

\noindent Formulae (9.4) and (9.8) from [KW] imply the following expressions for these characters :
\begin{equation}
	\label{eq:3.9}
\left( \overset{2}{R}^{(\epsilon)}_{\epsilon^\prime} 
ch^{N=2 [M, m,s;\epsilon]}_{j,k;\epsilon^\prime}\right) (\tau,z) = 
\Psi^{[M, m,s ; \epsilon]}_{j,k;\epsilon^\prime} (\tau, -z, z, 0),
\end{equation}
where the functions $\Psi^{[M,m,s;\epsilon]}_{j,k;\epsilon} (\tau,z_1,z_2,t)$ are defined by (\ref{eq:2.4}) and $j,k \in \epsilon^\prime + \ZZ_{\geq 0}$, subject to restrictions $j+k \leq M-1, \,j>0$.

Introduce the \emph{modified} $N=2$ characters and supercharacters, letting 
$$
\left(\overset{2}{R}^{(\epsilon)}_{\epsilon^\prime} \overset{\sim}{ch}^{N=2[M,m,s;\epsilon]}_{j,k;\epsilon^\prime}\right) (\tau,z) = \overset{\sim}{\Psi}^{[M, m,s;\epsilon]}_{j,k;\epsilon} (\tau,-z,z,0),
$$

\noindent where the modification $\overset{\sim}{\Psi}$ of $\Psi$ was introduced in Section 2 (before Theorem \ref{th:2.8}). Theorem \ref{th:2.8} along with Lemma 9.3 from [KW] give the following modular transformation properties of the modified $N=2$ characters and supercharacters, in the Neveu-Schwarz and Ramond sectors.

\begin{theorem}
\label{th:2.11}
Let $M \in \ZZ_{\geq 2}$ and let $m \in \ZZ_{\geq 0}$ be such that gcd\,$(M,2m+2)=1$ if $m>0$. Let $c= 3 \left(1-\frac{2m+2}{M}\right)$. 
Let $s\in\ZZ$ be such that $0\leq s\leq m$. 
Let $\epsilon, \epsilon^\prime=0$ or $\frac{1}{2}$, and let 
$$
\Omega^{(M)}_{\epsilon}= \{(j,k) \in (\epsilon +\ZZ_{\geq 0})^2 |\,\, j+k \leq M-1,\,j >0\}.
$$ 
Then we have the following modular transformation formulae for 
$\overset{\sim}{ch}^{N=2[M, m,s\,;\epsilon]}_{j, k;\epsilon^\prime},\, 
(j,k)\in \Omega^{(M)}_{\epsilon^\prime}$:
$$
\overset{\sim}{ch}^{N=2[M, m,s\,;\epsilon]}_{j, k;\epsilon^\prime} \left(-\frac{1}{\tau}, \frac{z}{\tau} \right) = e^{\frac{\pi i c}{6\tau} z^2} \sum\limits_{(a,b) \in \Omega^{(M)}_{\epsilon}}{S^{[M, m, \epsilon, \epsilon^\prime]}_{(j,k),(a,b)}\, 
\overset{\sim}{ch}^{N=2[M, m,s\,;\epsilon^\prime]}_{a, b;\epsilon} (\tau, z)},
$$
\noindent where
$$
S^{[M, m, \epsilon, \epsilon^\prime]}_{(j,k),(a,b)} = (-i)^{(1-2\epsilon)(1-2\epsilon^\prime)} \frac{2}{M} e^{\frac{\pi i (m+1)}{M} (j-k)(a-b)} \sin \frac{m+1}{M} (j+k)(a+b)\pi;
$$
$$
\overset{\sim}{ch}^{N=2[M, m,s;\epsilon]}_{j, k;\epsilon^\prime} (\tau+1,z) = e^{\frac{2\pi i (m+1)}{M} jk- \frac{\pi i \epsilon^\prime}{2}} \overset{\sim}{ch}^{N=2[M, m, s;|\epsilon -\epsilon^\prime|]}_{j, k;\epsilon^\prime} (\tau,z).
$$
\end{theorem}

This theorem coincides with Theorem 9.4 in [KW] in the case $s=0$. For an arbitrary $s \in [0,m]$ the proof is the same, and the modular transformation 
formulae are the same. Note that, unlike in the Lie algebra case, modular transformations do not mix characters with different $s$.

\section{Transformation properties of the mock theta functions ${\bf\Phi^{\bs{[B;m]}}}$.}
\label{sec:4}

In this section we study modular and elliptic transformation properties of supercharacters $ch^-_{L(\Lambda)}$ of partially integrable highest weight modules $L(\Lambda)$ over the affine Lie superalgebra 
$\hat{\fg}$,
where $\fg=osp_{3|2}$. We choose the set of simple roots 
of $\hat{\fg}$ to be $\hat{\Pi}=\{\alpha_0, \alpha_1, \alpha_2\}$, 
where $\alpha_0$ and $\alpha_2$ are even and $\alpha_1$ is odd, and the scalar products are :
\begin{equation}
\label{eq:4.1a}
(\alpha_0|\alpha_0)=2,\;\; (\alpha_1|\alpha_1)=0,\;\; (\alpha_2|\alpha_2)=-\frac{1}{2},\;\; (\alpha_0|\alpha_1)=-1,\;\; (\alpha_1|\alpha_2)=\frac{1}{2},\;\; (\alpha_0|\alpha_2)=0.
\end{equation}
For the underlying finite-dimensional Lie superalgebra $osp_{3|2}$ we have: 
its set of positive even roots is 
$\Delta_{\bz, +}=\{\alpha_2, \theta\}$, where $\theta=2(\alpha_1+\alpha_2)$, 
its set of positive odd roots is
$\Delta_{\bar{1},+}=\{\alpha_1, \alpha_1+\alpha_2, \alpha_1+2\alpha_2\}$, 
hence the highest root is $\theta$, the Weyl vector 
$\rho=-\frac{1}{2}\alpha_1$, and therefore, by ($3.1$) from [KW], the dual 
Coxeter number $h^\vee =\frac{1}{2}$. Furthermore, the Weyl group of $\fg$ 
is  $W=\{1, r_{\alpha_2}, r_\theta, r_{\alpha_2} r_\theta \}$, 
and also $\Delta^{\#}_{\bz}=\{\pm \theta \}$ (see (\ref{eq:3.3}) from [KW]) 
and $L^\# = \ZZ \theta$, since $(\theta|\theta)=2$. 

By the definition of a partially integrable module (see [KW], Definition 3.2), 
an $\widehat{osp}_{3|2}$-module $L(\Lambda)$ is partially integrable iff 
$\Lambda=m \Lambda_0$, where $m$ is a non-negative integer. The numerator $\widehat{R}^- ch^-_{L(m \Lambda_0)}$ of the supercharacter of $L(m \Lambda_0)$ is given by Conjecture 3.8 in [KW], proved in [GK]. 

In order to write down an explicit formula, note that 
$$
\hat{\rho}\, (= h^\vee \Lambda_0 + \rho) = \frac{1}{2} \Lambda_0 -\frac{1}{2} \alpha_1,\; \epsilon^- (t_{n \theta}) = (-1)^n,\; \epsilon^-(r_{\alpha_2})=-1,\; \epsilon^- (r_\theta)=1.
$$
\noindent We choose $T_{0}= \{\alpha_1\}$. 
Then 
we   get the following formula : 
$$
\widehat{R}^- ch^-_{L(m \Lambda_0)}=\sum\limits_{w \in \widehat{W}^\#}{\epsilon^- (w)} w \frac{e^{\left(m+\frac{1}{2}\right) \Lambda_0 - \frac{1}{2} \alpha_1}}{1-e^{-\alpha_1}},
$$
where $\widehat{W}^\# = <1,r_{\theta}> \ltimes t_{L^\#} 
\,\,(\subset \widehat{W})$. Explicitly: 
\begin{equation}
\label{eq:4.0}
 \widehat{R}^- ch^-_{L(m \Lambda_0)}=  e^{(m +\frac{1}{2})\Lambda_0}
\end{equation}
\[\times \sum\limits_{j \in \ZZ}{(-1)^j \left(\frac{e^{-(2m+1) j (\alpha_1+\alpha_2) -\frac{1}{2} \alpha_1}\; q^{\left(m+\frac{1}{2}\right) j^2 + \frac{j}{2}}}{1-e^{-\alpha_1} q^j}
-\frac{e^{-(2m+1) j (\alpha_1+\alpha_2) -\frac{1}{2}(\alpha_1+2\alpha_2)}\; q^{\left(m+\frac{1}{2}\right) j^2 + \frac{j}{2}}}{1-e^{-(\alpha_1 + 2\alpha_2)} q^j} \right).}\]

Introduce the following coordinates in the Cartan subalgebra $\widehat{\fh}$ 
of $\widehat{osp}_{3|2}$ :
\begin{equation}
	\label{eq:4.1}
h=2\pi i (-\tau \Lambda_0 - z_1 (\alpha_1+2\alpha_2) - z_2 \alpha_1 + t \delta) 
:=	(\tau, z_1, z_2, t),
\end{equation}
so that for $z=-z_1(\alpha_1+2\alpha_2) - z_2 \alpha_1$ we have :
\begin{equation}
	\label{eq:4.2}
	(z|z) = 2 z_1 z_2.
\end{equation}
\noindent Let
$$
\Phi^{[B;m]} (\tau, z_1, z_2, t) = \left(\widehat{R}^- ch^-_{L(m \Lambda_0)}\right) (h)
$$
be the  numerator of $ch^-_{L(m\Lambda_0)}=ch^-_{m\Lambda_0}$ in coordinates
(\ref{eq:4.1}). By ( \ref{eq:4.0}) we have:
\begin{equation}
	\label{eq:4.3}
\Phi^{[B;m]} (\tau, z_1, z_2, t) = \Phi^{[B;m]}_1 (\tau, z_1, z_2, t) - \Phi^{[B;m]}_1 (\tau, z_2, z_1, t),	
\end{equation}
where
\begin{equation}
	\label{eq:4.4}
	\Phi^{[B;m]}_1 (\tau, z_1, z_2, t) = e^{2 \pi i \left(m+\frac{1}{2}\right)t} 
	\sum\limits_{j \in \ZZ}{(-1)^j\; \frac{e^{\pi i j (2m+1)(z_1+z_2) + \pi i z_1} q^{\left(m+\frac{1}{2}\right)j^2+\frac{j}{2}}} {1-e^{2\pi i z_1}\; q^j}}.
\end{equation}
Let $\Phi^{[B;m]}_1 (\tau, z_1, z_2) = \Phi^{[B;m]}_1 (\tau, z_1, z_2,0).$

Recall the right action of $A=\left(\begin{array}{ll} a&b\\c &d \end{array}\right) \in \text{SL}_2 (\RR)$ on the space of meromorphic functions in the domain 
$X=\{h \in \widehat{\fh}^\ast | \;\text{Re}\; \delta (h) > 0$ (or Im $\tau  > 0$) \} for $\ell=2\,(=$ rank $osp_{3|2})$ :
\begin{equation}
	\label{eq:4.5}
	F|_A (\tau, z, t) = (c\tau +d)^{-1} F\left(\frac{a\tau + b}{c\tau+d}\,, \frac{z}{c\tau+d}\,, t-\frac{c(z|z)}{2(c \tau+d)} \right).
\end{equation}
For $S=\left( \begin{array}{lr} 0 & -1 \\ 1 & 0 \end{array}\right)$ we have,
by (\ref{eq:4.2}):
\begin{equation}
	\label{eq:4.6}
F|_S\; (\tau, z_1, z_2, t) = \tau^{-1}\; 
F\left(-\frac{1}{\tau}, \frac{z_1}{\tau}, \frac{z_2}{\tau}, t-\frac{z_1z_2}{\tau}\right).
\end{equation}
%

Note that the meromorphic function $\Phi^{[B;m]}_1 (\tau,z_1,z_2)$, 
viewed as a function 
of $z_1$, 
has poles at $z_1
 \in \ZZ +\tau \ZZ$, all of them are simple and have residues : 
$$
\text{Res}_{z_1= n + j \tau} \Phi^{[B;m]}_1 (\tau, z_1, z_2) = \frac{(-1)^{(j+1)(n+1)}}{2\pi i} e^{-\pi i j (2m+1)z_2}. 
$$
The function $(\Phi^{[B;m]}_1|_S) (\tau, z_1, z_2, t=0)$ has the same poles, all of them are simple, and have the same residues at these poles. Hence for the functions $\Phi^{[B;m]}$ and $\Phi^{[B;m]}_1$  we have an analogue of
Lemma \ref{lem:1.1}(0):

\begin{lemma}
\label{lem:4.1}
The functions $\Phi^{[B;m]}_1 - \Phi^{[B;m]}_1|_S$ and $\Phi^{[B;m]} - 
\Phi^{[B;m]}|_S$ are holomorphic in the domain $X$.
\hfill $\Box$
\end{lemma}

\begin{lemma}
\label{lem:4.2}
The functions $\Phi^{[B;m]}$ (and $\Phi^{[B;m]}_1$) satisfy the following properties : 
\begin{list}{}{}
\item (1) $\Phi^{[B;m]} (\tau, -z_1, -z_2,t) = - \Phi^{[B;m]} (\tau, z_1, z_2, t)$.

\item(2) $\Phi^{[B;m]} (\tau, z_1+a, z_2 +b, t) = (-1)^a\; \Phi^{[B;m]} (\tau, z_1, z_2, t)$ if $a, b \in \ZZ$ have the same parity.

\item(3) $\Phi^{[B;m]} (\tau, z_1+j\tau, z_2+j\tau, t) = (-1)^j \;e^{-\pi i j (2m+1)(z_1+z_2)}\;\, q^{-j^2 \left(m+\frac{1}{2}\right)}\;\, \Phi^{[B;m]} (\tau, z_1, z_2, t)$. 
\end{list}
\end{lemma}

\begin{proof}
The proof of (1) is straightforward by changing $j$ to $-j$ in the RHS of (\ref{eq:4.4}). The proof of (2) and (3) is the same as that of Lemma \ref{lem:1.1} (2) and (5). 
\end{proof}

As in Section \ref{sec:1}, we change the coordinates, letting 	
$$
z_1=u+v,\;\; z_2=u-v,\;\;\; \text{i.e.}\;\;\; u=\frac{z_1+z_2}{2},\;\; v=\frac{z_1-z_2}{2},
$$
and denote 
$\varphi^{[B;m]} (\tau, u, v, t) = \Phi^{[B;m]} (\tau, z_1, z_2, t)$, and 
similarly for $\varphi^{[B;m]}_1$ and $\Phi^{[B;m]}_1$.

\vspace{1ex}
\noindent Then we get :

$$
\varphi^{[B;m]} (\tau, u, v, t) = \varphi^{[B;m]}_1 (\tau, u, v, t) - \varphi^{[B;m]}_1 (\tau, u, -v, t),
$$
\noindent where 
\begin{equation}
\label{eq:4.7}
\varphi^{[B;m]}_1 (\tau, u, v, t) = e^{2\pi i \left(m+\frac{1}{2}\right) t} 
\sum\limits_{j \in \ZZ}{(-1)^j\;\;  \frac{e^{2\pi i j (2m+1) u + \pi i (u+v)}\; q^{\left(m+\frac{1}{2}\right)j^2 +\frac{j}{2}}} {1-e^{2\pi i (u+v)} q^j}}.
\end{equation}
\indent In order to state an analogue of Lemma \ref{lem:1.2} (3), (4), we shall need the following ``alternate'' analogues of the theta functions $\Theta_{j,m}$, defined by (\ref{eq:1.2}) : 
\begin{equation}
	\label{eq:4.8}
	\Theta^-_{j,m} (\tau, z, t) = e^{2\pi i m t} \sum\limits_{n\in \ZZ}{(-1)^n e^{2\pi i m z \left(n+\frac{j}{2m}\right)} q^{m \left(n+\frac{j}{2m}
\right)^2}}.
\end{equation} 
Here $m \in \frac{1}{4} \ZZ_{\geq 1}$ is the degree and $j\in \frac{1}{2} \ZZ$. These are holomorphic functions in the domain $X_0$. As before, let $\Theta^-_{j,m} (\tau, z) = \Theta^-_{j,m} (\tau, z, 0)$. We obviously have :
\begin{equation}
	\label{eq:4.9}
\Theta^-_{j+2ma,m} (\tau,z) = (-1)^a \;\;\Theta^-_{j,m} (\tau, z) \;\text{if}\; a \in \ZZ.	
\end{equation} 
In particular, $\Theta^-_{j,m}$ depends only on $j \;\text{mod}\; 4m$. 

The functions $\Theta^-_{j,m} (\tau, z)$ satisfy the following modular transformation properties.

\begin{proposition}
\label{prop:4.3}
Let $m \in \frac{1}{4} \ZZ_{\geq 1}$. Then we have 
\begin{list} {}{}
\item (1) $\Theta^-_{j,m} \left(-\frac{1}{\tau}, \frac{z}{\tau}\right) = 
\left(\frac{-i\tau}{2m}\right)^{\frac{1}{2}} e^{\frac{\pi i m z^2}{2\tau}} \sum\limits^{4m-1}_{k = 1 \atop k\;\text{odd}}{e^{\frac{-\pi i j k}{2 m}} \Theta^-_{\frac{k}{2},m} (\tau,z)}$, provided that $j\in \frac{1}{2}+\ZZ$.

\item (2) If $j\in \ZZ$, then the formula in (1) holds if we replace $\Theta^-_{\frac{k}{2}, m}$ in the RHS by $\Theta_{\frac{k}{2},m}$. 

\item (3) If $j \in \frac{1}{2} + \ZZ$, then the formula in (1) holds if we replace $\Theta^-_{j,m} (\text{resp} \; \Theta^-_{\frac{k}{2},m})$ in the LHS (resp. RHS) by $\Theta^-_{j,m} +  \Theta^-_{-j,m}$ $(\text{resp. by}\; 
\Theta^-_{\frac{k}{2},m} + \Theta^-_{-\frac{k}{2},m})$.

\item (4) If $m\in \ZZ_{\geq 0}$ and $j\in \ZZ$, then
$$
\Theta^-_{\pm(j+\frac{1}{2}), m+\frac{1}{2}} (\tau+1,z) = 
e^{\frac{\pi i}{2m+1}(j+\frac{1}{2})^2}
\Theta^-_{\pm(j+\frac{1}{2}), m+\frac{1}{2}} (\tau,z). 
$$ 
\end{list}
\end{proposition}
\begin{proof} 
Formulae (1)-(3) follow from (A.5) in the Appendix of \cite{KW}, using
the obvious identity
$$ 
\Theta^\pm_{j,m}(\tau,z)= \Theta_{2j,4m}(\tau,\frac{z}{2})\pm \Theta_{2j-4m,4m}
(\tau, \frac{z}{2}). 
$$
Formula (4) is a special case of Proposition \ref{prop:A.3} from the 
Appendix to the present paper.
\end{proof}

\begin{lemma}
\label{lem:4.4}
The functions $\varphi^{[B;m]}_1$ and $\varphi^{[B;m]}$ satisfy the following properties :
\begin{list}{}{}
\item(1) $\varphi^{[B;m]}_1 (\tau, u+a, v+b, t) = (-1)^{a+b} \varphi^{[B;m]}_1 (\tau, u, v, t)$ if $a, b \in \ZZ$, and the same holds for the functions
$\varphi^{[B;m]}$.

\item(2) $\varphi^{[B;m]}_1(\tau, u,v, t)+e^{2\pi i \left(m+\frac{1}{2}\right) (2v-\tau)} \varphi^{[B;m]}_1 (\tau, u, v-\tau, t)$

$=e^{2\pi i \left(m+\frac{1}{2}\right) t} \sum\limits^{2m}_{k=0}{e^{2\pi i \left(k+\frac{1}{2}\right) v} \; q^{-\frac{\left(k+\frac{1}{2}\right)^2} {2(2m+1)}}  \;\;\Theta^-_{k+\frac{1}{2}, m+\frac{1}{2}} (\tau,2u)}$.

\item(3) $ \varphi^{[B;m]} (\tau,u,v,t) + e^{2\pi i \left(m+\frac{1}{2}\right)(2v-\tau)}  \varphi^{[B;m]} (\tau,u,v-\tau,t)$

$=e^{2\pi i \left(m+\frac{1}{2}\right) t} \sum\limits^{2m}_{k=0}
{e^{2\pi i \left(k+\frac{1}{2}\right) v}\; q^{-\frac{\left(k+\frac{1}{2}\right)^2} {2(2m+1)}}  
(\Theta^-_{k+\frac{1}{2}, m+\frac{1}{2}} + \Theta^-_{-\left(k+\frac{1}{2}\right), m+\frac{1}{2}})(\tau, 2u)}$.
\end{list}
\end{lemma}

\begin{proof}
\label{}
Claim (1) is obvious, and claim (3)  follows from claim (2), using Lemma \ref{lem:4.2} (1) and (\ref{lem:4.9}). The proof of (2) is the same as that of Lemma \ref{lem:1.1} (3).  Namely, taking the shift $j \rightarrow  j+1$ in the RHS of (\ref{eq:4.7}), and assuming without loss of generality, that $t=0$, we have :
$$
\begin{array}{ll}
&\varphi^{[B;m]}_1 (\tau, u, v-\tau, 0) = -\sum\limits_{j \in \ZZ}{(-1)^j\; \frac{e^{2\pi i (j+1)  (2m+1)u + \pi i (u+v)}}{1- e^{2 \pi i (u+v)} q^j} q^{\left(m+\frac{1}{2}\right) (j+1)^2 +\frac{j}{2}}}\\
&=-e^{-2 \pi i (2m+1)v} q^{m+\frac{1}{2}} \sum\limits_{j \in \ZZ}
{(-1)^j \;\frac{e^{2\pi i j (2m+1)u + \pi i (u+v)} q^{\left(m+\frac{1}{2}\right)j^2+\frac{j}{2}}} {1-e^{2\pi i(u+v)} q^j}} \left(e^{2\pi i (u+v)} q^j \right)^{2m+1}.
\end{array}
$$
\noindent Multiplying both sides by $e^{2 \pi i (2m+1)v} q^{-\left(m+\frac{1}{2}\right)}$, we obtain :
\begin{eqnarray}
	\label{eq:4.10}
&&e^{2\pi i (2m+1)v} q^{-\left(m + \frac{1}{2}\right)} \varphi^{[B;m]}_1 (\tau,u, v-\tau,0) \\
&& = -\sum\limits_{j\in \ZZ}{(-1)^j \frac{e^{2\pi i j (2m+1)u + \pi i (u+v)} q^{\left(m+\frac{1}{2}\right)j^2+\frac{j}{2}}} {1-e^{2\pi i(u+v)} q^j} 
\left(e^{2\pi i (u+v)} q^j \right)^{2m+1}}. \nonumber
\end{eqnarray}
Adding (\ref{eq:4.8}) and (\ref{eq:4.10}), we obtain : 
$$
\begin{array}{ll}
&\varphi^{[B;m]}_1 (\tau, u, v, 0) + e^{2\pi i (2m+1)v}\; q^{-\left(m+\frac{1}{2}\right)} \;\varphi^{[B;m]}_1 (\tau,u,v-\tau, 0) \\[1ex]
&=\displaystyle\sum\limits_{j \in \ZZ}{(-1)^j {e^{2\pi i j (2m+1)u + \pi i (u+v)}}\; q^{\left(m+\frac{1}{2}\right)j^2 +\; \frac{j}{2}}\;\, \frac{1-(e^{2\pi i (u+v)} q^j)^{2m+1}} {1-e^{2 \pi i (u+v)} q^j}} \\[1ex]
&=\displaystyle\sum\limits^{2m}_{k=0}{e^{2 \pi i \left(k+\frac{1}{2}\right)v} \;q^{-\frac{\left(k+\frac{1}{2}\right)^2} {2(2m+1)}}} 
\sum\limits_{j \in \ZZ}{(-1)^j  e^{2 \pi i (2m+1) \left( j + \frac{k+\frac{1}{2}}{2m+1}\right)u} q^{\left(m+\frac{1}{2}\right) \left(j+\frac{k+\frac{1}{2}}{2m+1}\right)^2}}\\[1ex]
&=\displaystyle\sum\limits^{2m}_{k=0} {e^{2\pi i \left(k+\frac{1}{2}\right)v}\; q^{-\frac{\left(k+\frac{1}{2}\right)^2}{2(2m+1)}} \;\, \Theta^-_{k+\frac{1}{2},m+\frac{1}{2}} (\tau,2u)},
\end{array}
$$
proving claim (2).
\end{proof}

Next, in analogy with Section \ref{sec:1}, introduce the following function :
$$
G^{[B;m]} (\tau,u,v,t) := \varphi^{[B;m]} (\tau, u, v, t) - \varphi^{[B;m]}|_S (\tau, u, v, t),
$$
\noindent where (cf. (\ref{eq:4.6})) :
$$
\varphi^{[B;m]}|_S (\tau, u, v, t) = \frac{1}{\tau} \varphi^{[B;m]} \left(-\frac{1}{\tau}, \frac{u}{\tau}, \frac{v}{\tau}, t - \frac{u^2-v^2}{\tau}\right).
$$

\begin{lemma}
\label{lem:4.5}
The function $G^{[B;m]}$  satisfies the following three properties (1), (2), (3) :
\begin{list} {}{}
\item (1) $G^{[B;m]}$ is a holomorphic function in the domain $X$. 

\item(2)  $G^{[B;m]} (\tau, u, v, t) + e^{\pi i (2m+1)(2v-\tau)} G(\tau,u,v-\tau,t)$

$=e^{2\pi i \left(m+\frac{1}{2}\right) t} \sum\limits^{2m}_{k=0}{e^{2\pi i \left(k+\frac{1}{2}\right) v} q^{-\frac{\left(k+\frac{1}{2}\right)^2}{2(2m+1)}} \left(\Theta^-_{k+\frac{1}{2},\,m+\frac{1}{2}} + \Theta^-_{-\left(k+\frac{1}{2}\right),\,m+\frac{1}{2}}\right) (\tau,2u)}.
$

\item(3) $G^{[B;m]} (\tau,u,v+1,t) + G^{[B;m]} (\tau,u,v,t) 
= \frac{i}{\sqrt{2m+1}} (-i \tau)^{-\frac{1}{2}} 
e^{2\pi i \left(m+\frac{1}{2}\right) t}$
 
$\times \sum\limits^{2m}_{k=0}{e^{\frac{\pi i (2m+1)}{\tau} \left(v+\frac{k+\frac{1}{2}}{2m+1}\right)^2}} 
\displaystyle\sum\limits_{1 \leq j \leq 4m+1 \atop j\,
\text{odd}}  
{e^{-\frac{\pi i j}{2m+1} \left(k+\frac{1}{2}\right)} \left(\Theta^-_{\frac{j}{2},\,m+\frac{1}{2}} + \Theta^-_{-\frac{j}{2},m+\frac{1}{2}}\right) (\tau,2u)}.
$

\item (4) The function $G^{[B;m]} (\tau, u, v, t)$ is uniquely determined by the above three properties (1), (2), (3).
\end{list}
\end{lemma}
\begin{proof}
\label{}
Property (1) of $G^{[B;m]}$ is immediate by Lemma \ref{lem:4.1}. The proof of property (2) is straightforward (as that of the analogous Lemma \ref{lem:1.4} (2)).

The proof of property (3) follows the same lines as that of Lemma \ref{lem:1.4} (3), using Proposition \ref{prop:4.3} (3) and Lemma \ref{lem:4.4} (1).

The proof of the uniqueness of the function, satisfying properties (1), (2), (3) is the same as that of Lemma \ref{lem:1.4} (4).
\end{proof}

We keep proceeding in the same way as in Section \ref{sec:1}. The omitted proofs are the same. 

\begin{lemma}
\label{lem:4.6}
Let $a_j (\tau, v)\,\, (j \in \ZZ,\, 0 \leq j \leq 2m)$ be functions 
satisfying the following conditions (i), (ii), (iii) :
\begin{list}{}{}
\item (i) $a_j (\tau, v)$ is holomorphic with respect to $v$, 

\item(ii) $a_j (\tau, v) + e^{\pi i (2m+1)(2v-\tau)} a_j (\tau, v-\tau) = e^{2\pi i \left(j +\frac{1}{2} \right) v} q^{-\frac{1}{2(2m+1)} \left(j+\frac{1}{2}\right)^2}$

\item (iii) $a_j (\tau,v) + a_j (\tau,v+1) = \frac{i}{\sqrt{2m+1}} (-i \tau)^{-\frac{1}{2}} \displaystyle \sum\limits^{2m}_{k=0} 
e^{-\frac{2 \pi i}{2m+1}\left(j+\frac{1}{2}\right)\left(k+\frac{1}{2}\right)} 
e^{\frac{2 \pi i \left(m+\frac{1}{2}\right)} {\tau} \left(v+\frac{k+\frac{1}{2}}{2m+1}\right)^2}$
\end{list} 
\noindent ($a_j (\tau,v)$ is uniquely determined by these properties).

\noindent Then the function 
\begin{equation}
	\label{eq:4.11}
G^{[B;m]} (\tau,u,v,t) := 	e^{2\pi i \left(m+\frac{1}{2}\right)t} \sum\limits^{2m}_{j=0}{a_j(\tau,v) \left(\Theta^{-}_{j+\frac{1}{2},m+\frac{1}{2}} + \Theta^{-}_{-\left(j+\frac{1}{2}\right),m+\frac{1}{2}}\right) (\tau,2u)}
\end{equation}
satisfies the properties (1), (2), (3) of Lemma \ref{lem:4.5}.
\end{lemma}

Now we introduce the following functions ($j \in  \ZZ,\,\,m\in \ZZ_{\geq 0}$):

\bigskip
\noindent $ R^{[B]}_{j+\frac{1}{2}, m+\frac{1}{2}} (\tau,v):=  $
\[ \sum\limits_{n\in \ZZ+\frac{1}{2} \atop n \equiv j +\frac{1}{2} \;\text{mod}\; 2m+1} 
\hspace{-5ex}(-1)^{n-\left(j+\frac{1}{2}\right)} \left\{\text{sgn} (n) - E \left( \left(n+(2m+1) \frac{\text{Im}\; (v)}{\text{Im}\;(\tau)}\right) \sqrt{\frac{\text{Im}\; (\tau)}{m+\frac{1}{2}}}\;\right)\right\} e^{-\frac{\pi i n^2 \tau}{2m+1}-2\pi i n v}. \]

\begin{lemma}
\label{lem:4.7}
The function 
$R^{[B]}_{j+\frac{1}{2}, m+\frac{1}{2}}$ 
has the following properties :
\begin{list}{}{}
\item (1) $R^{[B]}_{j+\frac{1}{2}, m+\frac{1}{2}} (\tau, v+1) = -R^{[B]}_{j+\frac{1}{2}, m+\frac{1}{2}} (\tau,v)$.

\item(2) For $0 \leq j \leq 2m$ we have :

$
\begin{array}{l@{\quad}l}
&R^{[B]}_{j+\frac{1}{2}, m+\frac{1}{2}} (\tau, v-\tau) =- e^{\pi i (2m+1)(\tau-2v)} R^{[B]}_{j+\frac{1}{2}, m+\frac{1}{2}} (\tau,v) \\[2ex]
&+ 2e^{-\frac{\pi i \tau}{2m+1}\left(j+\frac{1}{2}\right)^2 + 2\pi i  \tau \left(j+\frac{1}{2}\right)} e^{-2\pi i \left(j+\frac{1}{2}\right) v}.
\end{array}
$

\item(3) $R^{[B]}_{2m - j +\frac{1}{2}, m+\frac{1}{2}} (\tau,v) = 
R^{[B]}_{j +\frac{1}{2}, m+\frac{1}{2}} (\tau,-v)$.

\item(4) $R^{[B]}_{j+\frac{1}{2}, m+\frac{1}{2}}(\tau+1,v)=
e^{\frac{-\pi i}{2m+1}(j+\frac{1}{2})^2}R^{[B]}_{j+\frac{1}{2}, m+\frac{1}{2}}(\tau,v).$ 
\hfill $\Box$
\end{list}
\end{lemma}

\noindent Letting $(\tau,v) \rightarrow \left(-\frac{1}{\tau}, \frac{v}{\tau}\right)$ in Lemma \ref{lem:4.7} (1), (2), we obtain the following :

\begin{lemma}
\label{lem:4.8}
$$$$
\vspace{-5ex}
\begin{list}{}{}
\item (1) $R^{[B]}_{j +\frac{1}{2}, m+\frac{1}{2}} \left(-\frac{1}{\tau}, \frac{v-\tau}{\tau}\right) = -R^{[B]}_{j +\frac{1}{2}, m+\frac{1}{2}} \left(-\frac{1}{\tau}, \frac{v}{\tau}\right)$.

\item (2) For $0 \leq j \leq 2m$, we have :

$
\begin{array}{ll}
&R^{[B]}_{j +\frac{1}{2}, m+\frac{1}{2}} \left(-\frac{1}{\tau}, \frac{v+1}{\tau}\right) = -e^{-\frac{\pi i (2m+1)}{\tau} (2v+1)} 
R^{[B]}_{j +\frac{1}{2}, m+\frac{1}{2}} \left(-\frac{1}{\tau}, \frac{v}{\tau}\right) \\[2ex] 
&+ 2e^{\frac{\pi i}{(2m+1)\tau}\left(j+\frac{1}{2}\right)^2 - \frac{2\pi i}{\tau} \left(j+\frac{1}{2}\right)} 
e^{-\frac{2\pi i}{\tau}\left(j+\frac{1}{2}\right)v}.
\end{array}
$
\end{list}
\hfill $\Box$
\end{lemma}

Let $(a,b)$ and $y$ be the real coordinates defined by (\ref{eq:1.8a}). Then we have the following analogue of Lemma \ref{lem:1.7} and Lemma \ref{lem:1.8}.
\begin{lemma}
\label{lem:4.9}
$$$$

\vspace{-5ex}

\begin{list}{}{}
\item (1) $(\frac{\partial}{\partial a}+\tau \frac{\partial}{\partial b}) R^{[B]}_{j+\frac{1}{2},m+\frac{1}{2}}(\tau,v)
=-\sqrt{2(2m+1)y}\; e^{-2\pi(2m+1)a^2y}\; \Theta^{-}_{j+\frac{1}{2},m+\frac{1}{2}}(-\bar{\tau},-2\bar{v})$.

\item (2) $(\frac{\partial}{\partial a}+\tau \frac{\partial}{\partial b}) R^{[B]}_{j+\frac{1}{2},m+\frac{1}{2}} (-\frac{1}{\tau},\frac{v}{\tau})=-\frac{\tau}{|\tau|}\sqrt{2(2m+1)y}\; e^{-2\pi(2m+1)\frac{b^2y}{|\tau|^2}}\; 
\Theta^{-}_{j+\frac{1}{2},m+\frac{1}{2}}(\frac{1}{\bar{\tau}},-\frac{2\bar{v}}{\bar{\tau}}).$

\item (3) $(\frac{\partial}{\partial a}+\tau \frac{\partial}{\partial b}) R^{[B]}_{j+\frac{1}{2},m+\frac{1}{2}}(-\frac{1}{\tau},\frac{v}{\tau})=-i(-i\tau)^{\frac{1}{2}} \sqrt{2y}\;e^{\frac{\pi(2m+1)}{2y\tau}(\bar{\tau}v^2-2\tau v\bar{v}+\tau \bar{v}^2)}$

$$ \times 
\sum\limits^{2m}_{k=0} 
e^{-\frac{2\pi i}{2m+1}(j+\frac{1}{2})(k+\frac{1}{2})} \Theta^{-}_{k+\frac{1}{2},m+\frac{1}{2}}(-\bar{\tau},2\bar{v}).$$
\end{list}
\hfill $\Box$
\end{lemma}

Let 

\bigskip
\noindent $ \overset{\sim}{a}_{j}(\tau,v)\!=  $

\[ \!R^{[B]}_{j+\frac{1}{2},m+\frac{1}{2}}(\tau,v)+ \frac{i}{\sqrt{2m+1}}(-i\tau)^{-\frac{1}{2}} e^{\frac{2\pi i(m+\frac{1}{2})}{\tau}v^2} \sum\limits^{2m}_{k=0} e^{-\frac{2\pi i}{2m+1}(j+\frac{1}{2})(k+\frac{1}{2})} R^{[B]}_{k+\frac{1}{2},m+\frac{1}{2}}\left(\!-\frac{1}{\tau},\frac{v}{\tau}\!\right). \]
Then we have the following analogue of Lemma \ref{lem:1.9}.

\begin{lemma}
\label{lem:4.10}
The functions $\overset{\sim}{a}_{j}(\tau,v),\,\, 0\leq j\leq 2m$, satisfy the following properties :
\begin{list}{}{}
\item (1) $\overset{\sim}{a}_{j}(\tau,v)$ is holomorphic with respect to $v$,

\item (2) $\overset{\sim}{a}_{j}(\tau,v)+e^{\pi i(2v-\tau)} \overset{\sim}{a}_{j}(\tau,v-\tau)=2e^{-\frac{\pi i \tau}{2m+1}(2m-j+\frac{1}{2})^2} 
e^{2\pi i(2m-j+\frac{1}{2})v}$,

\item (3) $\overset{\sim}{a}_{j}(\tau,v)+\overset{\sim}{a}_{j}(\tau,v+1)=\frac{2i}{\sqrt{2m+1}}(-i\tau)^{-\frac{1}{2}} \sum\limits^{2m}_{k=0}e^{-\frac{2\pi i}{2m+1}(2m-j+\frac{1}{2})(k+\frac{1}{2})}e^{\frac{\pi i(2m+1)}{\tau}(v+\frac{k+\frac{1}{2}}{2m+1})^2}$.
\end{list}
\hfill $\Box$
\end{lemma}

Lemma \ref{lem:4.10} shows that the functions $a_j(\tau,v):=\frac{1}{2}\overset{\sim}{a}_{2m-j}(\tau,v)\,\, (0\leq j \leq 2m)$ satisfy the conditions $(i),(ii),(iii)$  of Lemma \ref{lem:4.6}. Hence, by Lemma \ref{lem:4.6} and Lemma \ref{lem:4.5}(4), we have the following analogue of Proposition \ref{prop:1.10}

\begin{proposition}
\label{prop:4.11}
$$
G^{[B;m]}(\tau,u,v,t)=\frac{1}{2}e^{2\pi i(m+\frac{1}{2})t} \sum\limits^{2m}_{j=0} 	\overset{\sim}{a}_{2m-j}(\tau,v) (\Theta^{-}_{j+\frac{1}{2},m+\frac{1}{2}} + \Theta^{-}_{-(j+\frac{1}{2}),m+\frac{1}{2}}) (\tau,2u).
$$
\hfill $\Box$
\end{proposition}

Finally, let (cf. (1.13) and (1.14)) :
\begin{equation}
		\label{eq:4.12}
		\varphi^{[B;m]}_{add}(\tau,u,v,t):=
\frac{1}{2}e^{2\pi i(m+\frac{1}{2})t}\sum\limits^{2m}_{j=0} 
R^{[B]}_{j+\frac{1}{2},m+\frac{1}{2}}(\tau,v)
(\Theta^{-}_{j+\frac{1}{2},m+\frac{1}{2}}+\Theta^{-}_{-(j+\frac{1}{2}),m+\frac{1}{2}})(\tau,2u)
\end{equation}
and
\begin{equation}
		\label{eq:4.13}
		\overset{\sim}{\varphi}^{[B;m]}:=\varphi^{[B;m]} + \varphi^{[B;m]}_{add}.
\end{equation}

Then we have, as in the proof of Theorem \ref{th:1.11}:
\begin{lemma}
\label{lem:4.12}
$$$$
\vspace{-7ex}
\begin{list}{}{}
\item (1) $\varphi^{[B;m]}_{add}-\varphi^{[B;m]}_{add}|_{S}=-G^{[B;m]};\,\,
\varphi^{[B;m]}_{add}(\tau+1,u,v,t)=\varphi^{[B;m]}_{add}(\tau,u,v,t)$.

\item (2) $\varphi^{[B;m]}_{add}(\tau,u+a,v+b,t)=(-1)^{a+b}\varphi^{[B;m]}_{add}(\tau,u,v,t)$ if $a,b\in \ZZ$.

\item (3) $\varphi^{[B;m]}_{add}(\tau,u,-v,t)=-\varphi^{[B;m]}_{add}(\tau,u,v,t);\,\, \varphi^{[B;m]}_{add}(\tau,-u,v,t)=\varphi^{[B;m]}_{add}(\tau,u,v,t).$
\end{list}
\hfill
$\Box$
\end{lemma}

Thus we obtain the following analogue of Theorem \ref{th:1.11}.
\begin{theorem}
\label{th:4.13}
$$$$
\vspace{-7ex}
\begin{list}{}{}
\item (1) $\overset{\sim}{\varphi}^{[B;m]}|_{S}=\overset{\sim}{\varphi}^{[B;m]}$, 
i.e. $\overset{\sim}{\varphi}^{[B;m]}(-\frac{1}{\tau},\frac{u}{\tau},\frac{v}{\tau},t+\frac{v^2-u^2}{\tau})=\tau \overset{\sim}{\varphi}^{[B;m]} (\tau, u, v,t).$

\item (2) $\overset{\sim}{\varphi}^{[B;m]}(\tau+1,u,v,t)=
{\overset{\sim}\varphi}^{[B;m]}(\tau,u,v,t).$ 
\item (3) For $a,b\in \ZZ$, we have\,:

$
\overset{\sim}{\varphi}^{[B;m]}(\tau,u+a,v+b,t)=(-1)^{a+b}\overset{\sim}{\varphi}^{[B;m]}(\tau,u,v,t);$

$\overset{\sim}{\varphi}^{[B;m]}(\tau,u+a\tau,v+b\tau,t)=(-1)^{a+b}e^{2\pi i(2m+1)(bv-au)}q^{(m+\frac{1}{2})(b^2-a^2)}\overset{\sim}{\varphi}^{[B;m]}(\tau,u,v,t).$

\item (4) $\overset{\sim}{\varphi}^{[B;m]}(\tau,u,-v,t)=-\overset{\sim}{\varphi}^{[B;m]}(\tau,u,v,t), \,\,
\overset{\sim}{\varphi}^{[B;m]}(\tau,-u,v,t)=\overset{\sim}{\varphi}^{[B;m]}(\tau,u,v,t)$.

\hfill
$\Box$

\end{list}

\end{theorem}

Translating the formulas for $\varphi^{[B;m]}$ and its modification to $\Phi^{[B;m]}$ and its modification, we obtain :
$$
\Phi^{[B;m]}_{add}(\tau,z_1,z_2,t)= \frac{1}{2}e^{2\pi i(m+\frac{1}{2})t} \sum\limits_{j=0}^{2m} R^{[B]}_{j+\frac{1}{2},m+\frac{1}{2}}
(\tau,\frac{z_1-z_2}{2})
(\Theta^-_{j+\frac{1}{2},m+\frac{1}{2}}+\Theta^-_{-(j+\frac{1}{2}),m+\frac{1}{2}}) (\tau,z_1+z_2)
$$
and
$$
\overset{\sim}{\Phi}^{[B;m]} := \Phi^{[B;m]} + \Phi^{[B;m]}_{add},
$$
where $\Phi^{[B;m]}_{add}$ is obtained from $\varphi^{[B;m]}_{add}$ by the change of variables $u=\frac{z_1+z_2}{2}, \,\,
v=\frac{z_1-z_2}{2}$. Then Theorem \ref{th:4.13} implies
\begin{theorem}
\label{th:4.14}
The function $\tilde{\Phi}^{[B;m]}$ ($m\in \ZZ_{\geq 0}$) 
has the following modular and elliptic transformation properties : 
\begin{list}{}{}
\item (1) $\overset{\sim}{\Phi}^{[B;m]}|_S=\overset{\sim}{\Phi}^{[B;m]}$, i.e. $\overset{\sim}{\Phi}^{[B;m]}(-\frac{1}{\tau},\frac{z_1}{\tau},\frac{z_2}{\tau},t-\frac{z_1z_2}{\tau})=\tau \overset{\sim}{\Phi}^{[B;m]} (\tau,z_1,z_2,t)$.

\item (2) $\overset{\sim}{\Phi}^{[B;m]}(\tau+1,z_1,z_2,t)=\overset{\sim}{\Phi}^{[B;m]} (\tau,z_1,z_2,t)$.

\item (3) For $a,b\in \ZZ$ such that $a+b\in 2\ZZ$, we have : 

$\overset{\sim}{\Phi}^{[B;m]} (\tau,z_1+a,z_2+b,t)=(-1)^a \overset{\sim}{\Phi}^{[B;m]} (\tau,z_1,z_2,t)$; 

$\overset{\sim}{\Phi}^{[B;m]} (\tau ,z_1+a\tau,z_2+b\tau,t)=(-1)^a e^{-2\pi i(m+\frac{1}{2})(az_2+bz_1)} q^{-(m+\frac{1}{2})ab} \overset{\sim}{\Phi}^{[B;m]} (\tau,z_1,z_2,t)$.

\item (4) $\overset{\sim}{\Phi}^{[B;m]} (\tau,z_2,z_1,t)=-
\overset{\sim}{\Phi}^{[B;m]} 
(\tau,z_1,z_2,t)$; $\,\,
\overset{\sim}{\Phi}^{[B;m]} (\tau,-z_1,-z_2,t)= - \overset{\sim}{\Phi}^{[B;m]} (\tau,z_1,z_2,t)$

\hfill $\Box$

\end{list}
\end{theorem}

\begin{remark}   
\label{rem:4.15}
It follows from (\ref{eq:4.9}) that $\Phi_{add}^{[B;0]}=0$, hence 
$\overset{\sim}{\Phi}^{[B;0]}=\Phi^{[B;0]}$. 
\end{remark}

\section{
Modular transformation formulae for modified normalized characters of admissible $\widehat{osp}_{3|2}$-modules.}
\label{sec:5}

Throughout this section $\fg=osp_{3|2}$. Let $m\in \ZZ_{\geq 0}$ and $M\in \ZZ_{\geq 1}$ be such that $\gcd(M,4m+2)=1$. Recall that $h^{\vee}=\frac{1}{2}$ and $\widehat{\rho}=\frac{1}{2}\Lambda_0-\frac{1}{2} \alpha_1$.
Hence, by (\ref{eq:2.1}), all principal admissible weights with respect to a simple subset $S=t_\beta y(S_{M})$, associated to the pair $(M,m\Lambda_0)$, 
are as follows :
\begin{equation}
		\label{eq:5.1}	\Lambda=k\Lambda_0+(k+\frac{1}{2})\beta+\frac{1}{2}(\alpha_1-y\alpha_1)+\frac{1}{2}((\beta|y\alpha_1)-(k+\frac{1}{2})
(\beta|\beta))\delta,
\end{equation}
and all of them have level
\begin{equation}
		\label{eq:5.2}
		k=\frac{2m+1}{2M}-\frac{1}{2}.
\end{equation}

\newpage
There are four kinds of simple subsets :

\vspace{1ex}
$
S^{(1)}_{k_1,k_2}=\{k_0\delta+\alpha_0, k_1\delta+\alpha_1, k_2\delta+\alpha_2\},\, M=k_0+2(k_1+k_2)+1,\, k_i\in \ZZ_{\geq 0},$ 

\vspace{1ex}
$y=1,\, \beta=-k_1(\alpha_1+2\alpha_2)-(k_1+2k_2)\alpha_1;
$

\vspace{2ex}
$
S^{(2)}_{k_1,k_2}=\{k_0\delta-\alpha_0, k_1\delta-\alpha_1, k_2\delta-\alpha_2\},\, M=k_0+2(k_1+k_2)-1,\, k_i\in \ZZ_{\geq 1},$

\vspace{1ex}
$y=r_{\alpha_2}r_{\theta},\, \beta=k_1(\alpha_1+2\alpha_2)+(k_1+2k_2)\alpha_1;
$

\vspace{2ex}
$
S^{(3)}_{k_1,k_2}=\{k_0\delta+\alpha_0, 
k_1\delta+\alpha_1+2\alpha_2, k_2\delta-\alpha_2\},\, M=k_0+2(k_1+k_2)+1,\, k_0,k_1\in \ZZ_{\geq 0}, k_2\in\ZZ_{\geq 1},$

\vspace{1ex}
$y=r_{\alpha_2}, \,\beta=-(k_1+2k_2)(\alpha_1+2\alpha_2)-k_1\alpha_1;
$

\vspace{2ex}
$
S^{(4)}_{k_1,k_2}=\{k_0\delta-\alpha_0, k_1\delta-\alpha_1-2\alpha_2, k_2\delta+\alpha_2\},\, M=k_0+2(k_1+k_2)-1,\, k_0,k_1\in \ZZ_{\geq 1}, k_2\in\ZZ_{\geq 0},$

\vspace{1ex}
$y=r_{\theta},\, \beta=(k_1+2k_2)(\alpha_1+2\alpha_2)+k_1\alpha_1.
$

\vspace{2ex}
By (\ref{eq:5.1}), all principal admissible weights with respect to the simple subsets $S^{(s)}_{k_1,k_2}$, associated to the pair $(M,\Lambda^0=m\Lambda_0)$, are as follows, where $k$ is given by (\ref{eq:5.2}) :
$$
\Lambda^{(1)}_{k_1,k_2}=k\Lambda_0-(k+\frac{1}{2})k_1(\alpha_1+2\alpha_2)-(k+\frac{1}{2}) (k_1+2k_2)\alpha_1-\varphi(k_1,k_2)\delta;
$$
$$
\Lambda^{(2)}_{k_1,k_2}=k\Lambda_0+(k+\frac{1}{2})k_1(\alpha_1+2\alpha_2)+((k+\frac{1}{2})(k_1+2k_2)+1)\alpha_1-\varphi(k_1,k_2)\delta;
$$
$$
\Lambda^{(3)}_{k_1,k_2}=k\Lambda_0-((k+\frac{1}{2})(k_1+2k_2)+\frac{1}{2})(\alpha_1+2\alpha_2)+(\frac{1}{2}- (k+\frac{1}{2})k_1)\alpha_1-\varphi(k_1,k_2)\delta;
$$
$$
\Lambda^{(4)}_{k_1,k_2}=k\Lambda_0+((k+\frac{1}{2})(k_1+2k_2)+\frac{1}{2})(\alpha_1+2\alpha_2)+((k+\frac{1}{2})k_1+\frac{1}{2})\alpha_1-\varphi(k_1,k_2)\delta,
$$
where
$$
\varphi(k_1,k_2)=\frac{1}{2}(\frac{2m+1}{M}k_1(k_1+2k_2)+k_1).
$$

As in the $\widehat{sl}_{2|1}$ case in Section 2, the non-twisted and twisted normalized admissible (super)characters can be written in terms of the following functions :
\begin{equation}
		\label{eq:5.3}	\Psi^{[B;M,m,\epsilon]}_{a,b;\epsilon^\prime}(\tau,z_1,z_2,t)=q^{\frac{2m+1}{2M}ab} e^{\frac{2\pi i(2m+1)}{2M}(bz_1+az_2)} \Phi^{[B;m]}(M\tau,z_1+a\tau+\epsilon,z_2+b\tau+\epsilon,\frac{t}{M}),
\end{equation}
where $\epsilon,\epsilon^\prime=0$ or $\frac{1}{2}$, $a,b\in \epsilon^\prime+\ZZ, a-b\in 2\ZZ$.

\vspace{0.5ex}
Formula (3.28) from [KW] (see [GK] for its proof) gives, using (\ref{eq:4.0}), 
the following expressions for the normalized 
supercharacters of principal admissible modules over $\widehat{osp}_{3|2}$ 
in terms of the functions $\Psi^{[B;M,m,\epsilon]}_{a,b;\epsilon^\prime}$.

\begin{lemma}
\label{lem:5.1}
Let $\Lambda=\Lambda^{(s)}_{k_1,k_2}$ be a principal admissible weight (of level (\ref{eq:5.2})) with respect to the simple subset $S^{(s)}_{k_1,k_2}$,
associated to the pair $(M,m\Lambda_0)$. Then the normalized supercharacters $ch^{-}_{\Lambda}$ are given by the following formulae (where $\widehat{R}^{-}$ is the affine $osp_{3|2}$ superdenominator):
\begin{list}{}{}
\item (1) $\Lambda=
\Lambda^{(1)}_{k_1,k_2}
\;:\;\widehat{R}^{-}ch^{-}_{\Lambda}=\Psi^{[B;M,m;0]}_{k_1,k_1+2k_2;0}$\,;

\item (2) $\Lambda=\Lambda^{(2)}_{k_1,k_2}\;:\;\widehat{R}^-ch^{-}_{\Lambda}=-\Psi^{[B;M,m;0]}_{-k_1,-(k_1+2k_2);0}$\,;

\item (3) $\Lambda=\Lambda^{(3)}_{k_1,k_2}\;:\;\widehat{R}^-\,ch^{-}_{\Lambda}=-\Psi^{[B;M,m;0]}_{k_1+2k_2,k_1;0}$\,;

\item (4) $\Lambda=\Lambda^{(4)}_{k_1,k_2}\;:\;\widehat{R}^-\,ch^{-}_{\Lambda}=\Psi^{[B;M,m;0]}_{-(k_1+2k_2),-k_1;0}$.

\hfill
$\Box$
\end{list}
\end{lemma}

\begin{corollary}
\label{co:5.2}
Let 
$\Lambda^{(s)}_{k_1,k_2}$ 
be a principal admissible weight 
with respect to the simple subset $S^{(s)}_{k_1,k_2}$  
, $s=1,2,3,4$. Let $\epsilon_s=(-1)^{\frac{(s-1)(s-2)}{2}}$. Introduce the 
following reparametrization of indices $(k_1,k_2)$ :
$$
\begin{array}{ll}
s=1:\; j=k_1, k=k_1+2k_2,\;\text{so that}\; 0\leq j \leq k,\, j+k\leq M-1;\\[1.5ex]
s=2:\; j=M-k_1, k=M-(k_1+2k_2),\;\text{so that}\; 1\leq k < j \leq M-1,\, j+k \geq M;\\[1.5ex]
s=3:\; j=k_1+2k_2, k=k_1,\;\text{so that}\; 0\leq k < j, \,j+k \leq M-1;\\[1.5ex]
s=4:\; j=M-(k_1+2k_2), k=M-k_1,\;\text{so that}\; 1\leq j < k \leq M-1,\, j+k \geq M.
\end{array}
$$
Then, as $s$ runs over $1,2,3,4$, the set of pairs $(j,k)$ 
fills exactly the set of points with integer coordinates,
such that $j-k$ is even, in the square 
$0\leq j, k\leq M-1$, and the supercharacter formula becomes :
$$
\widehat{R}^{-}ch^{-}_{\Lambda^{(s)}_{k_1,k_2}}=\epsilon_s \Psi^{[B;M,m;0]}_{j,k;0}.
$$
\end{corollary}

\begin{proof}
\label{}
We use the following simple observation, which follows from Lemma 
\ref{lem:4.2}(3):
\begin{equation} 
\label{eq:5.4}
		\Psi^{[B;M,m;\epsilon]}_{j+aM,k+aM;\epsilon^\prime}=(-1)^{a(1+2\epsilon)} \Psi^{[B;M,m;\epsilon]}_{j,k;\epsilon^\prime}\;\;\text{if}\;\;a\in \ZZ.
\end{equation}
\end{proof}

Formula (\ref{eq:5.3}) and Lemma \ref{lem:5.1} show that, in order to construct a modular invariant family of characters, we need to study modular transformation properties of the functions $\Phi^{[B;M]}(M\tau,z_1,z_2,t)$, or rather their modifications, $\overset{\sim}{\Phi}^{[B;M]}(M\tau,z_1,z_2,t)$.

\vspace{0.5ex}
Let $m\in \ZZ_{\geq 0}$, $M\in \ZZ_{\geq 1}$ be such that $\gcd(M,4m+2)=1$, and let (see (\ref{eq:2.4a})) :
$$
I=I^{[-1]}_{M,4m+2}.
$$

\begin{lemma}
\label{lem:5.3}
$
\Phi^{[B;m]}(\frac{\tau}{M},\frac{z_1}{M},\frac{z_2}{M},0)=\sum\limits_{0\leq a <M \atop b\in I} (-1)^a e^{\frac{\pi i(2m+1)}{M}((a+2b)z_1+az_2)}
$

$
\times q^{\frac{2m+1}{2M}a(a+2b)}\Phi^{[B;m]} (M\tau,z_1+a\tau,z_2+(a+2b)\tau,0).
$
\end{lemma}
\begin{proof}
\label{}
Letting $t=0$ and expanding the RHS of (\ref{eq:4.4}) in the geometric progression in the domain $\Im\,z_1>0$ and replacing $(\tau,z_1,z_2)$ by $\frac{1}{M}(\tau,z_1,z_2)$, we obtain :
\begin{equation}
\label{eq:5.4a}
\Phi^{[B;m]}_{1}\left(\frac{\tau}{M},\frac{z_1}{M},\frac{z_2}{M}\right)=
\left(\sum\limits_{j,k\geq0 \atop k\,odd}-\sum\limits_{j,k<0 \atop k\,odd}\right)(-1)^j e^{\frac{\pi i(2m+1)}{M}j(z_1+z_2)+\frac{\pi i}{M}k z_1} q^{\frac{1}{M}\left(\left(m+\frac{1}{2}\right)j^2+\frac{jk}{2}\right)}.
\end{equation}
Decomposing $j$ and $k$ as 
$$
j=j^\prime M+a, k=k^\prime M+(4m+2)b,\;\, \text{where}\;\, 0\leq a < M,\, b \in I,
$$
we have (cf. (i) and (ii) after (\ref{eq:2.4a})):
$$
(-1)^j=(-1)^{j^\prime+a},\, k\,\;\text{odd iff}\;\,k^\prime\,\;\text{odd},\;\, j \geq 0\,\;\text{iff}\;\,j^\prime \geq 0,\, k\geq 0\, \,\;\text{iff}\;\, k^\prime \geq 0.
$$
Hence the RHS of (\ref{eq:5.4a}) becomes:
\medskip\\

\noindent $ \sum\limits_{0\leq a < M \atop b \in I} (-1)^a e^{\frac{\pi i(2m+1)}{M}((a+2b)z_1+az_2)} q^{\frac{2m+1}{2M}a(a+2b)} $
\medskip\\
\noindent $\times(\sum\limits_{j^\prime,k^\prime \geq 0 \atop k^\prime\,odd} - \sum\limits_{j^\prime,k^\prime < 0 \atop k^\prime\,odd}) (-1)^{j^\prime} e^{\pi i(2m+1)j^\prime((z_1+a\tau)+(z_2+(a+2b)\tau))}e^{\pi ik^\prime(z_1+a\tau)} q^{M((m+\frac{1}{2})j^{\prime 2}+\frac{j^\prime k^\prime}{2})}. $
\medskip\\
Therefore
\begin{equation}
\label{eq:5.4b} 
\Phi^{[B;m]}_{1}\left(\frac{\tau}{M}, \frac{z_1}{M}, \frac{z_2}{M}\right) 
\end{equation}
\[ = \sum\limits_{0\leq a< M \atop b \in I} (-1)^{a} e^{\frac{\pi i(2m+1)}{M}((a+2b)z_1+az_2)} q^{\frac{2m+1}{2M}a(a+2b)} \Phi^{[B;m]}_{1}(M\tau,z_1+a\tau,z_2+(a+2b)\tau).\]
\bigskip 
Next, using Lemma \ref{lem:4.2}(1), we obtain from (\ref{eq:5.4b}):
\begin{equation}
\label{eq:5.4c}
\Phi^{[B;m]}_{1}\left(\frac{\tau}{M}, \frac{z_2}{M}, \frac{z_1}{M}\right)  
\end{equation} 
\[= \sum\limits_{0\leq a< M \atop b \in I} (-1)^{a} e^{\frac{-\pi i(2m+1)}{M}((a+2b)z_2+az_1)} q^{\frac{2m+1}{2M}a(a+2b)}  \Phi^{[B;m]}_{1}(M\tau,z_2-a\tau,z_1-(a+2b)\tau). \]

%
Decomposing $-(a+2b)=nM+a^\prime$, where $0\leq a^\prime < M$, we obtain, using Lemma \ref{lem:4.2}(3) :
$$
\begin{array}{l@{\!\!\!\!\!\!}l}
	&\Phi^{[B;m]}_{1}(M\tau,z_2-a\tau, z_1-(a+2b)\tau)\\
&= (-1)^{n} e^{-\pi i (2m+1)n(z_1+z_2)} q^{-(2m+1)n(a^\prime+b)-M(n+\frac{1}{2})n^2} \Phi^{[B;m]}_{1}(M\tau,z_2+(a^\prime+2b)\tau,z_1+a^\prime \tau).
\end{array}
$$
Plugging this in (\ref{eq:5.4c}), we obtain :
$$
\begin{array}{ll}
&\Phi^{[B;m]}_{1}\left(\frac{\tau}{M}, \frac{z_2}{M}, \frac{z_1}{M}\right)\\
&= \sum\limits_{0\leq a^{\prime}< M \atop b \in I} (-1)^{a^{\prime}} e^{\frac{\pi i(2m+1)}{M}((a^{\prime}+2b)z_1+a^{\prime} z_2)} q^{\frac{2m+1}{2M}a^{\prime}(a^{\prime}+2b)} \Phi^{[B;m]}_{1}(M\tau,z_2+(a^{\prime}+2b)\tau,z_1+a^\prime\tau).
\end{array}
$$
This formula together with (\ref{eq:5.4b}) concludes the proof.
\end{proof}

Translating Lemma \ref{lem:5.3} into the language of $\varphi^{[B;m]}$, we 
obtain :
\begin{equation}
		\label{eq:5.5}
		\varphi^{[B;m]}(\frac{\tau}{M},\frac{u}{M},\frac{v}{M},0)=\!\!\!\sum\limits_{0 \leq a < M \atop b \in I}\!\!(-1)^{a+b} 
e^{\frac{2\pi i(2m+1)}{M} (au+bv)} q^{\frac{2m+1}{2M}(a^2-b^2)} \varphi^{[B;m]}(M\tau,u+a\tau,v-b\tau,0).	
\end{equation}

Next, we derive a similar formula for the additional term, hence for the modified function $\overset{\sim}{\varphi}^{[B;m]}$ :
\begin{lemma}
\label{lem:5.4}
Formula (\ref{eq:5.5}) holds if we replace $\varphi^{[B;m]}$ by $\varphi^{[B;m]}_{add}$, hence, by $\overset{\sim}{\varphi}^{[B;m]}$.

\hfill
$\Box$
\end{lemma}

Translating Lemma \ref{lem:5.4} back, into the language of $\overset{\sim}{\Phi}^{[B;m]}$, we obtain

\begin{proposition}
\label{prop:5.5}
Let $m\in \ZZ_{\geq 0}$, $M\in \ZZ_{\geq 1}$ be such that $\gcd(M, 4m+2)=1$. Then 

\vspace{2.5ex}
$
\overset{\sim}{\Phi}^{[B;m]}(\frac{\tau}{M},\frac{z_1}{M},\frac{z_2}{M},0) = \sum\limits_{0\leq a < M \atop b \in I} (-1)^a e^{\frac{\pi i(2m+1)}{M}((a+2b)z_1+az_2)}$

$\times q^{\frac{2m+1}{2M}a(a+2b)} \overset{\sim}{\Phi}^{[B;m]}(M\tau,z_1+a\tau,z_2+(a+2b)\tau,0).	
$

\hfill $\Box$
\end{proposition}

Let 
$$
\Omega^{[B;M;\epsilon]}=\{(j,k)\in(\epsilon+\ZZ)^2|-M\leq j < M, 0\leq k <M,j\equiv k\,\text{mod}\,2\};\,\, \Omega^{[B;M]}=\Omega^{[B;M;0]}.
$$
Then, after applying Theorem \ref{th:4.14}, Proposition \ref{prop:5.5} leads to the following result.

\begin{proposition}
\label{prop:5.6}
Let $m\in \ZZ_{\geq 0}$, $M\in \ZZ_{\geq 1}$ be such that $\gcd(M, 4m+2)=1$. 
Then 

\vspace{2.5ex}
$
\overset{\sim}{\Phi}^{[B;m]}(-\frac{M}{\tau},\frac{z_1}{\tau},\frac{z_2}{\tau},t-\frac{z_1z_2}{\tau M})
$

$
 =\frac{\tau}{M} \sum\limits_{(a,b)\in \Omega^{[B;M]}} (-1)^a e^{\frac{2\pi i(2m+1)}{2M}(bz_1+az_2)}
q^{\frac{2m+1}{2M}ab} \overset{\sim}{\Phi}^{[B;m]}(M\tau,z_1+a\tau,z_2+b\tau,t).
$

\hfill $\Box$
\end{proposition}
Using Proposition \ref{prop:5.6}, we obtain the modular transformation formulae for the functions 
$\overset{\sim}{\Psi}^{[B;M,m;\epsilon]}_{j,k;\epsilon^\prime}$,
obtained from $\Psi^{[B;M,m;\epsilon]}_{j,k;\epsilon^\prime}$ in 
(\ref{eq:5.3}) by replacing 
$\Phi^{[B,m]}$ by $\tilde{\Phi}^{[B,m]}$.
\begin{theorem}
\label{th:5.7}
Let $m\in \ZZ_{\geq 0}$, $M\in \ZZ_{\geq 1}$ be such that $\gcd(M, 4m+2)=1$
, and let $\epsilon,\epsilon^\prime=0$ or $\frac{1}{2}$. Then for $j,k\in \Omega^{[B;M;\epsilon^\prime]}$ we have : 
\begin{list}{}{}
\item (1) $\overset{\sim}{\Psi}^{[B;M,m;\epsilon]}_{j,k;\epsilon^\prime}(-\frac{1}{\tau},\frac{z_1}{\tau},\frac{z_2}{\tau},t-\frac{z_1z_2}{\tau M})$ 

$=\frac{\tau}{M}(-1)^{\epsilon^\prime+j}\sum\limits_{(a,b)\in \Omega^{[B;M;\epsilon]}}(-1)^{a-\epsilon}e^{-\frac{2\pi i(2m+1)}{2M}(ak+bj)} 
\overset{\sim}{\Psi}^{[B;M,m;\epsilon^\prime]}_{a,b;\epsilon}(\tau,z_1,z_2,t)$.

\item (2) $\overset{\sim}{\Psi}^{[B;M,m;\epsilon]}_{j,k;\epsilon^\prime}(\tau+1,z_1,z_2,t)=
(-1)^{j-\epsilon^\prime+4\epsilon\epsilon^\prime} e^{\frac{2\pi i(2m+1)}{2M}jk} \overset{\sim}{\Psi}^{[B;M,m;|\epsilon-\epsilon^\prime|]}_{j,k;\epsilon^\prime}(\tau,z_1,z_2,t)$.

\hfill $\Box$
\end{list}
\end{theorem}

Comparing the set $\Omega^{[B;m]}$ with Corollary \ref{co:5.2}, we see that points $(j,k)\in \Omega^{[B;m]}$ with $j$ negative are missing. The reason is 
that the set $\{\Lambda^{(s)}_{k_1,k_2} \}$, that occurs in 
Lemma \ref{lem:5.1}, does not exhaust all principal admissible weights. 
This happens because the set of simple roots $\widehat{\Pi}$ of $\widehat{osp}_{3|2}$ with scalar products (\ref{eq:4.1a}) is not the only 
set of simple roots, up to equivalence. The other one is 
$r_{\alpha_1}\widehat{\Pi}$, where $r_{\alpha_1}$ is the odd reflection with respect to the isotropic simple root $\alpha_1$.

The corresponding simple subsets are 
$S^{\prime(s)}_{k_1,k_2}:=r_{\gamma_s}\,S^{(s)}_{k_1,k_2}\;(s=1,2,3,4)$, 
where $\gamma_s$ is the isotropic root in $S^{(s)}_{k_1,k_2}$. Taking into account the condition that a simple subset should consist of positive roots (with respect to $\widehat{\Pi}$), we obtain the following new four kinds of simple 
subsets :
$$
\begin{array}{ll}
	S^{\prime(1)}_{k_1,k_2}=\{(k_0+k_1+1)\delta-\alpha_1-2\alpha_2, -k_1\delta-\alpha_1, (k_1+k_2)\delta+\alpha_1+\alpha_2 \}, \\[1ex]
	M=k_0+2(k_1+k_2)+1,\, k_1+2k_2 \leq M-1, k_1 \leq -1, k_1+k_2\geq 0;\\[2ex]
	
	S^{\prime(2)}_{k_1,k_2}=\{(k_0+k_1-1)\delta+\alpha_1+2\alpha_2, -k_1\delta+\alpha_1, (k_1+k_2)\delta-\alpha_1-\alpha_2 \}, \\[1ex]
	M=k_0+2(k_1+k_2)-1, \,k_1+2k_2 \leq M, k_1 \leq 0, k_1+k_2\geq 1;\\[2ex]
	
	S^{\prime(3)}_{k_1,k_2}=\{(k_0+k_1+1)\delta-\alpha_1, 
-k_1\delta-\alpha_1-2\alpha_2, (k_1+k_2)\delta+\alpha_1+\alpha_2 \}, \\[1ex]
	M=k_0+2(k_1+k_2)+1,\, k_1+2k_2 \leq M-1, k_1 \leq -1, k_1+k_2\geq 0;\\[2ex]
	
	S^{\prime(4)}_{k_1,k_2}=\{(k_0+k_1-1)\delta+\alpha_1, -k_1\delta+\alpha_1+2\alpha_2, (k_1+k_2)\delta-\alpha_1-\alpha_2 \}, \\[1ex]
	M=k_0+2(k_1+k_2)-1,\, k_1+2k_2 \leq M, k_1 \leq 0, k_1+k_2\geq 1.\\
\end{array}
$$
The corresponding principal admissible weights are again 
$\Lambda^{(s)}_{k_1,k_2}$
, but the range of the pairs $(k_1,k_2)$ is different. We shall denote them by
 $\Lambda^{\prime(s)}_{k_1,k_2}$, in order to distinguish them from those, corresponding to the simple subsets $S^{(s)}_{k_1,k_2}$.
For these new principal admissible weights we have the following analogue of 
Lemma \ref{lem:5.1}.

\begin{lemma}
\label{lem:5.8}
Let $\Lambda=\Lambda^{\prime(s)}_{k_1,k_2}$ be a principal admissible weight (of level (\ref{eq:5.2})) with respect to the simple subset $S^{\prime(s)}_{k_1,k_2}$, associated to the pair $(M, m\Lambda_0)$. Then the normalized supercharacters are given by the following formulae :
\begin{list}{}{}
\item (1) $\Lambda=
\Lambda^{\prime(1)}_{k_1,k_2}
\;:\;\widehat{R}^{-}ch^{-}_{\Lambda}=\Psi^{[B;M,m;0]}_{k_1,k_1+2k_2;0}$\,;

\item (2) $\Lambda
=\Lambda^{\prime(2)}_{k_1,k_2}\;:\;\widehat{R}^-ch^{-}_{\Lambda}=-\Psi^{[B;M,m;0]}_{-k_1,-(k_1+2k_2);0}$\,;

\item (3) $\Lambda=\Lambda^{
\prime(3)}_{k_1,k_2}\;:\;\widehat{R}\,ch^{-}_{\Lambda}=-\Psi^{[B;M,m;0]}_{k_1+2k_2,k_1;0}$;

\item (4) $\Lambda=\Lambda^
{\prime(4)}_{k_1,k_2}\;:\;\widehat{R}\,ch^{-}_{\Lambda}=\Psi^{[B;M,m;0]}_{-(k_1+2k_2),-k_1;0}$.

\hfill
$\Box$
\end{list}
\end{lemma}

\begin{corollary}
\label{co:5.9}
Let 
$\Lambda = \Lambda^{\prime(s)}_{k_1,k_2}$ 
be a principal admissible weight 
with respect to the simple subset $S^{\prime(s)}_{k_1,k_2}$  
, $s=1,2,3,4$. Introduce the following reparametrization of indices $(k_1,k_2)$\,:
$$
\begin{array}{ll}
s=1:\; j=k_1, k=k_1+2k_2,\;\text{so that}\;\; j\leq -1,\, k \leq M-1,\, j+k\geq 0;\\[1.5ex]
s=2:\; j=-k_1, k=-k_1-2k_2,\;\text{so that}\;\;j\geq 0, \,k\geq -M,\, j+k \leq -2;\\[1.5ex]
s=3:\; j=k_1+2k_2, k=k_1,\;\text{so that}\;\; j \leq M-1,\, k\leq -1,\, j+k\geq 0;\\[1.5ex]
s=4:\; j=-k_1-2k_2, k=-k_1,\;\text{so that}\;\; j \geq -M,\, k\geq 0,\, j+k \leq -2.
\end{array}
$$
Then, as $s$ runs over $1$ and $4$ (resp. $3$ and $2$) the set of pairs 
$(j,k)$ 
fills exactly the points with integer coordinates,
such that $j-k$ is even, in the square 
$-M \leq j < 0,\, 0 \leq k \leq M-1$ (resp. $0 \leq j \leq M-1,\, -M \leq k<0$)
, and we obtain a unified supercharacter 
formula for the principal admissible highest weights
${\Lambda=\Lambda^{(s)}_{k_1,k_2}}$ and ${\Lambda^{\prime(s)}_{k_1,k_2}}$: 
$$
\widehat{R}^{-}ch^{-}_{\Lambda}=\pm \Psi^{[B;M,m;0]}_{j,k;0}.
$$
%
\hfill
$\Box$
\end{corollary}


As in [KW] and Section 2 of the presnt paper, we introduce the twisted 
normalized admissible characters and supercharacters $t_\xi\,ch^{\pm}_{\Lambda}$ of $\widehat{osp}_{3|2}$ and their denominators and superdenominators $t_\xi\,\widehat{R}^{\pm}$.

Choose $\xi=-(\alpha_1+\alpha_2)$. Then we have in the coordinates (\ref{eq:4.1}) :
$$
t_{-\xi}(h)=(\tau,z_1+\frac{\tau}{2},z_2+\frac{\tau}{2},t+\frac{z_1+z_2}{2}+\frac{\tau}{4}).
$$
The proof of the following lemma is straightforward.

\begin{lemma}
\label{lem:5.10}
$$
\Psi^{[B;M,m;\epsilon]}_{j,k;\epsilon^\prime}(\tau,z_1+\frac{\tau}{2},z_2+\frac{\tau}{2},t+\frac{z_1+z_2}{2}+\frac{\tau}{4}) = \Psi^{[B;M,m;\epsilon]}_{j+\frac{1}{2},k+\frac{1}{2};\frac{1}{2}-\epsilon^\prime}(\tau,z_1,z_2,t),
$$
and the same formula holds if we replace $\Psi$ by $\overset{\sim}{\Psi}$.

\hfill
$\Box$
\end{lemma}

We shall use the same notation (\ref{eq:2.6a}) and (\ref{eq:2.6b}) as before, 
and let 
$$
ch^{[B;M,m;\epsilon]}_{j+\epsilon^\prime,k+\epsilon^\prime;\epsilon^\prime}=ch^{(\epsilon)}_{\Lambda^{(s)}_{j,k};\epsilon^\prime}.
$$
Then
\begin{equation}
\label{eq:5.10a}
\{ch^{[B;M,m;\epsilon]}_{j+\epsilon^\prime,k+\epsilon^\prime;\epsilon^\prime}|
\,\,j,k\in\Omega^{[B;M]}\} 
\cup \{ch^{[B;M,m;\epsilon]}_{j+\epsilon^\prime,k+\epsilon^\prime;\epsilon^\prime}|\,\,j,k\in\ZZ,\,0\leq j\leq M-1,\,-M\leq k <0\}
\end{equation}
is (up to a sign) precisely the set of all normalized principal
admissible characters (resp. supercharacters), associated to the pair $(M,m\Lambda_0)$, if $\epsilon^\prime=0$ 
and $\epsilon=\frac{1}{2}\, (\text{resp.}\;\epsilon=0)$, and it is the set of 
all normalized twisted principal admissible characters (resp. supercharacters), associated to the pair $(M,m\Lambda_0)$, if $\epsilon^\prime=\frac{1}{2}$ and $\epsilon=\frac{1}{2}\; (\text{resp.}\;\epsilon=0)$.

By Corollaries \ref{co:5.2} and \ref{co:5.9}, and Lemma \ref{lem:5.10}, we obtain the following unified formula :
\begin{equation}
	\label{eq:5.6}
	\hat R^{(\epsilon)}_{\epsilon^\prime} 
ch^{[B;M,m;\epsilon]}_{j+\epsilon^\prime,k+\epsilon^\prime,\epsilon^\prime}=
\overset{\sim}{\epsilon}_{s} \Psi^{[B;M,m;\epsilon]}_{j+\epsilon^\prime,
k+\epsilon^\prime;\epsilon^\prime},\,\,\,\,j,k\in
\Omega^{[B;M]},
\end{equation}
where 
$\overset{\sim}{\epsilon}_{s} 
= \epsilon_s$ if $j\geq 0$ and 
$\overset{\sim}{\epsilon}_{s}=1$ if $j<0$.

In view of these observations, introduce the modified normalized untwisted and 
twisted (super)characters, letting
\begin{equation}
	\label{eq:5.7}
	\overset{\sim}{ch}^{[B;M,m;\epsilon]}_{j,k;\epsilon^\prime}(\tau, z_1,z_2,t)=(-1)^{j-\epsilon^\prime}\frac{\overset{\sim}{\Psi}^{[B;M,m;\epsilon]}_{j,k;\epsilon^\prime}(\tau,z_1,z_2,t)}{\widehat{R}^{(\epsilon)}_{\epsilon^\prime}(\tau,z_1,z_2,t)}, \,\,\,\,j,k \in \Omega^{[B;M;\epsilon^\prime]}.
\end{equation}
Since , by Theorem \ref{th:4.14}(3), 
we have:
$\overset{\sim}{\Psi}^{[B;M,m;\epsilon]}_{j+aM,k+bM;\epsilon^\prime}
=(-1)^{a+\epsilon(a+b)}\overset{\sim}{\Psi}^{[B;M,m;\epsilon]}_{j,k;\epsilon^\prime}$ for $a,b\in\ZZ$, $a+b$ even, 
in view of Corollary \ref{co:5.9}, after the modification 
the second subset in (\ref{eq:5.10a}) coincides with a half of the first 
one, hence the first subset
produces, after the modification, all untwisted and twisted modified 
normalized principal
admissible (super)characters, associated to the pair $(M,m\Lambda_0)$.

Using (\ref{eq:4.1a}), (\ref{eq:4.1}), (\ref{eq:4.3}), (\ref{eq:4.4}) 
from [KW], we can write the following unified formula for the untwisted 
and twisted 
(normalized) (super)denominators for $\widehat{osp}_{3|2}$ :
\begin{equation}
\label{eq:5.8}
\widehat{R}^{(\epsilon)}_{\epsilon^{\prime}}(\tau,z_1,z_2,t)=c_{\epsilon,\epsilon^{\prime}} e^{\pi i t} \frac{\vartheta(\tau)^3 \vartheta_{11}(\tau,z_1+z_2) \vartheta_{11}(\tau,\frac{z_1-z_2}{2})}{\vartheta_{1-2\epsilon^{\prime},1-2\epsilon}(\tau,z_1)\vartheta_{1-2\epsilon^\prime,1-2\epsilon}(\tau,z_2)\vartheta_{1-2\epsilon^\prime,1-2\epsilon}(\tau,\frac{z_1+z_2}{2})}\,,
\end{equation}
where $c_{\epsilon,\epsilon^{\prime}}=e^{\frac{\pi i}{2}(2(\epsilon-\epsilon^\prime)-4\epsilon\epsilon^{\prime}-1)}$.

The modular transformation formulae for the functions $\widehat{R}^{(\epsilon)}_{\epsilon^\prime}$ follow from Theorem 4.1 in [KW] :
\begin{lemma}
\label{lem:5.11}
$$
$$
\vspace{-7ex}

\begin{list}{}{}
\item (1) $\widehat{R}^{(\epsilon)}_{\epsilon^\prime}(-\frac{1}{\tau},\frac{z_1}{\tau},\frac{z_2}{\tau},t-\frac{z_1z_2}{\tau})=i^{-(2\epsilon+2\epsilon^\prime+4\epsilon\epsilon^\prime)} \tau \widehat{R}^{(\epsilon^\prime)}_{\epsilon}(\tau,z_1,z_2,t)$.

\item (2) $\widehat{R}^{(\epsilon)}_{\epsilon^\prime}(\tau+1,z_1,z_2,t)=e^{\frac{3\pi i}{2}\epsilon^\prime}\widehat{R}^{(|\epsilon-\epsilon^\prime|)}_{\epsilon^\prime}(\tau,z_1,z_2,t)$.

\hfill
$\Box$
\end{list}
\end{lemma}

Modular transformation properties of the modified normalized (super)characters (\ref{eq:5.7}) follow from Theorem \ref{th:5.7} and Lemma \ref{lem:5.11} :

\begin{theorem}
\label{th:5.12}
Let $m\in \ZZ_{\geq 0}$ and $M\in \ZZ_{\geq 1}$ be such that $\gcd(M, 4m+2)=1$
, and let $\epsilon,\epsilon^\prime=0$ or $\frac{1}{2}$. Then for $j,k\in \Omega^{[B;M;\epsilon^\prime]}$ we have : 
\begin{list}{}{}
\item (1) $\overset{\sim}{ch}^{[B;M,m;\epsilon]}_{j,k;\epsilon^\prime}(-\frac{1}{\tau},\frac{z_1}{\tau},\frac{z_2}{\tau},t-\frac{z_1z_2}{\tau})$ 

$=\frac{i^{4\epsilon \epsilon^\prime+2\epsilon-2\epsilon^\prime}}{M}\sum\limits_{(a,b)\in \Omega^{[B;M;\epsilon]}} e^{-\frac{\pi i(2m+1)}{M}(bj+ak)} 
\overset{\sim}{ch}^{[B;M,m;\epsilon^\prime]}_{a,b;\epsilon}(\tau,z_1,z_2,t)$.

\item (2) $\overset{\sim}{ch}^{[B;M,m;\epsilon]}_{j,k;\epsilon^\prime}(\tau+1,z_1,z_2,t)=
(-1)^{4\epsilon\epsilon^\prime} e^{\frac{\pi i(2m+1)}{M}jk+\pi i(j-\frac{3}{2}\epsilon^\prime)} \overset{\sim}{ch}^{[B;M,m;|\epsilon-\epsilon^\prime|]}_{j,k;\epsilon^\prime}(\tau,z_1,z_2,t)$.

\hfill $\Box$
\end{list}
\end{theorem}

\begin{remark}   
\label{rem:5.13}
  Due to Remark \ref{rem:4.15}, Theorem \ref{th:5.12} holds if we replace the modified normalized characters $\overset{\sim}{ch}$ by the non-modified ones 
$ch$ 
for any odd $M\in \ZZ_{\geq 1}$. In particular the normalized characters
$$
\{ ch^{[B;M,0;\epsilon]}_{j,k;\epsilon^\prime}|\,\, \epsilon, \epsilon^\prime=0\;\text{or}\; \frac{1}{2},\, (j,k)\in \Omega^{[B;M;\epsilon^\prime]} \}
$$
span an $SL_2(\ZZ)$-invariant space for any odd $M\in \ZZ_{\geq 1}$.
\end{remark}

\section{Modular transformation formulae for modified characters of admissible $\bs{N=3}$ modules.}
\label{sec:6}

Let $\fg=osp_{3|2}$ and let $\fh$ be its Cartan subalgebra, so that $l=
\text{dim}\,\fh=2$.

Choose the set of positive roots $\Delta_{+}$ as in Section 4 :
$$
\Delta_{+}=\{\alpha_1,\alpha_2, \alpha_1+\alpha_2, \alpha_1+2\alpha_2, \theta=2\alpha_1+2\alpha_2 \}.
$$
Let $x=\alpha_1+\alpha_2 \in \fh^{\ast}$, which is identified with $\fh$ 
via the bilinear form $(.|.)$. The eigenspace decomposition of $\fg$ with 
respect to $ad\,x$ is
$$
\fg=\fg_{-1}+\fg_{-\frac{1}{2}}+\fg_0+\fg_{\frac{1}{2}}+\fg_1 ,
$$
where $\fg_{\pm 1}=\CC e_{\pm \theta}$, $\fg_{\pm \frac12}$ are purely odd of 
dimension $3$, and $\fg_0=\CC x+\fg^{\#}$, where $\fg^{\#}$ is the orthogonal 
complement to $\CC x$ in $\fg_0$ with respect to $(.|.)$, isomorphic to 
$sl_2$. The subspace $\fg^{\#}\cap \fh$ is spanned by the element 
$J_0=-2 \alpha_2$.

Recall that the quantum Hamiltonian reduction associates to a 
$\widehat{\fg}$-module $L(\Lambda)$ of level $k\neq - \frac{1}{2}(=-h^{\vee})$, a module $H(\Lambda)$ over the $N=3$ superconformal algebra of Neveu-Schwarz 
type, such that the following properties hold [KRW], [KW2], [A] :
\begin{itemize}
	\item[(i)] the module $H(\Lambda)$ is either $0$ or an irreducible positive energy module\,;
	\item[(ii)] $H(\Lambda)=0$ iff $(\Lambda|\alpha_0)\in \ZZ_{\geq 0}$;
	\item[(iii)] the irreducible module $H(\Lambda)$ is characterized by three numbers:\\
	$(\alpha)$ the central charge
	\begin{equation}
	c_k=-3(2k+1),
	\label{eq:6.1}
	\end{equation}
	$(\beta)$ the lowest energy
	\begin{equation}
	h_{\Lambda}=\dfrac{(\Lambda+2\widehat{\rho}|\Lambda)}{2k+1}-(x+d| \Lambda),
	\label{eq:6.2}
	\end{equation}
	$(\gamma)$ the spin
	\begin{equation}
	s_\Lambda=\Lambda(J_0).
	\label{eq:6.3}
\end{equation}

\item[(iv)] the character $ch^{+}$ and the supercharacter $ch^{-}$ of the module $H(\Lambda)$ are given by the following formula :
\begin{equation}
	ch^{\pm}_{H(\Lambda)}(\tau, z):=tr^{\pm}_{H(\Lambda)} q^{L_0-\frac{c_k}{24}} e^{2\pi izJ_0}= (\widehat{R}^{\pm} ch^{\pm}_{\Lambda})
(\tau,-\tau x+J_0z,\frac{\tau}{4})\overset{3}{R}^{\pm}(\tau,z)^{-1},
	\label{eq:6.4}
\end{equation}
where
\begin{equation}
\label{eq:6.5}
\overset{3}{R}^{+}(\tau, z)= \frac{\eta(\frac{\tau}{2})\eta(2\tau)\vartheta_{11}(\tau,z)}{\vartheta_{00}(\tau,z)}\,\,\,\hbox{and}\,\,\,
\overset{3}{R}^{-}(\tau, z)=\frac{\eta(\tau)^3 \vartheta_{11}(\tau,z)}{\eta(\frac{\tau}{2})\vartheta_{01}(\tau,z)} 
\end{equation}
are the $N=3$ superconformal algebra normalized demominator and 
superdenominator.
\end{itemize}


%

Now we turn to the Ramond twisted sector. For each $\alpha \in \Delta_{+}$ 
choose $s_{\alpha} \in \ZZ$ (resp. $\in \frac{1}{2}+\ZZ$) if the root $\alpha$ is even (resp. odd), such that $s_{\theta}=0$ and 
$s_\alpha+s_{\theta-\alpha}=\delta_{\alpha,\theta/2}$ if both 
$\alpha$ and $\theta - \alpha$ are odd roots. Recall [KW3], [A] that, given such a suitable choice of $s{_\alpha}^\prime s$, the twisted quantum Hamiltonian 
reduction 
associates to a $\widehat{\fg}^{\tw}$-module $L^{\tw}(\Lambda)$ of level 
$k \neq -\frac{1}{2}$, a positive energy module $H^{\tw}(\Lambda)$ over the 
corresponding Ramond $N=3$ superconformal algebra, for which the 
properties (i) and (ii) hold with $H$ replaced by $H^{\tw}$.
A suitable choice of the $s_\alpha$ is 
(see \cite{KW3}) :
\begin{equation}
\label{eq:6.6}
s_{\alpha_1}=s_{\alpha_1+\alpha_2}=-s_{\alpha_1+2\alpha_2}= \frac{1}{2},\,
 s_{\alpha_2}=s_{\theta}=0.
\end{equation}
It is not difficult to see that, choosing the element 
$w=t_{\frac{1}{2}\theta} r_{\theta} \in \widehat{W}$, and a
lifting $\overset{\sim}{r}_{\theta}$ of $r_\theta$ in the corresponding $SL_2(C)$, we can lift $w$ to an isomorphism 
$$
\overset{\sim}{w}=t_{\frac{1}{2} \theta} \overset{\sim}{r}_{\theta} :
\widehat{\fg}^{\tw} \overset{\sim}{\rightarrow} \widehat{\fg},\;\text{where}\;\widehat{\fg}^{\tw} = \fg_{\bar{0}} [t,t^{-1}] \oplus \fg_{\bar{1}} [t,t^{-1}] t^{\frac{1}{2}} \oplus \CC K \oplus \CC d,
$$
such that the set of positive roots 
$\hat{\Delta}_+ $ of $\hat{\fg}$
corresponds to the set of positive roots 
$\hat{\Delta}^{\tw}_+ $ of 
$\hat{\fg}^{\tw}$, associated to the choice (\ref{eq:6.6}) of the $s_\alpha$'s.
The set of simple roots of 
$\hat{\Delta}^{\tw}_+$ is
\begin{equation}
\label{eq:6.7}
-\frac12 \delta+\alpha_1+2\alpha_2,\,\frac12 \delta +\alpha_1, \frac12 (\delta-\theta).
\end{equation}

It is easy to see that in the coordinates (\ref{eq:4.1}) we have:
\begin{equation}
\label{eq:6.8}
w(h)=w(\tau, z_1, z_2, t)=(\tau, 
-z_2+\frac{\tau}{2},-z_1+\frac{\tau}{2}, t-\frac{z_1+z_2}{2}+\frac{\tau}{4}).
\end{equation}
Note also that $w^2=1$.

Via the 
isomorphism $\tilde{w}$, the $\widehat{\fg}$-module $L(\Lambda)$ becomes a 
$\widehat{\fg}^{\tw}$-module, denoted by $L^{\tw}(\Lambda)$; 
its highest weight is 
\begin{equation}
	\Lambda^{\tw}=w(\Lambda),
	\label{eq:6.9}
\end{equation}
and its normalized character and supercharacter are :
\begin{equation}
	ch^{\tw,\pm}_{\Lambda}=w(ch^{\tw,\pm}_{\Lambda}),
	\label{eq:6.10}
\end{equation}
their denominator and superdenominator being 
$\widehat{R}^{\tw,\pm}=w(\widehat{R}^{\pm})$. 

The irreducible module $H^{\tw}(\Lambda)$ is again characterized by three 
numbers : the central charge $c_k$, given by (\ref{eq:6.1}), the lowest energy
\begin{equation}
	h^{\tw}_{\Lambda}=\frac{(\Lambda^{\tw}+2\widehat{\rho}^{\tw}|\Lambda^{\tw})}{2k+1} - (x+d|\Lambda^{\tw})-\frac{3}{16},
	\label{eq:6.11}
\end{equation}
and the spin
\begin{equation}
	s^{\tw}_{\Lambda}=\Lambda^{\tw}(J_0)-\frac{1}{2}.
	\label{eq:6.12}
\end{equation}

Furthermore, the (super)character of the module $H^{\tw}(\Lambda)$ is given by the following formula\,:
\begin{equation}
		ch^{\pm}_{H^{\tw}(\Lambda)}(\tau,z):\,= tr^{\pm}_{H^{\tw}(\Lambda)} q^{L^{\tw}_0-\frac{c_k}{24}} e^{2\pi izJ^{\tw}_0}=(\widehat{R}^{\tw,\pm} ch^{\tw,\pm}_{\Lambda}) (\tau, -\tau x+J_0 z,\frac{\tau}{4}) 
\overset{3}{R}^{\tw,\pm}(\tau, z)^{-1},\label{eq:6.13}
\end{equation}
where
\begin{equation}
\label{eq:6.14}
\overset{3}{R}^{\tw,+}(\tau,z)= =\dfrac{\eta(\tau)^3 \vartheta_{11}(\tau,z)}{\eta(2\tau) \vartheta_{10}(\tau,z)}
\end{equation}
is the Ramond $N=3$ superconformal algebra normalized denominator,
and $\overset{3}{R}^{\tw,-}(\tau,z)^{-1}=0$,
hence $ch^{-}_{H^{\tw}(\Lambda)}=0$. 

As in the $N=2$ case, we introduce a more convenient notation:
$\overset{3}{R}^{(\frac{1}{2})}_{\frac{1}{2}}
=\overset{3}{R}^{+}$, 
$\overset{3}{R}^{(0)}_{\frac{1}{2}}=\overset{3}{R}^{-}$, and
$\overset{3}{R}^{(\frac{1}{2})}_{0}=2^{-\frac{1}{2}}\overset{3}{R}^{\tw,+}$
(cf. (\ref{eq:6.5}) and (\ref{eq:6.14})).
Using the modular transformation formulae for the four Jacobi theta functions (see e.g. [KW], Proposition A.7) we obtain the modular transformation formulae 
for these functions.
\begin{proposition}
\label{prop:6.1}
Let $\epsilon, \epsilon^\prime = 0$ or $\frac{1}{2}$ be such that $\epsilon+\epsilon^\prime \neq 0$.
Then
\begin{list}{}{}
\item (1) $\overset{3}{R}^{(\epsilon)}_{\epsilon^\prime}(-\frac{1}{\tau},\frac{z}{\tau})=-\tau \overset{3}{R}^{(\epsilon^\prime)}_{\epsilon}(\tau,z)$.

\item (2) $\overset{3}{R}^{(\epsilon)}_{\epsilon^\prime}(\tau+1,z)=e^{\frac{\pi i}{12}(1+9 \epsilon^\prime)} \overset{3}{R}^{(|\epsilon-\epsilon^\prime|)}_{\epsilon^\prime} (\tau,z)$.
\end{list}
\hfill $\Box$
\end{proposition}

Let $\Lambda=\Lambda^{(s)}_{k_1,k_2}$ and $\Lambda=
\Lambda^{\prime(s)}_{k_1,k_2}$ 
be one of the principal admissible weights (for the simple subsets
$S^{(s)}_{k_1,k_2}$and 
 $S^{'(s)}_{k_1,k_2}$), 
described by Corollaries \ref{co:5.2} and \ref{co:5.9}, along with the reparametrization $(j,k)$ of the pairs $(k_1,k_2)$. Then we have :
$$
(\widehat{R}^{\tw,+} ch^{\tw,+}_{\Lambda}) (h)=(\widehat{R}^{+} ch^{+}_{\Lambda})(w^{-1}h)=(\widehat{R}^{+} ch^{+}_{\Lambda}) (wh).
$$
Using (\ref{eq:6.8}) and Lemma \ref{lem:5.10}, we deduce :
\begin{equation}
	(\widehat{R}^{\tw,+}
 ch^{\tw,+}_{\Lambda})(h)=\overset{\sim}{\epsilon}_s \Psi^{[B;M,m;\frac{1}{2}]}_{j+\frac{1}{2},k+\frac{1}{2};\frac{1}{2}} (\tau,-z_2,-z_1,t).
	\label{eq:6.15}
\end{equation}
Furthermore, we have in coordinates (\ref{eq:4.1a}) :
\begin{equation}
	2\pi i(-\tau \Lambda_0-\tau x+zJ_0)=(\tau,z+\frac{\tau}{2},-z+\frac{\tau}{2},0).
	\label{eq:6.16}
\end{equation}

Now we can apply formulae (\ref{eq:6.4}) and (\ref{eq:6.15}). 
Using (\ref{eq:6.16}) and Lemma \ref{lem:5.10}, we obtain the following.
\begin{proposition}
\label{prop:6.2}
Let $m\in \ZZ_{\geq 0}$ and $M\in \ZZ_{\geq 1}$ be such that $\gcd(M, 4m+2)=1$. Let $\Lambda=\Lambda^{(s)}_{k_1,k_2}$ or $\Lambda=\Lambda^{\prime(s)}_{k_1,k_2}$ 
be one of the principal admissible weights, described in Corollaries \ref{co:5.2} and \ref{co:5.9}, along with the reparametrization $(j,k)$ of the pairs 
$(k_1,k_2)$. Then\\

\begin{list}{}{}
\item (1) $(\overset{3}{R}^{\pm}{ch}^{\pm}_{H(\Lambda)})(\tau,z)=
\Psi^{[B;M,m;\frac{1}{4}(1\pm 1)]}_{j+\frac{1}{2},k+\frac{1}{2};\frac{1}{2}} (\tau,z,-z,0)$. 

\item (2) $(\overset{3}{R}^{\tw,+}{ch}^{+}_
{H^{\tw}(\Lambda)})(\tau,z)=
\Psi^{[B;M,m;\frac{1}{2}]}_{j,k;0} (\tau,z,-z,0)$.

\end{list}
\end{proposition}

By (\ref{eq:6.1}), the central charge of the $N=3$ modules 
$H(\Lambda)$ and $H^{\tw}(\Lambda)$, where 
$\Lambda =\Lambda^{(s)}_{k_1,k_2}$ or 
$\Lambda^{\prime(s)}_{k_1,k_2}$ 
have level $k$, given by formula (\ref{eq:5.2}),   
is equal to 
\begin{equation}
	\label{eq:6.17}
	c_k=-3\frac{2m+1}{M}.
\end{equation}

Recall that $H(\Lambda)$ and $H^{\tw}(\Lambda)$ are zero iff $(\Lambda | \alpha_0)\in \ZZ_{\geq 0}$, and that $H(\Lambda)$ is irreducible otherwise. If $M=1$, then a principal admissible $\widehat{osp}_{3|2}$-module is partially integrable, hence $(\Lambda | \alpha_0)\in \ZZ_{\geq 0}$ and therefore 
$H(\Lambda)=0=H^{\tw}(\Lambda)$. If $M>1$, it follows from the formulas for the principal admissible weights, listed in Section 5, that $(\Lambda | \alpha_0)\in\ZZ_{\geq 0}$ iff $k_0=0$ and $s=1$ or $3$.

Thus, in what follows we may assume that $M\geq 2$ and $k_0>0$ when $s=1$ or $3$. Using the formulas for the principal admissible weights in Section 5 and formulas (\ref{eq:6.2}), (\ref{eq:6.3}), (\ref{eq:6.11}) and (\ref{eq:6.12}), we obtain the following explicit formulae for the lowest energy $h_{\Lambda}$ (resp. $h_{\Lambda}^{\tw}$) and spin (resp. $s_{\Lambda}^{\tw}$) of $H(\Lambda)$ 
(resp. $H^{\tw}(\Lambda)$) , and the same formulae for 
$\Lambda^{\prime(s)}_{k_1,k_2}$ . :
\begin{equation}
	\label{eq:6.18}
	h_{\Lambda^{(1\,\text{or}\,3)}_{k_1,k_2}}=(k+\frac{1}{2}) ((k_1+\frac{1}{2}) 
(k_1+2k_2+\frac{1}{2}) -\frac{1}{4})+ \frac{k_1}{2} = h_{\Lambda^{(4\,\text{or}\,2)}_{k_1+1,k_2}},
\end{equation}
\begin{equation}
	\label{eq:6.19}
		s_{\Lambda^{(1)}_{k_1,k_2}}=
2(k+\frac{1}{2})k_2 = s_{\Lambda^{(4)}_{k_1+1,k_2}},
\end{equation}
\begin{equation}
	\label{eq:6.20}
	s_{\Lambda^{(3)}_{k_1,k_2}}=-2(k+\frac{1}{2})k_2-1 = s_{\Lambda^{(2)}_{k_1+1,k_2}},
\end{equation}
\begin{equation}
	\label{eq:6.21}
	h^{\tw}_{\Lambda^{(s)}_{k_1,k_2}}=(k+\frac{1}{2}) 
(k_1(k_1+2k_2)-\frac{1}{4}) +\frac{k_1}{2}- \frac{1}{16},\,\,\,s=1,2,3,4,
\end{equation}
\begin{equation}
	\label{eq:6.22}
	s^{\tw}_{\Lambda^{(1)}_{k_1,k_2}}=2(k+\frac{1}{2})k_2-\frac{1}{2} = s^{\tw}_{\Lambda^{(4)}_{k_1,k_2}},
\end{equation}
\begin{equation}
	\label{eq:6.23}
	s^{\tw}_{\Lambda^{(3)}_{k_1,k_2}}=-2(k+\frac{1}{2})k_2-\frac{3}{2} = 
s^{\tw}_{\Lambda^{(2)}_{k_1,k_2}}.
\end{equation}
Since the irreducible $N=3$ positive energy modules are uniquely determined by their characteristic numbers, we obtain the following isomorphisms of the 
$N=3$ Neveu-Schwarz type superconformal algebra modules:
\begin{equation}
	\label{eq:6.24}
H(\Lambda^{(1)}_{k_1,k_2}) \simeq H(\Lambda^{(4)}_{k_1+1,k_2}), 
\,H(\Lambda^{(3)}_{k_1,k_2}) \simeq H(\Lambda^{(2)}_{k_1+1,k_2}),\,
H(\Lambda^{\prime(1)}_{k_1,k_2}) \simeq H(\Lambda^{\prime(4)}_{k_1+1,k_2}), 
\end{equation}
and of the $N=3$ Ramond type modules :
\begin{equation}
	\label{eq:6.25}
H^{\tw}(\Lambda^{(1)}_{k_1,k_2}) \simeq H^{\tw}(\Lambda^{(4)}_{k_1,k_2}), \,
H^{\tw}(\Lambda^{(3)}_{k_1,k_2}) \simeq H^{\tw}(\Lambda^{(2)}_{k_1,k_2}),
\,H^{\tw}(\Lambda^
{\prime(1)}_{k_1,k_2}) \simeq H^{\tw}(\Lambda^{\prime(4)}_{k_1,k_2}).
\end{equation}	
Thus, we may consider only the $N=3$ modules 
$H(\Lambda^{(s)}_{k_1,k_2})$ and $H^{\tw}(\Lambda^{(s)}_{k_1,k_2})$ for $s=1$ 
and $3$, and $H(\Lambda^{\prime(s)}_{k_1,k_2})$ and 
$H^{\tw}(\Lambda^{\prime(s)}_{k_1,k_2})$ for $s=1$.

\vspace{1ex}
In the reindexing of Corollaries \ref{co:5.2} and \ref{co:5.9} we let :
$$
\begin{array}{ll}
H_{NS} \left(\Lambda_{j,k} \right) = H \left(\Lambda^{(1\;\text{or}\;3)}_{j-\frac{1}{2}, k-\frac{1}{2}} \right),\,\, 
H_{NS} \left(\Lambda_{j,k} \right) = H \left(\Lambda^{\prime (1)}_{j-\frac{1}{2}, k-\frac{1}{2}} \right), \\[4ex]
H_{R} \left(\Lambda_{j,k}\right) = H^{\tw} \left(\Lambda^{(1\;\text{or}\;3)}_{j,k}\right),\,\, 
H_{R} \left(\Lambda_{j,k} \right) = H^{\tw} 
\left(\Lambda^{\prime(1)}_{j,k} \right).
\end{array}
$$
Using (\ref{eq:6.24}), (\ref{eq:6.25}) 
and Lemma \ref{lem:5.10}, it is easy to see that the set of all pairs 
($j$, $k$) that occur in 
$H_{NS} \left(\Lambda_{j,k}\right)$ (resp. $H_R \left(\Lambda_{j,k}\right)$) 
is the following, where, as before $\epsilon=\frac{1}{2}$ (resp. $=0$) in the 
Neveu-Schwarz case (resp. Ramond case) :
$$
\Omega^{[3; M]}_{\epsilon} = \left\{(j,k) \in (\epsilon + \ZZ)^2 |\,\, 
0 \leq k < M,\, 0 < j+k < M,\, j-k \in 2\ZZ \right\}.
$$

It is easy to see from (\ref{eq:6.18})-(\ref{eq:6.23}) that the characteristic 
numbers of these $N=3$ modules are given by the following unified formulae
for $\Lambda=\Lambda^{(s)}_{k_1+1,k_2}$ or 
$\Lambda=\Lambda^{\prime(s)}_{k_1+1,k_2}$ :
\begin{equation}
\label{eq:6.26}
h_{\Lambda}=(k+ \frac{1}{2})( jk- \frac{1}{4} )+ \frac{1}{2}(j- \frac{1}{2}), 
\,\,s_{\Lambda}=(k+ \frac{1}{2})(k-j)\;\,\text{if}\,\;j \leq k,
\end{equation}
\begin{equation}
\label{eq:2.27}
h_{\Lambda}=(k+ \frac{1}{2} )( jk- \frac{1}{4} )+ \frac{1}{2}
(k- \frac{1}{2}),\,\, s_{\Lambda}=(k+ \frac{1}{2})(j-k)-1\,\;\text{if}\,\;j > k,
\end{equation}
\begin{equation}
\label{eq:6.28}
h^{\tw}_{\Lambda}=(k+ \frac{1}{2})( jk- \frac{1}{4})+ \frac{j}{2}- \frac{1}{16},\,\, s^{\tw}_{\Lambda}=(k+ \frac{1}{2})(k-j)- \frac{1}{2},\,\;
\text{if}\;\,j \leq k,
\end{equation}
\begin{equation}
\label{eq:6.29}
h^{\tw}_{\Lambda}=(k+ \frac{1}{2} )( jk- \frac{1}{4} )+ \frac{k}{2}- \frac{1}{16},\,\, s^{\tw}_{\Lambda}=(k+ \frac{1}{2})(j-k)- \frac{3}{2},\,\;\text{if}\;\,
j > k.
\end{equation}

As in Section \ref{sec:3}, introduce the following notation for the characters 
and supercharacters of these modules :
$$
\begin{array}{l@{\,}l}
ch^{N=3\,[M,m;\epsilon]}_{j,k;\frac{1}{2}} (\tau, z) &= ch^{\pm}_{H_{NS}(\Lambda_{j,k})} (\tau,z), \\[4ex]
ch^{N=3\,\left [M,m;\frac{1}{2}\right]}_{j,k;0}	(\tau,z) &= ch^+_{H_{R} (\Lambda_{j,k})} (\tau,z).
\end{array}
$$
Now the character formulae, given by Proposition \ref{prop:6.2}, can be rewritten in a unified way as follows ($\epsilon, \epsilon^\prime=0 \;\text{or}\; 
\frac{1}{2},\,\, \epsilon + \epsilon^\prime \neq 0$) :
\begin{equation}
\label{eq:6.30}
\left(\overset{3}{R}^{(\epsilon)}_{\epsilon^\prime} ch^{N=3[M,m;\epsilon]}_
{j,k;\epsilon^\prime}\right)(\tau,z) = \Psi^{[B;M, m;\epsilon]}_
{j,k;\epsilon\prime}(\tau,z,-z,0),
\,\,j,k\in \Omega^{[3;M]}_{\epsilon^\prime}.
\end{equation}    

Introduce the following modified $N=3$ characters $\overset{\sim}{ch}$ by the formula :
$$
\left(\overset{3}{R}^{(\epsilon)}_{\epsilon^\prime} \overset{\sim}{ch}^{N=3[M,m;\epsilon]}_{j,k;\epsilon^\prime}
\right) (\tau,z) = (-1)^{j-\epsilon^\prime}  
\overset{\sim}{\Psi}^{[B;M, m;\epsilon]}_{j,k;\epsilon^\prime}
(\tau,z,-z,0),\,\,j,k\in \Omega^{[3;M]}_{\epsilon^\prime}.
$$

Theorem \ref{th:5.7} and Proposition \ref{prop:6.1} imply the following 
modular transformation properties of these modified characters.

\begin{theorem}
\label{th:6.3}
Let $m \in \ZZ_{\geq 0}$, $M \in \ZZ_{\geq 2}$ be such that $\gcd (M,4m+2)=1$,
and let $\epsilon, \epsilon^\prime=0$ or $\frac{1}{2}$ be such that $\epsilon + \epsilon^\prime \neq 0$. Then for $j,k \in \Omega^{[3;M]}_{\epsilon^\prime}$ 
we have the following transformation formula of the modified characters of the $N=3$ modules with central charge $c$ given by (\ref{eq:6.1}) :
\begin{list}{}{}
\item (1) $\overset{\sim}{ch}^{N=3[M,m;\epsilon]}_{j,k;\epsilon^\prime} \left(-\frac{1}{\tau}, \frac{z}{\tau} \right)
=e^{\frac{\pi i c z^2}{3 \tau}} \sum\limits_{(a,b)\in \Omega^{[3;M]}_\epsilon}{S^{[\epsilon^\prime, \epsilon]}_{(j,k),(a,b)} \overset{\sim}{ch}^{N=3[M,m;\epsilon^\prime]}_{a,b;\epsilon}}(\tau,z)$,

where
$$
\begin{array}{ll}
S^{\left[\frac{1}{2},\frac{1}{2}\right]}_{(j,k),(a,b)} = \frac{2}{M}  e^{\frac{\pi i (2m+1)}{2M} (a-b)(j-k)} \cos \frac{2m+1}{2M} (a+b) (j+k) \pi,  \\[4ex]
S^{[\epsilon^\prime, \epsilon]}_{(j,k),(a,b)} = -\frac{2}{M} e^{\frac{\pi i (2m+1)}{2M} (a-b)(j-k)} \sin \frac{2m+1}{2M} (a+b) (j+k) \pi \;\,\,
\text{if}\;\,\,\epsilon^\prime \neq \epsilon.
\end{array}
$$

\item(2) $\overset{\sim}{ch}^{N=3[M,m;\epsilon]}_{j,k;\epsilon^\prime}
(\tau +1, z)
=e^{-\pi i (j+ \frac{1}{12} - \frac{\epsilon^\prime}{4}+ 4 \epsilon \epsilon^\prime)} e^{\frac{2\pi i (2m+1)}{2M} jk}
\overset{\sim}{ch}^{N=3[M,m;|\epsilon-\epsilon^\prime|]}_
{j,k;\epsilon^\prime}(\tau,z)$.
\end{list}
\end{theorem}
\hfill
$\Box$

\appendix
\numberwithin{equation}{section}
\section{Appendix. A brief review of theta functions}

In this appendix we review some basic facts about theta functions in a slightly more general setup than in \cite{K}, Chapter 13, or \cite{KW}.

Let $\fh$ be an $\ell$-dimensional vector space over $\CC$, endowed with a 
non-degenerate symmetric biliear form $(.|.)$; we shall identify $\fh$ with 
$\fh^*$ via this form. Let $k$ be a positive real number and let $ L $ be a 
lattice of rank $\ell$ (i.e. a free rank $\ell$ abelian subgroup) in $ \fh $, 
such that 
\[ k ( \alpha | \beta) \in \ZZ \quad \mbox{for all} 
\quad \alpha, \beta \in L , \] 
i.e. $ kL \subset L^* $, where $ L^* = \left\lbrace  \lambda \in \fh \: | 
\:(\lambda | L) \subset \ZZ \right\rbrace  $ is the dual lattice, and the 
restriction of the bilinear form $(.|.)$ to $L$ is positive-definite.

Let $ \widehat{\fh} = \fh \oplus \CC K \oplus \CC d$ be the $ \ell+2 $
-dimensional vector space over $ \CC $ with the (non-degenerate) symmetric 
bilinear form $ (.|.) $, extended from $ \fh $ by letting 
$ \fh \perp (\CC K + \CC d), (K|K) = 0 = (d|d), (K|d) = 1 $. 
We shall identify $  \widehat{\fh} $ and $ \widehat{\fh}^* $ using this form. 
Given $ \lambda \in \widehat{\fh} $, we denote by $ \bar{\lambda}$ its 
projection on $ \fh $.
Let $ X = \left\lbrace h \in \widehat{\fh} \: | \: \mathrm{Re} (K | h) > 0 \right\rbrace $.

Define the following representation of the additive group of the vector space 
$ \fh $ on the vector space $ \widehat{\fh} $ (cf. (\ref{eq:0.5})):
\[ t_{\alpha}(h) = h + (K|h)\alpha -((\alpha | h) + 
\frac{(\alpha | \alpha)}{2} 
(K| h) ) K, \quad \alpha \in \fh .\]
This action leaves the bilinear form
$ (.|.) $ on $  \widehat{\fh} $ invariant and fixes $ K $, hence leaves the 
domain $ X $ invariant. 

The additive group of $ \fh $ also acts on $ \widehat{\fh} $ by the affine 
transformations
\[ p_{\alpha} (h) = h + 2 \pi i \alpha ,\quad \alpha \in \fh,\]
which leave $ X $ invariant.

Denote by $ N^{'}_{\ZZ} $ the subgroup of affine transformations of the domain
$ X $, generated by the transformations $ t_{\alpha} $ and $ p_{\beta} $ 
for all $ \alpha, \beta \in L $. (This is a subgroup of the group $ N_{\ZZ} $, considered in \cite{K}, Chapter 13.)

A theta function of degree $ k $ is a holomorphic function $ F $ in the 
domain $ X $, satisfying the following two properties ($ h \in X $):

\begin{enumerate}
\item[(i)] $ F(n(h)) = F(h) \quad \mbox{for all} \quad n \in N^{'}_{\ZZ} $,
\item[(ii)] $ F ( h + aK) = e^{ka} F (h) \quad \mbox{for all} \quad a \in \CC$.
\end{enumerate}

Denote by  $\tilde{\mathrm{Th}}_k $ the space of all theta functions of degree
$k$.   Then 
$\tilde{\mathrm{Th}}_0 $ 
is the algebra of holomorphic functions in 
$ \tau \: (\mathrm{Im} \quad \tau > 0) $, \cite{K}, Lemma 3.2, and 
$ \tilde{\mathrm{Th}} := \oplus_{k \geq 0 \atop kL \subset L^*}  \; \tilde{\mathrm{Th}}_k$ is a $ \RR_{\geq 0} $-graded algebra over the subalgebra $\tilde{\mathrm{Th}}_0 $.  

Let $ D $ be the Laplace operator on $ \widehat{\fh} $, associated to the 
bilinear form $ (.|.) $, i.e. $ D e^h = (h|h) e^h, \quad h \in \widehat{\fh} $, and let $ \mathrm{Th} =  \oplus_{k \geq 0 \atop kL \subset L^*}  \; 
\mathrm{Th}_k $ denote the kernel of $ D $ in $\tilde{\mathrm{Th}}$ . 
This is an $ \RR_{\geq 0} $-graded algebra over $ \CC $. Elements of $ \mathrm{Th}_k $ are called classical theta functions (or Jacobi forms) of degree $ k $.

For $ k > 0 $, such that $ kL \subset L^* $, let
\[ P_k = \left\lbrace \lambda \in \widehat{\fh} \:| \: (\lambda | K) = k \quad \mathrm{and} \quad \bar{\lambda} \in L^* \right\rbrace. \]
Given $ \lambda \in P_k $, let
\[ \Theta_\lambda = e^{-\frac{(\lambda | \lambda)}{2k} K}  \sum_{\alpha \in L}^{} e^{t_{\alpha} (\lambda)} . \]
This series converges to a holomorphic function in the domain $ X $, which is 
an example of a theta function (=Jacobi form). Note that
\[ \Theta_{\lambda + k \alpha + aK} = \Theta_{\lambda} \quad \mathrm{for} \quad \alpha \in L,\: a \in \CC .\]

\begin{proposition}
\label{prop:A.1}
The set $ \left\lbrace \Theta_{\lambda} \: | \: \lambda \in P_k \mathrm{mod} (kL + \CC K) \right\rbrace $ is a $ \CC  $-basis of 
$ \mathrm{Th}_k $ (resp. $ \tilde{\mathrm{Th}}_0 $-basis of $\tilde
{\mathrm{Th}}_k $) 
if $ k > 0 $, and $ \mathrm{Th}_0  = \CC$.
\end{proposition}
\begin{proof}
It is the same as that of Proposition 13.3 and Lemma 13.2 in \cite{K} .
\end{proof}

Introduce coordinates $ (\tau, z, t) $ on $ \widehat{\fh} $ by (\ref{eq:0.1}),
so that $ X = \left\lbrace (\tau, z, t) \: | \: \mathrm{Im}\tau >0 \right\rbrace  $ and $ q := e^{2 \pi i \tau} = e^{-K} $. In these coordinates we have the 
usual formula for the Jacobi form $ \Theta_{\lambda} $, $ \lambda \in P_k $, 
of degree $ k > 0 $: 
\begin{equation}
\Theta_{\lambda} (\tau, z, t) = e^{2 \pi i k t } \sum_{\gamma \in L + \frac{\bar{\lambda}}{k}}^{} q^{\frac{k (\gamma | \gamma)}{2}} e^{2 \pi i k(\gamma|z)}.
\end{equation}

\begin{proposition}
\label{prop:A.2}
One has the following transformation formulae of a Jacobi form 
$ \Theta_{\lambda} $ of degree $ k \in R_{>0} $, such that $ kL \subset L^* $, where $ \lambda \in P_k $:
\begin{enumerate}
\item[(a)] $\Theta_{\lambda} (-\frac{1}{\tau}, \frac{z}{\tau}, t - \frac{(z | z)}{2 \tau}) =  
(-i \tau)^{\frac{\ell}{2}}  | L^* /kL|^{-\frac{1}{2}} \: 
\sum_{\mu \in P_k \: \mathrm{mod} (kL + \CC K)}^{} 
e^{-\frac{2 \pi i }{k} (\bar{\lambda}| \bar{\mu}) } 
\Theta_{\mu} (\tau, z, t)$.
\item[(b)] $\Theta_{\lambda} (\tau +1, z, t ) =   e^{2 \pi i k t} e^{\frac{\pi i (\bar{\lambda}|\bar{\lambda})}{k}} \sum_{\alpha \in L}^{} (-1)^{k (\alpha | \alpha)} e^{2\pi i(\bar{\lambda} + k \alpha | z)} 
q^{\frac{| \bar{\lambda} + k \alpha|^2 }{2k}}$.
\end{enumerate}
In particular, 
\[ \Theta_{\lambda} (\tau +1, z, t ) = e^{\frac{\pi i ( \bar{\lambda} | \bar{\lambda})}{k}} \quad \Theta_{\lambda} (\tau, z, t),\]
hence the $ \CC $-span of $\{ \Theta_{\lambda}  \}_ {\lambda \in P_k \: \mathrm{mod}(kL + \CC K)} $ is $ \mathrm{SL}_2(\ZZ) $-invariant
(up to the weight factor), provided that 
$ k(\alpha | \alpha) \in 2\ZZ$ for all $\alpha \in L$.
\end{proposition}
\begin{proof}
It is is the same as in \cite{KW}, using the same argument as in \cite{K}, 
Theorem 13.5, where in place of Lemma 13.4 we use that $ (S, j) $ normalizes 
the group $ N^{'}_{\ZZ}. $
\end{proof}

Let 
$ L_{\bar{0}}$ (resp. $L_{\bar{1}}$)
$ = \{ \alpha \in  L \: | \: k (\alpha | \alpha) \in 2 \ZZ$
(resp. $ \in 1 + 2 \ZZ ) \} $. 
Proposition \ref{prop:A.2} 
says that the $ \CC $-span of $\{ \Theta_{\lambda}  \}_
{\lambda \in P_k \: \mathrm{mod}(kL + \CC K)} $ is $ \mathrm{SL}_2(\ZZ) $-invariant, provided that $ L = L_{\bar{0}}. $

Now we construct an $ \mathrm{SL}_2(\ZZ) $-invariant family of classical theta functions in the case when $ L \neq L_{\bar{0}}. $  
In this case $ L_{\bar{1}} = \beta_0 + L_{\bar{0}} $ for any $ \beta_0 \in L $
, such that $ k (\beta_0 | \beta_0) \in 1 + 2 \ZZ $. In what follows we shall 
often write $ \Theta_{\lambda} = \Theta_{\lambda, L} $ in order to emphasize 
the dependence on $ L $. We let:
\[ \Theta^\pm_\lambda = \Theta_{\lambda, L_{\bar{0}}} \pm \Theta_{\lambda + k \beta_0, L_{\bar{0}}}. \]
Note that $ \Theta^+_\lambda = \Theta_{\lambda, L} $ and that 
$ \Theta^-_\lambda $ is an alternate analogue of a Jacobi form:
\[ \Theta^-_\lambda = e^{-\frac{(\lambda | \lambda)}{2k}K} \: \sum_{\alpha \in L}^{} (-1)^{k(\alpha | \alpha)}  e^{t_\alpha(\lambda)}.\]
Fix $ \gamma_0 \in L^*_{\bar{0}} \setminus L^* $, and let 
\[ \Theta_\lambda^{\pm, \gamma_0} = \Theta_{\lambda + \gamma_0,\: L_{\bar{0}}} \pm \Theta_{\lambda + \gamma_0 + k \beta_0, \: L_{\bar0}} \:. \]
Note that the $\CC$-span of the classical theta functions
$\Theta^\pm_{\lambda}$ and $\Theta^{\pm, \gamma_0}_{\lambda}$ 
of degree $k$ is the span of all classical theta functions of degree $k$
for the lattice $L_{\bar{0}}$.
 
It is easy to deduce the next proposition from Proposition \ref{prop:A.2}.
\begin{proposition}
\label{prop:A.3}
Let $ k \in \RR_{>0} $ be such that $ kL \subset L^* $ and assume that
$ L_{\bar{0}} \neq L $, and fix $ \gamma_0 \in L^*_{\bar{0}} \setminus L^* $. 
Let $ \lambda \in P_k. $ Then:
\begin{enumerate}
\item[(a)] $ \Theta^-_\lambda (-\frac{1}{\tau}, \frac{z}{\tau}     , t - \frac{(z | z)}{2\tau}    ) = $
\[ (-i \tau)^{\frac{\ell}{2}} | L^* / kL|
^{-\frac12} 
\sum_{\mu \in P_k \: \mathrm{mod} \: kL + \CC K}^{} e^{-\frac{2 \pi i}{k} (\bar{\lambda} | \bar{\mu} + \gamma_0)} 
\Theta^{-,\gamma_0}_\mu (\tau, z, t); \]
$ \Theta^{+, \gamma_0}_\lambda (-\frac{1}{\tau}, \frac{z}{\tau}     , t - \frac{(z | z)}{2\tau}    )=  $
\[ (-i \tau)^{\frac{\ell}{2}} | L^* / kL|^{-\frac12}  \sum_{\mu \in P_k \: \mathrm{mod} \: kL + \CC K}^{} e^{-\frac{2 \pi i}{k} (\bar{\lambda} + \gamma_0 | \bar{\mu})} \Theta^{-}_\mu (\tau, z, t); \]
$ \Theta^{-, \gamma_0}_\lambda (-\frac{1}{\tau}, \frac{z}{\tau}     , t - \frac{(z | z)}{2\tau}    )= $
\[ (-i \tau)^{\frac{\ell}{2}} | L^* / kL|^{-\frac12}  \sum_{\mu \in P_k \: \mathrm{mod} \: kL + \CC K}^{} e^{-\frac{2 \pi i}{k} (\bar{\lambda} + \gamma_0 | \bar{\mu} + \gamma_0)} \Theta^{-, \gamma_0} (\tau, z, t).  \]
\item[(b)] $ \Theta^\pm_\lambda (\tau + 1, z, t) = e^{\frac{\pi i}{k} 
(\bar{\lambda} |\bar{\lambda} )} \Theta^{\mp}_\lambda (\tau, z, t);  $
\bigskip
$ \Theta^{\pm, \gamma_0}_\lambda (\tau + 1, z, t) = e^{\frac{\pi i}{k} 
(\bar{\lambda} + \gamma_0 | \bar{\lambda} + \gamma_0) } \: \Theta^{\pm, \gamma_0}_\lambda (\tau, z, t). $
\end{enumerate}
Consequently, the $ \CC $-span of the classical theta functions 
$ \left\lbrace \Theta_{\lambda}^{+} , \Theta^{-}_{\lambda}, \Theta^{+, \gamma_0}_{\lambda}, \Theta^{-, \gamma_0}_{\lambda}\right\rbrace _
{\lambda \in P_k \: \mathrm{mod} \: kL + \CC K}$ 
is $ \mathrm{SL}_2(\ZZ) $-invariant.
\hfill $\Box$
\end{proposition}
\begin{proposition}
\label{prop:A.4}
One has the following elliptic transformation properties: 
$$ 
\Theta^\pm_\lambda (\tau , z+\beta, t) 
= \Theta^\pm_\lambda (\tau , z, t);\,\,
 \Theta^{\pm,\gamma_0}_\lambda (\tau , z+\beta, t) = (-1)^{k(\beta|\beta)} \Theta^{\pm,\gamma_0}_\lambda (\tau , z, t)\,\,\hbox{ if}\,\, \beta\in L . 
$$
$$ 
\Theta^\pm_\lambda (\tau , z+\tau \beta, t)=e^{-2\pi ik (\beta | z )}q^{-\frac{k}{2}(\beta|\beta)} \Theta^{\mp}_{\lambda+k\beta} (\tau, z, t);
$$
$$
\Theta^{\pm,\gamma_0}_\lambda (\tau , z+\tau \beta, t) = 
e^{-2\pi ik (\beta | z )}q^{-\frac{k}{2}(\beta|\beta)} \Theta^{\mp,\gamma_0}_
{\lambda+k\beta} (\tau, z, t) \,\,\hbox{if}\,\, \beta\in \frac{1}{k}L^*. 
$$
\hfill $\Box$
\end{proposition}

\end{document}